\documentclass[10pt]{article}

\PassOptionsToPackage{greek}{babel}
\usepackage{persobase}
\usepackage{persomath}
\usepackage{tikz-cd}

\title{First-order theory of \\ torsion-free Tarski monsters}

\author{
	R\'emi Coulon, 
	Francesco Fournier-Facio,
    and
	Meng-Che ``Turbo'' Ho
}

\setbftrue
\mcgfrenchtrue

\begin{document}
\selectlanguage{english}

\maketitle
\begin{abstract}
We develop methods to control the first-order theory of groups arising as certain direct limits of torsion-free hyperbolic groups, answering several questions in the literature. 
We construct simple torsion-free Tarski monsters $\Gamma$ (non-abelian groups whose non-trivial, proper subgroups are infinite cyclic) that are $\exists \forall \exists$-elementarily embedded into $\Gamma \ast \Z$. In particular, such $\Gamma$ have the same two-quantifier theory as $\Gamma*\Z$,
and hence the same positive theory as a non-abelian free group.
All previously known examples of groups with the same positive theory as the free group admit a non-elementary action on a hyperbolic space, while our examples cannot act on a hyperbolic space with a loxodromic element. 
Along the way, we solve the one-quantifier Knight conjecture for random quotients of arbitrary torsion-free, non-elementary, hyperbolic groups in the few-relator model.
\end{abstract}

\tableofcontents

\pagebreak

\section{Introduction}

\counterintro

The \emph{positive theory} of a group $\Gamma$ is the collection of first-order sentences without negations that hold in $\Gamma$. Some important algebraic properties of $\Gamma$ can be expressed by a positive sentence: satisfying a law, surjectivity of a word map, uniform perfection, and uniform simplicity, to name a few.

Building on the work of Merzlyakov \cite{positivefree}, Makanin proved that all non-abelian free groups have the same positive theory \cite{makanin}. Since positive sentences pass to quotients, the positive theory of every group contains the positive theory of a non-abelian free group. Because of this, we say that the positive theory of a group is \emph{trivial} if it coincides with that of a non-abelian free group. This property has been studied for many groups with hyperbolic features: hyperbolic groups \cite{positivehyp1, positivehyp2, andre:vf}, graph products that are not directly decomposable \cite{graphproducts}, many groups acting on trees \cite{positivetrees}, and acylindrically hyperbolic groups \cite{positiveAH}. We refer the reader to \cite{positivetrees} for a comprehensive overview.

Something that is apparent from the list above is that all of these groups admit a non-elementary action on a hyperbolic space, and indeed, it is this action that is used to prove the triviality of their positive theory. 
In particular, they contain non-abelian free subgroups.
We construct examples of a completely different nature.

Recall that, following the terminology of \cite{NL}, a group has \emph{property (NL)} if it cannot act on a hyperbolic space with a loxodromic element.

\begin{theo}
\label{intro:thm:NL}

    There exists a finitely generated, simple, non-amenable group that has trivial positive theory, has property (NL), and has no non-abelian free subgroups.
\end{theo}

\autoref{intro:thm:NL} answers two questions of Casals-Ruiz, Garreta and de la Nuez Gonz{\'a}lez \cite[Question 9.11 and 9.16]{positivetrees}, which asked about the existence of groups without non-abelian free subgroups and trivial positive theory.
The group we construct actually has a number of additional striking properties that are in contrast with previous examples of groups with trivial positive theory. 
Most importantly, it is a simple \emph{torsion-free Tarski monster}, that is, a non-abelian group in which every proper non-trivial subgroup is infinite cyclic. 
Such groups were initially constructed by Olshanskii \cite{noetherian}. This construction has been replicated and refined in many different contexts to achieve exotic properties, see, for example, \cite{lacunary, diversity, dimensions:simple}. 
In particular, these groups have no non-abelian free subgroups. 
(Note that torsion-free Tarski monsters do not always have trivial positive theory: it was pointed out to us by Osin that, combining the constructions of \cite{noetherian} and \cite{verballycomplete}, one can build torsion-free Tarski monsters that are \emph{verbally complete}, i.e., every word map is surjective.)

Non-amenability of our group comes from the fact that it can easily be arranged to have property (T).
The absence of free subgroups implies that our group has no general type action on a hyperbolic space. To prevent other types of actions and prove that it has property (NL), it is enough to ensure that it has no unbounded quasimorphisms \cite{manning}, or equivalently, that the stable commutator length vanishes \cite{bavard, calegari}. We show something much stronger.

\begin{theo}
\label{intro:thm:norms}

    There exists a finitely generated, simple, torsion-free Tarski Monster $\Gamma$ with trivial positive theory such that every non-trivial conjugacy-invariant norm on $\Gamma$ is stably bounded.

    In particular, the stable commutator length on $\Gamma$ vanishes identically, $\Gamma$ has no unbounded quasimorphisms, and the comparison map $H^2_b(\Gamma) \to H^2(\Gamma)$ from bounded to ordinary cohomology in degree $2$ is zero.
\end{theo}

All previously known examples of groups with trivial positive theory have infinite-dimensional kernel for the comparison map in degree $2$. \autoref{intro:thm:norms} partially answers \cite[Question 9.4]{positivetrees}, which speculated about a possible connection between having infinite-dimensional second bounded cohomology and having trivial positive theory. Examples of groups with non-trivial positive theory and infinite-dimensional second bounded cohomology, answering \cite[Question 9.5]{positivetrees}, are given in \cite{Qgroup}.

It was a famous open question whether there exist groups whose commutator length is unbounded but whose stable commutator length vanishes. For example, this appears in \cite[Question 10.4.1]{mimura}. It is also implicit in a question that Monod attributed to Ab{\'e}rt in his ICM proceedings paper \cite[Question E]{monod:invitation}, which was about a specific candidate. The first examples were provided by Muranov \cite{muranov}, and further examples were then given by Mimura \cite{mimura}.

Commutator length (\resp, stable commutator length) is a special case of \emph{$w$-length} (\resp, \emph{stable $w$-length}) where $w \in F_2$ is the commutator word \cite{segal}. This can be naturally generalized to arbitrary words in a free group of any rank.
There are words for which these notions are not interesting: $w$ could be freely equivalent to $1$, or it could define a surjective word map in every group; such words are called \emph{silly} by Segal \cite[Section 3.1]{segal}.

\begin{coro}
\label{intro:cor:words}
    There exists a finitely generated, simple group $\Gamma$ such that for every non-silly word $w$, the $w$-length is unbounded and the stable $w$-length vanishes.
\end{coro}

Note that when $w \in F_n \setminus [F_n, F_n]$, the stable $w$-length vanishes on every group \cite[Lemma 2.13]{stablewlength}. However, when $w \in [F_n, F_n] \setminus \{1\}$, the stable $w$-length does not vanish on any acylindrically hyperbolic group \cite{verbalwidth}.

\medskip

There is another aspect in which stronger results hold for acylindrically hyperbolic groups. Andr{\'e} and Fruchter proved that if $\Gamma$ is an acylindrically hyperbolic group with no non-trivial finite normal subgroup, then the two-quantifier theory of $\Gamma$ coincides with the two-quantifier theory of the free product $\Gamma \ast \Z$ \cite{positiveAH}, namely, 
\begin{equation*}
	\Th_{\forall \exists}(\Gamma) = \Th_{\forall \exists}(\Gamma \ast \Z).
\end{equation*}
Thanks to quantifier reduction for positive sentences \cite[Theorem F]{positivetrees}, it is easy to see that this property implies the triviality of the positive theory.

In fact, they prove something slightly stronger. Let $\vec{\gamma}$ be a tuple of elements of $\Gamma$, and denote by $\Th(\Gamma, \vec{\gamma})$ the collection of first-order formulas $\Sigma(\vec{z})$ such that $\vec{z}$ has the same arity as $\vec{\gamma}$ and $\Sigma(\vec{\gamma})$ is true in $\Gamma$. As above, we decorate this notation with the sequence of quantifiers that we want to restrict to. Then for all tuples $\vec{\gamma}$ of elements of $\Gamma$ (acylindrically hyperbolic and with no non-trivial finite normal subgroups) we have:
\begin{equation*}
	\Th_{\exists \forall \exists}(\Gamma, \vec{\gamma}) \subset \Th_{\exists \forall \exists}(\Gamma \ast \Z, \vec{\gamma}),
\end{equation*}
that is, $\Gamma$ is \emph{$\exists \forall \exists$-elementarily embedded} in $\Gamma \ast \Z$.

This property is strictly stronger than the triviality of the positive theory: for example, if $\Gamma$ is a non-solvable Baumslag--Solitar group, then it has trivial positive theory \cite{positivetrees}, yet it is not $\exists \forall \exists$-elementarily embedded into its free product with $\Z$ \cite[Remark 10.7]{positiveAH}. We show that our group enjoys even this stronger property.

\begin{theo}
\label{intro:thm:AH}

    There exists a finitely generated, non-acylindrically hyperbolic group $\Gamma$ that is $\exists \forall \exists$-elementarily embedded into $\Gamma \ast \Z$.
\end{theo}

This answers a question of Andr{\'e} and Fruchter \cite[Question 10.2]{positiveAH}. It also shows that the two-quantifier theory does not detect acylindrical hyperbolicity of finitely generated groups, partially answering a question attributed to Osin: see \cite[Question 1.1]{andre:AH} and \cite[Question 10.3]{positiveAH}. In contrast, the two-quantifier theory does detect hyperbolicity of finitely generated groups, see \cite{positivehyp1, andre:hyp-elementary}.

\medskip

As in previous constructions of torsion-free Tarski monsters, our groups arise as limits of sequences of \emph{small cancellation} quotients $\Gamma_0 \twoheadrightarrow \Gamma_1 \twoheadrightarrow \Gamma_2 \twoheadrightarrow \cdots$ of torsion-free, non-elementary, hyperbolic groups. 
By choosing the relations carefully, we are able to control the first-order theory of the limit in terms of the first-order theory of the groups in the sequence. This is our main result, which covers all of the theorems above.

\begin{theo}[\autoref{res: tarski}, \autoref{res: tarski bounded}]
\label{intro:thm:precise}
    Let $\Gamma$ be a non-elementary, torsion-free, hyperbolic group.
	Then there is a sequence of non-elementary, torsion-free, hyperbolic groups
	\begin{equation*}
		\Gamma = \Gamma_0 \onto \Gamma_1 \onto \Gamma_2 \onto \cdots \onto \Gamma_k \onto \Gamma_{k+1} \onto \cdots
	\end{equation*}
	whose direct limit $\Gamma_\infty$ has the following properties. Let $\pi_k \colon \Gamma \to \Gamma_k : k \in \N \cup \{\infty\}$ denote the projection.
	\begin{enumerate}
		\item \label{enu: tarski - tarski}
		The group $\Gamma_\infty$ is a lacunary hyperbolic, simple, torsion-free Tarski monster. 
		In particular, it has no non-abelian, free subgroups.
        \item \label{intro main enu bounded} Every non-trivial conjugacy-invariant norm on $\Gamma_\infty$ is stably bounded.
		\item \label{intro main enu limit theories}
		For every tuple $\vec{\gamma}$ of elements of $\Gamma$, we have
		\begin{equation*}
			{\Th}_{\forall \exists}(\Gamma_\infty, \pi_\infty(\vec{\gamma})) 
			= \lim\limits_{k \to \infty} {\Th}_{\forall \exists}(\Gamma_k, \pi_k(\vec{\gamma})).
		\end{equation*}
		\item \label{intro main enu embedding}
		The group $\Gamma_\infty$ $\exists\forall\exists$-elementarily embeds in $\Gamma_\infty \ast \Z$.
	\end{enumerate}
\end{theo}

We recall that a group is lacunary hyperbolic if one of its asymptotic cones is an $\R$-tree; this is a property that occurs naturally among certain direct limits of hyperbolic groups \cite{lacunary}.
Item \ref{intro main enu limit theories} implicitly says that the limit exists, namely, all $\forall \exists$-formulas that hold in infinitely many $\Gamma_k$ actually hold in all but finitely many $\Gamma_k$. As we mentioned above, Item \ref{intro main enu embedding} implies that $\Gamma_\infty$ has trivial positive theory (see the proof of \autoref{res: tarski positive}).
So, taking $\Gamma$ to have property (T), the group $\Gamma_\infty$ satisfies the properties from Theorems~\ref{intro:thm:NL},
\ref{intro:thm:norms} and~\ref{intro:thm:AH} at once.

\medskip

Let us illustrate the full strength of \autoref{intro:thm:precise} with two more applications. First, recall that a mixed identity on a group $\Gamma$ is a word $w \in \Gamma \ast F_n$ such that $w(\vec{\gamma}) = 1$ for all $\vec\gamma \in \Gamma^n$ \cite{mixedidentities}. A group is \emph{mixed identity free (MIF)} if the only mixed identity is the trivial one. The group $\Gamma_\infty$ from \autoref{intro:thm:precise} is easily seen to be MIF, because elementary embeddings allow us to deal also with formulas with constants. This property has recently found striking applications in operator algebras \cite{MIFOA}, and it has revealed itself to be especially important for finitely presented simple groups, in the context of the Boone--Higman conjecture \cite{bh:autfn}.

A sufficient condition for a finitely generated simple group to be MIF is that it admits a faithful highly transitive action \cite[Theorem 5.9]{hull:osin}. However, there exist finitely generated simple groups that are MIF despite not admitting a highly transitive action \cite[Theorem 4.13]{confined}. We provide a sharper example (see \cite[Proposition 1.8]{hull:osin} for a related example).

\begin{coro}[\autoref{res: tarski MIF}]
\label{intro:cor:MIF}
    The group $\Gamma_\infty$ from \autoref{intro:thm:precise} is MIF, but does not admit a $2$-transitive action.
\end{coro}

As another application, we obtain for $\Gamma_\infty$ a striking property known for the free group \cite{KM:genus}, which is closely related to commutator and stable commutator length.

\begin{coro}[\autoref{res: torsion in ab of product}]
\label{intro:cor:ab}
    For the group $\Gamma_\infty$ from \autoref{intro:thm:precise}, the direct power $\Gamma_\infty^{\N}$ has $2$-torsion in the abelianization.
\end{coro}

As we mentioned above, the group $\Gamma_\infty$, which is the main focus of our applications, is obtained as the direct limit of a sequence $(\Gamma_k)$ of hyperbolic groups.
More precisely, this sequence is built recursively, where each $\Gamma_{k+1}$ is a \emph{small cancellation} quotient of $\Gamma_k$.
Hence, most of our work consists of understanding the solutions of a system of equations with coefficients in the small cancellation quotients $\bar \Gamma$ of an arbitrary torsion-free non-elementary hyperbolic group $\Gamma$.
Quadratic equations in small cancellation quotients of the free group have been investigated, for instance, by Schupp \cite{Schupp:1980aa} and Lysenok \cite{Lysenok:1988tg}.
The next statement can be seen as a vast generalization of Schupp's work.

\begin{theo}[\autoref{res: lifting quotient}]
\label{intro:thm:lifting}
    Let $\Gamma$ be a torsion-free non-elementary hyperbolic group. Let $G$ be a finitely generated group and $H$ an arbitrary group with morphisms $\jmath \colon H \to G$ and $\iota \colon H \to \Gamma$.
	There are parameters $\lambda, \epsilon \in (0,1)$, with the following properties.
    
	Let $\bar \Gamma$ be a tight $C'(\lambda, \epsilon)$ strengthened small cancellation quotient of $\Gamma$ and $\pi \colon \Gamma \onto \bar \Gamma$ the corresponding projection.
	Let $\bar \varphi \colon G \to \bar \Gamma$ be a morphism such that $\bar \varphi \circ \jmath = \pi \circ \iota$.
	Then there is a morphism $\varphi \colon G \to \Gamma$ lifting $\bar \varphi$ such that $\varphi \circ \jmath = \iota$.
\end{theo}

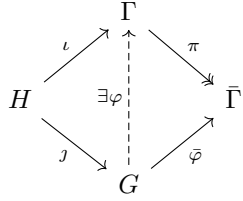
\[\begin{tikzcd}
	& \Gamma \\
	H && {\bar{\Gamma}} \\
	& G
	\arrow["\pi", two heads, from=1-2, to=2-3]
	\arrow["\iota", from=2-1, to=1-2]
	\arrow["\jmath"', from=2-1, to=3-2]
	\arrow["\exists\varphi", dashed, from=3-2, to=1-2]
	\arrow["{\bar{\varphi}}"', from=3-2, to=2-3]
\end{tikzcd}\]

\begin{rema}
\label{rem: discussion torsion}
	Note that we work here with a small cancellation condition that is stronger than the usual one.
	The ``strengthening'' requires that the relations are very long compared to other relevant metric invariants of $\Gamma$.
	Among other things, it allows us to make sure that the quotient $\bar \Gamma$ is non-elementary, unlike some silly small cancellation groups such as $\bar \Gamma = \group{a,b | a,b}$.
	The ``tightness'' (see \autoref{defi tight}) essentially asks that we add only finitely many relations, none of which is not a proper power. 
	This is a crucial assumption.
	Indeed, consider the groups 
	\begin{equation*}
		G = \group{x,y,z \mid x^2y^2z^2} 
		\quad \text{and} \quad 
		\bar \Gamma = \group {a,b \mid (a^2b^2)^{2n+1}}.
	\end{equation*}
	Note that $\bar \Gamma$ is a small cancellation quotient of the free group $F_2$, provided $n$ is sufficiently large.
	On the one hand, it is well-known that any morphism $G \to F_2$ has cyclic image \cite{Lyndon:1959vy}.
	On the other hand, the map 
	\begin{equation*}
		x \mapsto a, \quad y \mapsto b, \quad z \mapsto (a^2b^2)^n
	\end{equation*}
	induces a well-defined epimorphism $G \onto \bar \Gamma$.
	Hence, in contrast with \autoref{intro:thm:lifting}, not every morphism $G \to \bar \Gamma$ comes from a morphism $G \to F_2$.
	Actually, we already encountered here all the difficulties present in the work of Coulon and Sela about equations in Burnside groups \cite{Coulon:2021wg}.

    In the literature, constructions of torsion-free Tarski monsters typically go hand-in-hand with constructions of infinite torsion groups, which are obtained similarly by adding small cancellation relations that are large proper powers. However, torsion groups often have non-trivial positive theory; moreover, if $\Gamma$ is a torsion group, then $\Gamma$ and $\Gamma \ast \Z$ are distinguished by their $\forall \exists$ theory: see \autoref{rema:pthroot}.
\end{rema}

\autoref{intro:thm:lifting} also has another application that does not involve direct limits. It was conjectured by Knight that a first-order sentence holds in the free group if and only if it holds in a few-relator random quotient of the free group with overwhelming probability \cite[Conjecture 1]{kharlampovich:sklinos}. This conjecture was addressed for the more general density model by Kharlampovich--Sklinos \cite{kharlampovich:sklinos, kharlampovich:sklinos2, kharlampovich:sklinos3} and Massalha \cite{massalha, massalha2}. In both cases, one-quantifier sentences are treated separately \cite{kharlampovich:sklinos, massalha}. We are able to treat these (and more generally, one-quantifier \emph{formulas}) for random quotients of arbitrary torsion-free hyperbolic groups.

\begin{theo}[\autoref{res: knight}]
\label{intro:thm:knight}
    Let $\Gamma$ be a torsion-free, non-elementary, hyperbolic group, and let $\vec{\gamma}$ be a tuple of elements of $\Gamma$. Let $k \geq 1$, and let $\Sigma(\vec{z})$ a one-quantifier formula, where $\vec{z}$ has the same arity as $\vec{\gamma}$. Then $\Sigma(\vec{\gamma})$ holds in $\Gamma$ if and only if $\Sigma(\pi(\vec{\gamma}))$ holds in a $k$-relator random quotient $\pi \colon \Gamma \to \bar \Gamma$ with overwhelming probability.
\end{theo}

The only thing we are using about random quotients of $\Gamma$ is that they satisfy an arbitrarily strong small cancellation condition (\autoref{res: theory quotients onequant}). 
Therefore, our approach is suitable for several variations of the few-relator model, but does not apply to the density model, unlike in \cite{kharlampovich:sklinos, massalha}. 
On the other hand, it is crucial for our applications to torsion-free Tarski monsters that our control does not come from \emph{random} relations, but rather from \emph{well-chosen} small cancellation relations.

\medskip

We expect that \autoref{intro:thm:lifting} will find applications beyond the ones presented in this paper, be it to construct other exotic groups with trivial positive theory, or in a more in-depth study of the first-order theory of small cancellation quotients of torsion-free hyperbolic groups.
Moreover, it is likely that many of the techniques in this paper can be extended to more general groups acting on hyperbolic spaces. 
However, as discussed in \autoref{rem: discussion torsion}, generalizing to groups with torsion seems more challenging.

\paragraph*{Acknowledgements.} The authors thank Montserrat Casals-Ruiz, Jonathan Fruchter, Daniel Groves, Vincent Guirardel, Ilya Kazachkov, Nicolas Monod, Denis Osin, Chlo{\'e} Perin, Alessandro Sisto, Henry Wilton, and Matt Zaremsky for useful discussions. They are indebted to Simon Andr{\'e}, who suggested strengthening the results from the first version to prove the statements on $\exists \forall \exists$-elementary embeddings.
Rémi Coulon acknowledges support from the Agence Nationale de la Recherche under the Grant \emph{GoFR} ANR-22-CE40-0004 as well as the Région Bourgogne-Franche-Comté under the grant \emph{ANER 2024 GGD}.
His institute receives support from the EIPHI Graduate School (contract ANR-17-EURE-0002). Francesco Fournier-Facio is supported by the Herchel Smith Postdoctoral Fellowship Fund. Meng-Che Ho is partially supported by the National Science Foundation under Grant No.~DMS-2054558.

\paragraph*{Outline.} We start by reviewing the framework of geometric small cancellation theory in \autoref{sec: sc}. The section culminates in \autoref{res: recap sc}, which lists properties of small cancellation quotients $\Gamma \to \bar \Gamma$. Of particular importance is a version of \autoref{intro:thm:lifting} for morphisms $G \to \bar \Gamma$ with ``small energy''.

The proof of \autoref{intro:thm:lifting} for morphisms with ``big energy'' proceeds by contradiction. If the theorem fails, then one can produce a sequence of quotients $\Gamma \onto \Gamma_k$ with small cancellation constants going to zero, and morphisms $G \to \Gamma_k$ that do not lift. This sequence of morphisms can be used to build an action of $G$ on an $\R$-tree, and then the techniques of Rips and Sela (in particular the ``shortening argument'') apply to reach the desired contradiction. This argument takes up the next three sections: \autoref{sec: graph of groups} contains background on graphs of actions and the Rips--Sela machine for groups acting on $\R$-trees; \autoref{sec: limit} introduces limit groups, and proves a suitable version of the shortening argument that is needed for the proof; finally, \autoref{sec: lifts} uses these results to conclude the proof of \autoref{intro:thm:lifting}.

In \autoref{sec: random} we prove \autoref{intro:thm:knight}. We first review the translation principle between first-order sentences and morphisms, which also serves as a warm-up for the next section. Once this has been established, the result becomes a rather direct application of \autoref{intro:thm:lifting}.

Finally, \autoref{sec: tarski} contains our main applications: \autoref{intro:thm:precise} and its corollaries. We start by explaining how to find small cancellation relations with arbitrarily good constants that allow for the construction of torsion-free Tarski monsters. The groups are then constructed by an iterated sequence of quotients by these relations. Unlike in the previous section, here the relations are not random, and have to be chosen carefully to ensure that all the desiderata from \autoref{intro:thm:precise} are fulfilled. The vanishing of stable conjugacy invariant norms from item \ref{intro main enu bounded} is treated separately in \autoref{subsec: boundedness}.

\paragraph*{Conventions.} We use $\R_+$ to denote the set of non-negative real numbers, and $\R_+^* = \R_+ \setminus \{0\}$. Graphs, and in particular Cayley graphs, are always endowed with a path distance where all edges have the same positive length, although we allow the length of an edge to be different than $1$: this will be important when rescaling a metric (see \autoref{res: recap sc}).

\countermain

%%%%%%%%%%%%%%%%%%%%%%%%%%%%%%%%%%%%%%%%%%%%%%%%%%%%%%%%%%%%%%%%%%%%%%%%%%%%%%%%%%%%%%%%%%%%%%
%%%%%%%%%%%%%%%%%%%%%%%%%%%%%%%%%%%%%%%%%%%%%%%%%%%%%%%%%%%%%%%%%%%%%%%%%%%%%%%%%%%%%%%%%%%%%%
%
\section{Geometric small cancellation theory}
\label{sec: sc}

%
%%%%%%%%%%%%%%%%%%%%%%%%%%%%%%%%%%%%%%%%%%%%%%%%%%%%%%%%%%%%%%%%%%%%%%%%%%%%%%%%%%%%%%%%%%%%%%
%%%%%%%%%%%%%%%%%%%%%%%%%%%%%%%%%%%%%%%%%%%%%%%%%%%%%%%%%%%%%%%%%%%%%%%%%%%%%%%%%%%%%%%%%%%%%%

In this section, we review the framework of small cancellation theory and prove several results about a small cancellation quotient $\Gamma \to \bar \Gamma$, where $\Gamma$ is a torsion-free hyperbolic group. These results are collected in \autoref{res: recap sc}. One of these results is a version of \autoref{intro:thm:lifting} for morphisms $G \to \bar \Gamma$ with ``small energy''; the case of ``big energy'' will be the focus of the next three sections.

%%%%%%%%%%%%%%%%%%%%%%%%%%%%%%%%%%%%%%%%%%%%%%%%%%%%%%%%%%%%%%%%%%%%%%%%%%%%%%%%%%%%%
%
\subsection{Hyperbolic geometry}
%
%%%%%%%%%%%%%%%%%%%%%%%%%%%%%%%%%%%%%%%%%%%%%%%%%%%%%%%%%%%%%%%%%%%%%%%%%%%%%%%%%%%%%
\label{sec: hyperbolic geometry}

We review here some general facts about hyperbolic geometry in the sense of Gromov.
For more details, we refer the reader to Gromov's original article \cite{Gromov:1987tk}, or \cite{Coornaert:1990tj}, \cite{Ghys:1990ki}, and \cite{Bridson:1999ky}.

%-----------------------------------------------------------------
\paragraph{Four point inequality.}
%----------------------------------------------------------------

Let $X$ be a length space.
For every $x,y \in X$, we write $\dist[X] xy$, or simply $\dist xy$, for the distance between $x$ and $y$.
If one exists, we write $\geo xy$ for a geodesic joining $x$ to $y$.
The \emph{Gromov product} of three points $x,y,z \in X$ is
\begin{equation*}
	\gro xyz = \frac 12 \left[ \dist xz + \dist yz - \dist xy\right].
\end{equation*}
Let $\delta \in \R_+$.
For the rest of this section, we always assume that $X$ is \emph{$\delta$-hyperbolic}, i.e., for every $x,y,z,t \in X$,
\begin{equation}
\label{eqn: four point hyp}
	\min\left\{ \gro xyt, \gro yzt \right\} \leq \gro xzt + \delta.
\end{equation}
We denote by $\partial X$ the boundary at infinity of $X$.
The map $(x,y,z) \mapsto \gro xyz$ naturally extends to a map $(X \cup \partial X) \times (X \cup \partial X) \times X \to \R_+ \cup\{\infty\}$.

%-----------------------------------------------------------------
\paragraph{Quasi-convex subsets.}
%-----------------------------------------------------------------
Let $\alpha \in \R_+$.
A subset $Y$ of $X$ is \emph{$\alpha$-quasi-convex} if $d(x,Y) \leq \gro y{y'}x + \alpha$, for every $x \in X$, and $y,y' \in Y$.
We denote by $\distV[Y]$ the length metric on $Y$ induced by the restriction of $\distV[X]$ to $Y$: note that this is \emph{not} the restriction of $\distV[X]$ to $Y \times Y$.
We say that $Y$ is \emph{strongly quasi-convex} if it is $2 \delta$-quasi-convex and for every $y,y' \in Y$,
\begin{equation}
\label{eqn: def strongly qc}
	\dist[X]y{y'} \leq \dist[Y]y{y'} \leq \dist[X]y{y'} + 8\delta.
\end{equation}

%-----------------------------------------------------------------
\paragraph{Isometries.}
%-----------------------------------------------------------------
An isometry $\gamma$ of $X$ is either \emph{elliptic} (its orbits are bounded), \emph{parabolic} (its orbits admit exactly one accumulation point in $\partial X$), or \emph{loxodromic} (its orbits admit exactly two accumulation points in $\partial X$).
In order to measure the action of an isometry $\gamma$ on $X$, we use the \emph{translation length} and the \emph{stable translation length} defined respectively by
\begin{equation*}
	\norm[X] \gamma = \inf_{x \in X}\dist {\gamma x}x
	\quad \text{and} \quad
	\snorm[X] \gamma = \lim_{n \to \infty} \frac 1n \dist{\gamma^nx}x.
\end{equation*}
Note that the definition of $\snorm[X] \gamma$ is independent of the choice of $x$. If there is no ambiguity, we will omit the space $X$ from the notation.
These lengths are related as follows \cite[Chapitre~10, Proposition~6.4]{Coornaert:1990tj}:
\begin{equation}
\label{eqn: regular vs stable length}
	\snorm \gamma\leq \norm \gamma \leq \snorm \gamma+ 8\delta.
\end{equation}
In addition, $\gamma$ is loxodromic if and only if $\snorm \gamma > 0$.
In this case, the accumulation points of $\gamma$ in $\partial X$ are
\begin{equation*}
	\gamma^- = \lim_{n \to \infty} \gamma^{-n}x
	\quad \text{and} \quad
	\gamma^+ = \lim_{n \to \infty} \gamma^nx.
\end{equation*}
They are the only points of $X\cup\partial X$ fixed by $\gamma$.
We write $Q_\gamma$ for the union of all $L$-local $(1, \delta)$-quasi-geodesic joining $\gamma^-$ to $\gamma^+$ with $L > 12\delta$.
The \emph{cylinder} of $\gamma$ is the set
\begin{equation*}
	Y_\gamma = \set{x \in X}{ d(x, Q_\gamma) \leq 20\delta}.
\end{equation*}
It is a $2\delta$-quasi-convex subset of $X$ which can be thought of as the ``smallest'' $\gamma$-invariant quasi-convex subset.
If in addition $X$ is proper and geodesic, then $Y$ is a closed strongly quasi-convex subset \cite[Lemmas~2.31 and 2.32]{Coulon:2014fr}.

\begin{defi}
\label{def: fix}
	Let $U$ be a finite set of isometries of $X$.
	The $\ell^\infty$ and $\ell^1$-\emph{energy} of $U$ are defined respectively as
	\begin{equation*}
		E_\infty(U, X) = \inf_{x \in X} \sup_{\gamma \in U} \dist {\gamma x}x
		\quad \text{and}\quad
		E_1(U, X) = \inf_{x \in X} \sum_{\gamma \in U} \dist {\gamma x}x.
	\end{equation*}
\end{defi}

Although the second definition often does not make sense if $U$ is infinite, the first one works for an infinite $U$ as well.
By convention, the energy of the empty set is zero.
If there is no ambiguity, we simply write $E_\infty(U)$ and $E_1(U)$.
If $\varphi \colon G \to \isom X$ is a morphism and $U$ is a finite subset of $G$, its energies are defined in the same way, that is
\begin{equation*}
	E_\infty(\varphi, U, X) = E_\infty(\varphi(U), X)
	\quad \text{and}\quad
	E_1(\varphi, U, X) = E_1(\varphi(U), X).
\end{equation*}
As above, if there is no ambiguity, we simply denote them by $E_\infty(\varphi, U)$ and $E_1(\varphi, U)$.
Note that the $\ell^\infty$ and $\ell^1$-energy are bi-Lipschitz equivalent as functions of $\varphi$.
More precisely,
\begin{equation*}
	E_\infty(\varphi, U) \leq E_1(\varphi, U) \leq \card U E_\infty(\varphi, U).
\end{equation*}

\begin{rema}
\label{rem: energy invariant by inner automorphisms}
	Note also that if $\alpha$ is an inner automorphism of $G$, then 
	\begin{equation*}
		E_\infty(\varphi\circ \alpha, U) =  E_\infty(\varphi, U)
		\quad \text{and} \quad 
		E_1(\varphi\circ \alpha, U) = E_1(\varphi, U). \qedhere
	\end{equation*}
\end{rema}

\begin{defi}
	Let $U$ be a set of isometries of $X$.
	Given $r \in \R_+$, we define
	\begin{equation}
	\label{eqn: def fix}
		\fix{U,r} = \set{x \in X}{\forall \gamma \in U,\ \dist{\gamma x}x \leq r}.
	\end{equation}
	If $U$ is empty, we adopt the convention that $\fix{U,r} = X$.
\end{defi}

For simplicity, we write $\fix{U}$ for $\fix{U,0}$.
If the set $U = \{\gamma\}$ is reduced to a single isometry, then $E_\infty(U) = \norm \gamma$ and we simply write $\fix{\gamma,r}$ for $\fix{U,r}$.

\begin{lemm}[{\cite[Lemma~2.8]{Coulon:2018vp}}]
\label{res: set of almost fixed points}
	Let $U$ be a set of isometries of $X$.
	If $r > \max\{E_\infty(U), 5\delta\}$, then $\fix{U, r}$ is $8\delta$-quasi-convex.
	Moreover, for every $x \in X \setminus \fix{U,r}$, we have
	\begin{equation}
	\label{eqn: displacement outside fixed set}
		\sup_{\gamma \in U} \dist{\gamma x}x \geq 2d\left(x, \fix{U,r}\right) + r - 10\delta.
	\end{equation}
\end{lemm}

\begin{defi}
\label{def: thin isom}
	Let $\alpha \in \R_+$.
	\begin{itemize}
		\item An elliptic isometry $\gamma$ of $X$ is \emph{$\alpha$-thin at $x \in X$}, if for every $r \in \R_+$ the set $\fix{\gamma, r}$ is contained in $B(x, r/2 + \alpha)$.
		It is \emph{$\alpha$-thin} if it is $\alpha$-thin at some point of $X$.
		\item A loxodromic isometry $\gamma$ of $X$ is \emph{$\alpha$-thin} if for every $r \in \R_+$ and for every $y \in \fix{\gamma,r}$ we have
		\begin{equation*}
			\gro{\gamma^-}{\gamma^+}y \leq \frac 12 (r - \norm \gamma) + \alpha.
		\end{equation*}
	\end{itemize}
	By convention, parabolic isometries are never thin.
\end{defi}

\begin{rema}
\label{rem: thin isom - loxo}
	It follows from the stability of quasi-geodesics that any loxodromic isometry $\gamma$ is $\alpha$-thin, for some $\alpha$ satisfying 
	\begin{equation*}
		\alpha \leq 100\left(\frac{\delta}{\snorm \gamma} + 1\right)\delta.
	\end{equation*}
	In our context, the thinness provides a way to control the action of an isometry $\gamma$ when we do not have any lower bound for its stable length $\snorm \gamma$.
\end{rema}

%-----------------------------------------------------------------
\paragraph{Group actions.}
%-----------------------------------------------------------------

Let $\Gamma$ be a group acting by isometries on a $\delta$-hyperbolic length space $X$.
We denote by $\Lambda(\Gamma)$ its \emph{limit set}, i.e., the set of accumulation points in $\partial X$ of the $\Gamma$-orbit of a point $x \in X$.
We say that $\Gamma$ is \emph{non-elementary} (\emph{for its action on $X$}) if $\Lambda(\Gamma)$ contains at least three points.
If $\Lambda(\Gamma)$ is empty (\resp, contains exactly one point, or exactly two points), then $\Gamma$ is called \emph{elliptic} (\resp, \emph{parabolic}, or \emph{loxodromic}).

The \emph{elementary closure} of a loxodromic element $\gamma \in \Gamma$ is the largest elementary subgroup of $\Gamma$ containing $\gamma$.

\begin{lemm}
\label{res: parabolic large fixed point set}
	Let $\xi \in \partial X$.
	Let $U$ be a finite subset of $\Gamma$, fixing $\xi$, which consists of non-loxodromic elements.
	Then the set $\fix{U, 40\delta}$ is not bounded.
\end{lemm}

\begin{proof}
	Fix a base point $x \in X$.
	Since $X$ is a length space, for every $L \in \R_+$ there is an $L$-local $(1, 10\delta)$-quasi-geodesic ray joining $x$ to $\xi$.
	The statement now follows from \cite[Lemma~2.11]{Coulon:2018vp}.
\end{proof}

We associate with our group two numerical invariants that will be useful to control the structure of limit groups (see \autoref{sec: actions on R-trees}).

\begin{defi}
\label{def: inj radius}
	The \emph{injectivity radius} for the action of $\Gamma$ on $X$, denoted by $\inj \Gamma X$, is the quantity
	\begin{equation*}
		\inj \Gamma X = \inf \snorm \gamma
	\end{equation*}
	where the infimum runs over all loxodromic isometries $\gamma \in \Gamma$.
\end{defi}

The injectivity radius is positive whenever the action of $\Gamma$ is non-elementary, proper, and co-compact.

\begin{defi}
\label{def: acyl inv}
	Let $r \in \R_+$.
	The \emph{(local) acylindricity parameter at scale $r$}, denoted by $A(\Gamma,X,r)$ is defined as 
	\begin{equation*}
		 A(\Gamma,X,r) = \sup_{U \subset \Gamma} \diam {\fix{U,r}},
	\end{equation*}
	where $U$ runs over all subsets of $\Gamma$ generating a non-elementary subgroup.
\end{defi}

\begin{prop}
\label{res: margulis lemma}
	Suppose that $\Gamma$ is torsion-free and acts properly on $X$.
	For every $r \in \R_+$, we have
	\begin{equation*}
		A(\Gamma,X, r) \leq 4r + A(\Gamma,X,400\delta) + 24\delta.
	\end{equation*}
\end{prop}

\begin{proof}
	This is a particular case of \cite[Proposition~3.5]{Coulon:2018vp}.
	Indeed, \cite{Coulon:2018vp}, which handles any group $\Gamma$ acting on a hyperbolic space, involves a third invariant $\nu(\Gamma, X)$.
	However, for a torsion-free group acting properly on a hyperbolic space, we have $\nu(\Gamma, X) = 1$.
\end{proof}

This motivates the next definition.

\begin{defi}
\label{def: acylindricity}
	Assume that $\Gamma$ is torsion-free and acts properly on $X$.
	The \emph{(global) acylindricity parameter}, denoted by $A(\Gamma, X)$, is the smallest non-negative number such that for every $r \in \R_+$,
	\begin{equation*}
		A(\Gamma ,X, r) \leq 4r + A(\Gamma, X).
	\end{equation*}
\end{defi}

Note that \autoref{res: margulis lemma} implies that $A(\Gamma, X) \leq A(\Gamma, X, 400 \delta) + 24\delta$.
Under our assumption, $A(\Gamma, X)$ is always finite, see for instance \cite[Section~6.2]{Coulon:2016if}.

%%%%%%%%%%%%%%%%%%%%%%%%%%%%%%%%%%%%%%%%%%%%%%%%%%%%%%%%%%%%%%%%%%%%%%%%%%%%%%%%%%%%%
%
\subsection{The small cancellation condition}
%
%%%%%%%%%%%%%%%%%%%%%%%%%%%%%%%%%%%%%%%%%%%%%%%%%%%%%%%%%%%%%%%%%%%%%%%%%%%%%%%%%%%%%

Let $\Gamma$ be a torsion-free, non-elementary hyperbolic group and $X$ a Cayley graph of $\Gamma$ (from now on, we will always assume that our hyperbolic groups are torsion-free).
In particular, $X$ is a $\delta$-hyperbolic geodesic space endowed with a proper and co-compact action by isometries of $\Gamma$.
For simplicity, we assume that $\delta > 0$.
Let $\mathcal Q$ be a collection of pairs $(R, Y)$ where $R$ is a subgroup of $\Gamma$ and $Y$ is an $R$-invariant closed strongly quasi-convex subset of $X$.
We say that $\mathcal Q$ is a \emph{relation family} if $\mathcal Q$ is invariant under the action of $\Gamma$ given by $\gamma \cdot (R,Y) = (\gamma R\gamma^{-1},\gamma Y)$, for every $(R,Y) \in \mathcal Q$ and every $\gamma \in \Gamma$.
For such a family, we define two numerical parameters:
\begin{eqnarray*}
	\Delta(\mathcal Q,X) & = & \sup \set{\diam{Y_1^{+3\delta} \cap Y_2^{+3\delta}}}{(R_1,Y_1) \neq (R_2,Y_2) \in \mathcal Q}, \\
	T(\mathcal Q, X) & = & \inf\set{\norm h}{h \in R,\ (R,Y) \in \mathcal Q}.
\end{eqnarray*} 
For readers familiar with the usual small cancellation theory, they respectively play the role of the length of the largest piece and the length of the smallest relation.
Note that if $\Delta(\mathcal Q, X)$ is finite, then $R$ is normal in $\stab Y$ for every $(R,Y) \in\mathcal Q$.
Denote by $K(\mathcal Q)$, or simply $K$, the (normal) subgroup of $\Gamma$ generated by all $R$ where $(R,Y)$ runs over $\mathcal Q$.
The goal of small cancellation theory is to study the quotient $\bar \Gamma = \Gamma / K$, provided $\Delta(\mathcal Q, X) \ll T(\mathcal Q, X)$ and $\delta \ll T (\mathcal Q, X)$.
This condition is formalized by the next definition.

\begin{defi}
\label{defi: small cancellation}
	Let $\lambda, \epsilon \in \R_+^*$.
	\begin{itemize}
		\item The relation family $\mathcal Q$ satisfies the \emph{$C'(\lambda, \epsilon)$ small cancellation condition} if 
		\begin{equation*}
			\Delta(\mathcal Q, X) \leq \lambda T(\mathcal Q, X)
			\quad \text{and} \quad
			\delta \leq \epsilon T (\mathcal Q, X).
		\end{equation*}
		\item The relation family $\mathcal Q$ satisfies the $C'(\lambda, \epsilon)$ \emph{strengthened} small cancellation condition if, in addition to the previous inequalities, we have 
		\begin{equation*}
			\inj \Gamma X \leq \epsilon T(\mathcal Q, X).
		\end{equation*}
		\item A group $\bar \Gamma$ is a \emph{$C'(\lambda, \epsilon)$  (strengthened) small cancellation quotient} of $\Gamma$ if there exists a relation family $\mathcal Q$  satisfying the $C'(\lambda, \epsilon)$ (strengthened) small cancellation condition such that $\bar \Gamma = \Gamma / K(\mathcal Q)$.
	\end{itemize}
\end{defi}

\begin{rema}
    Technically, the definition of a small cancellation quotient depends on $X$. In practice, the choice of $X$ does not make a difference for what we are interested in. Thus, we will suppress $X$ and simply say $\bar \Gamma$ is a \emph{$C'(\lambda, \epsilon)$} small cancellation quotient of $\Gamma$.
\end{rema}

\begin{rema}
	There are cases where $\Delta(\mathcal Q, X)$ and $\delta$ can be zero (or arbitrarily small if we stick with positive hyperbolicity constant).
	Consider, for instance, the free group $\Gamma$  generated by $a$ and $b$ and its Cayley graph $X$ with respect to $\{a,b\}$.
	Let $Y_a$ and $Y_b$ be the respective axes of $a$ and $b$.
	The relation family 
	\begin{equation*}
		\mathcal Q = \set{ \left( \gamma\group u \gamma^{-1}, \gamma Y_u \right)}{\gamma \in \Gamma, u \in \{a, b\}}
	\end{equation*}
	satisfies the $C'(\lambda, \epsilon)$ small cancellation condition, for any $\lambda, \epsilon \in \R_+^*$.
	Nevertheless, the quotient $\Gamma / K(\mathcal Q)$ is trivial, which is not what would be expected from a small cancellation quotient.
	The reason is that the small cancellation condition alone is not enough to impose that the relations of $\mathcal Q$ are very long.
	The strengthened small cancellation condition imposes an additional hypothesis (that is invariant under renormalization) and states that relations are very long compared to other relevant lengths of $\Gamma$ and $X$.
\end{rema}

\begin{defi}
\label{defi tight}
We say that the relation family $\mathcal Q$ is \emph{tight} if: $\mathcal Q / \Gamma$ is finite, the action of $R$ on $Y$ is cobounded, and $R = \stab Y$ for every $(R, Y) \in \mathcal Q$.
Similarly, the small cancellation quotient $\bar \Gamma$ of $\Gamma$ is \emph{tight} if the underlying relation family is tight.
\end{defi}

%------------------------------------------------------------------------------------
%
\subsection{Cone and cone-off}
\label{sec: cone and cone-off}
%
%------------------------------------------------------------------------------------

\paragraph{Cone over a metric space.}
Let $Y$ be a metric space.
The \emph{cone of radius $\rho$ over $Y$}, denoted by $D_\rho(Y)$ or simply $D(Y)$, is the quotient of $Y\times \left[0,\rho\right]$ by the equivalence relation that identifies all the points of the form $(y,0)$.	
The equivalence class of $(y,0)$, denoted by $c$, is called the \emph{apex} or \emph{cone point} of $D(Y)$. 
By abuse of notation, we still write $(y,r)$ for the equivalence class of $(y,r)$.
Given two points $x = (y,r)$ and $x' = (y',r')$ in $D(Y) \setminus\{c\}$, the \emph{angle at $c$ between $x$ and $x'$} is the quantity
\begin{equation*}
	\theta_c(x,x') = \frac {\dist y{y'}}{\sinh \rho}.
\end{equation*}
It defines a pseudo-metric on $D(Y) \setminus\{c\}$.
The cone over $Y$ is endowed with a metric characterized as follows \cite[Chapter I.5, Proposition 5.9]{Bridson:1999ky}.
If $x=(y,r)$ and $x'=(y',r')$ are two points of $D(Y)$, then
\begin{equation}
\label{eqn: sc - metric cone}
	\cosh \dist x{x'} = \cosh r \cosh r' - \sinh r \sinh r' \cos \left(\min\left\{\pi, \theta_c(x,x')  \right\} \right).
\end{equation}
Although the angle $\theta_c$ is not defined when $x$ or $x'$ equals $c$, the above formula still makes sense in this situation, because $\sinh r \sinh r' = 0$.
The embedding $\iota \colon Y \to D(Y)$ sending $y$ to $(y,\rho)$ satisfies
\begin{equation*}
	\dist {\iota(y)}{\iota(y')} = \mu \left(\dist y{y'}\right), \quad \forall y,y' \in Y,
\end{equation*}
where $\mu \colon \R_+ \to \R_+$ is the non-decreasing concave map characterized by 
\begin{equation*}
	\cosh \mu(t) = \cosh^2 \rho - \sinh^2 \rho \cos \left(\min\left\{\pi, \frac {t}{\sinh \rho}\right\} \right), \quad \forall t \in \R_+.
\end{equation*}
In addition, the cone comes with a \emph{radial projection} $p \colon D(Y)\setminus\{c\} \to Y$ mapping $(y,r)$ to $y$.
The next statement provides an estimate of the inverse of $\mu$ on $[0, \pi \sinh \rho]$ that does not depend on $\rho$.

\begin{lemm}
\label{res: invert mu}
	For every $t \in \R_+$, if $\mu(t) < 2\rho$, then $t \leq \pi \sinh(\mu(t)/2)$.
\end{lemm}

\begin{proof}
	The condition $\mu(t) < 2\rho$ forces that $t < \pi \sinh \rho$.
	The conclusion is now a straightforward computation of (hyperbolic) trigonometry.
\end{proof}

\paragraph{Cone-off space.}
Let $\rho \in \R_+^*$.
Let $(X, \distV)$ be a $\delta$-hyperbolic graph and $\mathcal Y$ a collection of strongly quasi-convex subsets of $X$.
Given $Y \in \mathcal Y$, we denote by $\distV[Y]$ the length metric on $Y$ induced by the restriction of $\distV$ on $Y$.
We write $D_\rho(Y)$ for the cone of radius $\rho$ over $(Y, \distV[Y])$.

The \emph{cone-off of radius $\rho$ relative to $\mathcal Y$}, denoted by $\dot X_\rho(\mathcal Y)$ or simply $\dot X$ is the space obtained by attaching for every $Y \in \mathcal Y$, the cone $D_\rho(Y)$ on $X$ along $Y$ according to $\iota \colon Y \to D_\rho(Y)$. We denote by $\mathcal C$ the subset of $\dot X$ consisting of all apices of all attached cones.
We write $\distV[\dot X]$ for the largest pseudo-metric on $\dot X$ such that the canonical maps $X \to \dot X$ as well as $D_\rho(Y) \to \dot X$ are $1$-Lipschitz.
It turns out that $\distV[\dot X]$ is a geodesic metric on $\dot X$ \cite[Theorem~5.38]{Dahmani:2017ef}.
The metrics of $X$ and $\dot X$ are related as follows \cite[Lemma~5.8]{Coulon:2014fr}.
For every $x,x' \in X$ we have
\begin{equation}
\label{eqn:  lower bound dist dot X}
	\mu \left(\dist[X] x{x'}\right) \leq \dist[\dot X] x{x'} \leq \dist[X] x{x'}.
\end{equation}
By construction, the embedding $D_\rho(Y) \to \dot X$ is $1$-Lipschitz for every $Y \in \mathcal Y$.
Since we passed to the length metric on $Y$ before building the cones, this map is not an isometry in general.
Nevertheless, we have the following statements.

\begin{lemm}[{\cite[Proposition~A.2]{Coulon:2021wg}}]
\label{res: qi cone in cone-off}
	For every $Y \in \mathcal Y$, the map $D(Y) \to \dot X$ is $(1,8\delta)$-quasi-isometric embedding.
\end{lemm}

\begin{lemm}[{\cite[Lemma~5.7]{Coulon:2014fr}}]
\label{res: isom embedding cone in cone-off}
	Let $Y \in \mathcal Y$ and $x \in D(Y)$. Denote by $d(x, Y)$ the distance between $x$ and $\iota(Y)$ computed with the metric of $D(Y)$.
	For every $x' \in \dot X$, if $\dist[\dot X] x{x'} < d(x,Y)$, then $x' \in D(Y)$ and $\dist[\dot X] x{x'} = \dist[D(Y)] x{x'}$.
\end{lemm}

In particular, if $c$ stands for the apex of $D(Y)$, then the map $D(Y) \to \dot X$ induces a bijection from $D(Y) \setminus \iota(Y)$ onto $B(c, \rho)$.
Recall that we defined a notion of angle $\theta_c$ at $c$ in $D(Y)$.
The previous observation allows us to extend this notion to $\dot X$.

\begin{defi}
	Let $c \in \mathcal C$.
	Let $\nu$ and $\nu'$ be two non-degenerate geodesics starting at $c$.
	The \emph{angle at $c$ between $\nu$ and $\nu'$} is the quantity
	\begin{equation*}
		\angle_c(\nu, \nu') = \lim_{t \to 0^+} \theta_c\left(\nu(t), \nu'(t)\right)
	\end{equation*}
\end{defi}

Note that any geodesic $\nu \colon [0, \ell] \to D(Y)$ starting at $c$ has the form $\nu(t) = (y,t)$, for some $y \in Y$.
Hence, in the above definition, the quantity $\theta_c\left(\nu(t), \nu'(t)\right)$ is constant on a neighborhood of zero, and the limit exists.

%------------------------------------------------------------------------------------
%
\subsection{The small cancellation theorem}
%
%------------------------------------------------------------------------------------

\paragraph{Metric spaces.}
In this section, $\Gamma$ is still a non-elementary, torsion-free, hyperbolic group and $X$ is a Cayley graph of $\Gamma$.
Let $\mathcal Q$ be a relation family.
For simplicity, we let $K = K(\mathcal Q)$.

We fix a parameter $\rho \in \R_+^*$.
Its value will be made precise later (see \autoref{res: small cancellation}).
We make an abuse of notation and write  $\dot X_\rho(\mathcal Q)$ or simply $\dot X$ for the cone-off of radius $\rho$ over $X$ relative to the collection
\begin{equation*}
	\mathcal Y = \set{Y}{(R,Y) \in \mathcal Q}.
\end{equation*}
Recall that $\mathcal Q$ is $\Gamma$-invariant.
Hence, the action of $\Gamma$ on $X$ naturally extends to an action by isometries on $\dot X$: for every $\gamma \in \Gamma$, $(R,Y) \in \mathcal Q$, and $x = (y,r) \in D(Y)$, we define $\gamma x$ to be the point of $D(\gamma Y)$ given by $\gamma x = (\gamma y,r)$. 
The \emph{radial projection} $p \colon \dot X \setminus \mathcal C \to X$ is the $\Gamma$-equivariant map whose restriction to $X$ is the identity and whose restriction to any punctured cone $D(Y) \setminus \{c\}$ is the radial projection defined above.

\bigskip
The \emph{quotient space}, denoted by $\bar X_\rho(\mathcal Q)$ or simply $\bar X$, is the quotient of $\dot X_\rho(\mathcal Q)$ by the action of the normal subgroup $K$.
The metric of $\dot X$ induces a pseudo-metric on $\bar X$.
The group $\bar \Gamma$ naturally acts by isometries on $\bar X$ so that the canonical projection $f\colon \dot X \to \bar X$ is a $1$-Lipschitz, $\Gamma$-equivariant map.
If $x$ is a point in $X$, we write $\bar x = f(x)$ for its image in $\bar X$.
Similarly, we denote by $\bar {\mathcal C}$ the image of $\mathcal C$ in $\bar X$.
It is a $\bar \Gamma$-invariant subset of $\bar X$.
Moreover, for every distinct $\bar c, \bar c' \in \bar {\mathcal C}$, we have $\dist{\bar c}{\bar c'} \geq 2 \rho$.

\medskip
The next statement is a combination of Propositions~6.4 and 6.7, Corollary~3.12, and Proposition~3.15 in \cite{Coulon:2014fr}.

\begin{theo}
\label{res: small cancellation}
	There exist $\delta_0, \delta_1, \Delta_0, \rho_0 \in \R_+^*$, with the following properties.
	Let $\Gamma$ be a non-elementary, torsion-free, hyperbolic group and $X$ a $\delta$-hyperbolic graph endowed with a proper and co-compact action by isometries of $\Gamma$.
	Let $\mathcal Q$ be a relation family and set $\bar \Gamma = \Gamma / K(\mathcal Q)$.
	Assume that $\rho \geq \rho_0$.
	If $\delta \leq \delta_0$, $\Delta(\mathcal Q,X) \leq \Delta_0$ and $T(\mathcal Q,X) \geq 10\pi \sinh \rho$, then the following hold.
	\begin{enumerate}
		\item \label{enu: small cancellation - hyp cone-off}
		The cone-off space $\dot X = \dot X_\rho(\mathcal Q)$ is geodesic and $\dot \delta$-hyperbolic with $\dot \delta \leq \delta_1$.
		\item \label{enu: small cancellation - hyp}
		The quotient space $\bar X = \bar X_\rho(\mathcal Q)$ is geodesic and $\bar \delta$-hyperbolic with $\bar \delta \leq \delta_1$.
		\item \label{enu: small cancellation - local embedding}
		Let $(R,Y) \in \mathcal Q$.
		Let $\bar c$ be the image in $\bar X$ of the apex $c$ of $D(Y)$.
		The projection $\pi \colon \Gamma \onto \bar \Gamma$ induces an isomorphism from $\stab Y/R$ onto $\stab{\bar c}$.
		\item \label{enu: small cancellation - local isom}
		For every $r \in (0, \rho/20]$, and $x \in \dot X$, if $d(x, \mathcal C) \geq 2r$, then the map $f \colon \dot X \to \bar X$ induces an isometry from $B(x,r)$ onto $B(\bar x, r)$.
		\item \label{enu: small cancellation - translation kernel}
		For every $x \in \dot X$, and $\gamma \in K(\mathcal Q) \setminus\{1\}$, we have $\dist[\dot X]{\gamma x}x \geq \min \{2r, \rho/5\}$, where $r = d(x, \mathcal C)$.
		In particular, $K(\mathcal Q)$ acts freely on $\dot X \setminus \mathcal C$ and the map $f \colon \dot X \to \bar X$ induces a covering map $\dot X \setminus \mathcal C \to \bar X \setminus \bar{\mathcal C}$.
	\end{enumerate}
\end{theo}

\begin{rema}\label{rem: order of magnitude}
	Note that the constants $\delta_0$ and $\Delta_0$ (\resp, $\rho_0$) can be chosen arbitrarily small (\resp, large).
	From now on, we will always assume that $\rho_0 > 10^{20}(\delta_1+1)$ whereas $\delta_0, \Delta_0 < 10^{-10}\delta_1$.
	In particular, for every $\rho \geq \rho_0$ and $r \in \intval 0{10\delta_1}$, we have
	\begin{equation}
	\label{eqn: rho for comparing projection}
		\frac{\sinh \rho}{\sinh(\rho - r)} \leq 2 e^r.
	\end{equation}
	These estimates are absolutely not optimal.
	We chose them very generously to ensure that all the inequalities that we need later are satisfied.
	What really matters is their orders of magnitude, recalled below
	\begin{equation*}
		\max\left\{\delta_0, \Delta_0\right\} \ll \delta_1  \ll \rho_0 \ll \pi \sinh \rho_0. \qedhere
	\end{equation*}
\end{rema}
	
	From now on, we assume that $X$, $\Gamma$, and $\mathcal Q$ are as in \autoref{res: small cancellation}.
	In particular, $\dot X$ and $\bar X$ are respectively $\dot\delta$- and $\bar \delta$-hyperbolic.
	Up to increasing one constant or the other, we can actually assume that $\dot \delta = \bar \delta$.
	Nevertheless, we still keep two distinct notations to remember which space we are working in.

%%%%%%%%%%%%%%%%%%%%%%%%%%%%%%%%%%%%%%%%%%%%%%%%%%%%%%%%%%%%%%%%%%%%%%%%%%%%%%%%%%%%%
%
\subsection{Lifting properties}
%
%%%%%%%%%%%%%%%%%%%%%%%%%%%%%%%%%%%%%%%%%%%%%%%%%%%%%%%%%%%%%%%%%%%%%%%%%%%%%%%%%%%%%

\begin{lemm}
\label{res: subset far from apices}
	Let $r \in \R_+$.
	Let $\bar \gamma \in \bar \Gamma$ such that the set $\bar Z = \fix{\bar \gamma, r}$ is non-empty.
	If $\bar c \in \bar{\mathcal C}$ satisfies $d(\bar c, \bar Z) < \rho - r/2$, then $\bar \gamma$ fixes $\bar c$.
	In particular,  if $\mathcal Q$ is tight and $\bar \gamma \neq 1$, then  $d(\bar c, \bar Z) \geq \rho - r/2$, for every $\bar c \in \bar{\mathcal Q}$.
\end{lemm}

\begin{proof}
	This statement is often implicitly used in \cite{Coulon:2014fr, Coulon:2016if}. Compare also with \cite[Lemma~5.9.3]{Delzant:2008tu}.
	For the benefit of the reader, we review the argument here.
	Suppose that  $d(\bar c, \bar Z) < \rho - r/2$ for some $\bar c \in \bar{\mathcal C}$.
	In particular, there is a $\bar z \in \bar Z$ such that $\dist{\bar c}{\bar z} < \rho - r/2$.
	The triangle inequality yields
	\begin{equation*}
		\dist{\bar \gamma \bar c}{\bar c} \leq 2 \dist {\bar c}{\bar z} + \dist{\bar \gamma \bar z}{\bar z} < 2\rho.
	\end{equation*}
	However, two distinct cone points are at a distance of at least $2\rho$ from each other.
	Hence, $\bar \gamma$ fixes $\bar c$, which completes the proof of the first assertion.
	If $\mathcal Q$ is tight, then the stabilizer of any cone point is trivial -- see \autoref{res: small cancellation}~\ref{enu: small cancellation - local embedding} -- and the second assertion follows.
\end{proof}

\autoref{res: small cancellation}~\ref{enu: small cancellation - local isom} can be used to lift small scale pictures.
More generally, any quasi-convex subset of $\bar X$ avoiding the cone points lifts isometrically in $\dot X$.
This is the purpose of the next statement, which is a combination of Lemmas~4.17 and 4.18 of \cite{Coulon:2018vp}.
See also \cite[Proposition~3.21]{Coulon:2014fr}.

\begin{lemm}
\label{res: sc - lifting quasi-convex}
	Let $\bar Z$ be a subset of $\bar X$ such that for every $\bar z, \bar z' \in \bar Z$, and every $\bar c \in \bar{\mathcal C}$, we have $\gro{\bar z}{\bar z'}{\bar c} > 13 \bar \delta$.
	Let $\bar z_0$ be a point of $\bar Z$ and $z_0$ a pre-image of $\bar z_0$ in $\dot X$.
	Then there is a unique subset $Z$ of $\dot X$ containing $z_0$ such that the map $\dot X \to \bar X$ induces an isometry from $Z$ onto $\bar Z$.
	Moreover, the following holds.
	\begin{enumerate}
		\item \label{enu: sc - lifting quasi-convex - prestab}
		For every $z_1, z_2 \in Z$, and every $\bar \gamma \in \bar \Gamma$, if $\bar \gamma \bar z_1 = \bar z_2$, then there exists a unique pre-image $\gamma \in \Gamma$ of $\bar \gamma$ such that $\gamma z_1 = z_2$.
		Moreover, for every $z,z' \in Z$, if $\bar \gamma \bar z = \bar z'$, then $\gamma z = z'$.
		\item \label{enu: sc - lifting quasi-convex - stab}
		The projection $\Gamma \onto \bar \Gamma$ induces an isomorphism from $\stab Z$ onto $\stab{\bar Z}$.
	\end{enumerate}
\end{lemm}

\begin{rema}
	\autoref{res: sc - lifting quasi-convex} holds, in particular, if $\bar Z$ is an $\alpha$-quasi-convex subset of $\bar X$ such that $d(\bar c, \bar Z) > \alpha + 13 \bar \delta$, for every $\bar c \in \bar {\mathcal C}$.
\end{rema}

\begin{lemm}
\label{res: sc - lifting isometries}
	Assume that $\mathcal Q$ is tight.
	Let $U_0$ be a subset of $\Gamma$ such that $U_0 \cap K \subset \{1\}$ and let $\bar U_0$ be its image in $\bar \Gamma$.
	Let $\bar U$ be a non-empty subset of $\bar \Gamma$ containing $\bar U_0$.
	Assume that 
	\begin{equation*}
		E_\infty(U_0, \dot X) < \frac \rho{100}
		\quad \text{and}\quad
		E_\infty(\bar U, \bar X) <  \frac \rho{100}.
	\end{equation*}
	Then there exists a subset $U$ of $\Gamma$ containing $U_0$ with the following properties.
	\begin{enumerate}
		\item \label{enu: sc - lifting isometries - bij}
		The epimorphism $\pi \colon \Gamma \onto \bar \Gamma$ induces a bijection from $U$ onto $\bar U$.
		\item \label{enu: sc - lifting isometries - product}
		For every $\gamma,\gamma' \in U$, if $\pi(\gamma\gamma')$ belongs to $\bar U$, then $\gamma\gamma' \in U$.
		\item \label{enu: sc - lifting isometries - energy}
		If $U_1$ is a subset of $U$ and $\bar U_1$ its image in $\bar \Gamma$, then $E_\infty(U_1, \dot X) = E_\infty(\bar U_1, \bar X)$ and $E_\infty(U_1, X) \leq \pi \sinh E_\infty(\bar U_1, \bar X)$.
	\end{enumerate}
\end{lemm}

\begin{proof}
	Without loss of generality, we can assume that $U_0$ contains at least the trivial element.
	The existence of a subset $U$ satisfying \ref{enu: sc - lifting isometries - bij} and \ref{enu: sc - lifting isometries - product} is obvious if $\bar U$ is reduced to the trivial element.
	Indeed, in this case, it follows from our assumptions that $U_0$ is reduced to the trivial element, and we can choose $U = \{1\}$.
	In this case, all the energies vanish.
	
	From now on, we will assume that $\bar U$ contains a non-trivial element. 
	We fix $E \in (0, \rho/100)$ such that 
	\begin{equation*}
		E > \max \left\{ E_\infty(U_0, \dot X), E_\infty(\bar U, \bar X), 7\bar \delta \right\}.
	\end{equation*}

	In particular, the subset $Z_0 = \fix{U_0, E}$ of $\dot X$ is non-empty.
	We now build a subset $\bar Z$ of $\bar X$ such that 
	\begin{enumerate}
		\item $\bar Z$ intersects the image of $Z_0$ in $\bar X$.
		\item $\bar Z$ contains $\fix{\bar U, 3E}$.
		\item $\bar Z$ is $10\bar \delta$-quasi-convex and $d(\bar c, \bar Z) \geq \rho - 3E/2$, for every $\bar c \in \bar{\mathcal C}$.
	\end{enumerate}
	If $U_0$ is reduced to the trivial element, we simply take $\bar Z = \fix{\bar U, 3E}$.
	If $U_0$ contains a non-trivial element, say $\gamma_0$, it follows from our assumption that its image $\bar \gamma_0$ in $\bar \Gamma$ is not trivial, and we take  $\bar Z = \fix{\bar U_0, 3E}$.
	In both cases, the conclusions follow from \autoref{res: subset far from apices} and the fact that a set of almost fixed points is $10\bar \delta$-quasi-convex (\autoref{res: set of almost fixed points}).
	
	Pick a point $z_0 \in Z_0$ whose image $\bar z_0$ in $\bar X$ belongs to $\bar Z$.
	We apply \autoref{res: sc - lifting quasi-convex} with the set $\bar Z$ and the point $z_0$.
	In particular, there is a subset $Z \subset \dot X$ containing $z_0$ such that the map $\dot X \to \bar X$ induces an isometry from $Z$ onto $\bar Z$.
	Let $\bar z$ be a point in $\fix{\bar U, E}$ (hence in $\bar Z)$ and $z$ its pre-image in $Z$.
	For every $\bar \gamma \in \bar U$, there is a unique $\gamma \in \Gamma$ such that $\dist[\dot X]{\gamma z}{z} = \dist[\bar X]{\bar \gamma \bar z}{\bar z}$.
	We write $U$ for the set of all elements obtained in this way.

	Let $\gamma \in U_0$ and $\gamma'$ the pre-image in $U$ of $\bar \gamma$.
	Note that
	\begin{equation*}
		\dist[\dot X]{\gamma' z_0}{z_0} \leq \dist[\bar X]{\bar \gamma \bar z_0}{\bar z_0} \leq \dist[\dot X]{\gamma z_0}{z_0} \leq  E \leq \frac \rho{100}.
	\end{equation*}
	The first equality follows from \autoref{res: sc - lifting quasi-convex}, while the second one follows from the fact that the map $\dot X \to \bar X$ is $1$-Lipschitz.
	In particular, $\dist[\dot X]{\gamma z_0}{\gamma' z_0} \leq 2E$.
	By \autoref{res: small cancellation}\ref{enu: small cancellation - translation kernel}, since $z_0$ is far from $\mathcal C$, every element in $K \setminus \{1\}$ moves $z_0$ by a distance of at least $\rho /5$.
	Consequently, $\gamma = \gamma'$.
	In other words, $U_0$ is contained in $U$.
	Reasoning in the same way with the point $z$, we see that the projection $\Gamma \onto \bar \Gamma$ induces a bijection from $U$ onto $\bar U$.
	
	Let us move to the second part of the statement.
	Let $\gamma, \gamma' \in U$.
	Assume that $\pi(\gamma\gamma')$ belongs to $\bar U$.
	By construction, $\pi(\gamma\gamma')$ has a unique pre-image in $U$, say $\gamma_1$.
	Hence, $\gamma_1^{-1}\gamma\gamma'$ is an element of $K$ which moves $z$ by at most $3E$.
	By \autoref{res: small cancellation}\ref{enu: small cancellation - translation kernel}, $\gamma_1^{-1}\gamma\gamma'$ is trivial, hence $\gamma\gamma'$ belongs to $U$.
	
	Let $U_1$ be a subset of $U$ and $\bar U_1$ its image in $\bar \Gamma$ (which is a subset of $\bar U$).
	Without loss of generality, we can assume that $\bar U_1$ is not contained in $\{1\}$.
	Otherwise, so is $U_1$, and all the energies vanish.
	Since the map $\dot X \to \bar X$ is $1$-Lipschitz, we always have $E_\infty(\bar U_1, \bar X) \leq E_\infty(U_1, \dot X)$.
	
	Let $\bar Z_1 = \fix{\bar U_1, 3E}$.
	Note that $\bar Z_1$ contains $\fix{\bar U, E}$ hence $\bar z$.
	It may not intersect $\bar Z_0$, though.
	As a set of almost fixed points, $\bar Z_1$ is $10 \bar \delta$-quasi-convex.
	Moreover, $d(\bar c, \bar Z_1) \geq \rho - 3E/2$, for every $\bar c \in \bar{\mathcal C}$ (\autoref{res: subset far from apices}).
	Therefore, we can also apply \autoref{res: sc - lifting quasi-convex} with the set $\bar Z_1$ and the point $z$.
	It provides a subset $Z_1 \subset \dot X$ such that the projection $\dot X \to \bar X$ induces an isometry from $Z_1$ onto $\bar Z_1$.
	An argument using \autoref{res: sc - lifting quasi-convex}, similar to the previous ones, shows that for every $x,x' \in Z_1$, and every $\gamma \in U_1$, if $\bar \gamma \bar x = \bar x'$, then $\gamma x = x'$, and thus $\dist[\dot X] {\gamma x}x = \dist[\bar X]{\bar \gamma \bar x}{\bar x}$.

	Note that $E_\infty(\bar U_1, \bar X) \leq E_\infty(\bar U, \bar X)$.
	Fix $E_1 \in \R$ such that $E_\infty(\bar U_1, \bar X) < E_1 < E$.
	Choose a point $\bar x$ in $\fix{\bar U_1, E_1} \subset \bar Z_1$.
	Denote by $x \in Z_1$ the pre-image of $\bar x$.
	Observe that 
	\begin{equation*}
		E_\infty(U_1, \dot X) \leq \sup_{\gamma \in U_1} \dist[\dot X]{\gamma x}x \leq \sup_{\gamma \in U_1} \dist[\bar X]{\bar \gamma \bar x}{\bar x} \leq E_1.
	\end{equation*}
	According to \autoref{res: subset far from apices}, $d(\bar x, \bar {\mathcal C}) \geq \rho - E_1/2$, hence $d(x, \mathcal C) \geq \rho  - E_1/2$.
	Let $y \in X$ be the radial projection of $x$ onto $X$ so that $\dist[\dot X] xy \leq E_1/2$.
	Note that $\bar U_1 \bar x$ is contained in $\bar Z_1$.
	It follows from the previous discussion and the triangle inequality that for every $\gamma \in U_1$,
	\begin{equation*}
		\dist[\dot X] {\gamma y}y \leq \dist[\dot X]{\gamma x}x + E_1 \leq \dist[\bar X]{\bar \gamma \bar x}{\bar x} + E_1 \leq 2E_1.
	\end{equation*}
	Combined with \eqref{eqn:  lower bound dist dot X}, we get
	\begin{equation*}
		\mu\left( \dist[\bar X]{\gamma y}y \right) \leq 2E_1, \quad \forall \gamma \in U_1.
	\end{equation*}
	Recall that $\mu$ is non-increasing and continuous.
	Since the above discussion applies for every $E_1$ sufficiently close to $E_\infty(\bar U_1, \bar X)$, we obtain $E_\infty(U_1, \dot X) \leq E_\infty(\bar U_1, \bar X)$ and  $\mu\left( E_\infty(U_1, X)\right) \leq 2 E_\infty(\bar U_1, \bar X)$.
	The conclusion now follows from \autoref{res: invert mu}.
\end{proof}

We can now prove the first version of \autoref{intro:thm:lifting}, for morphisms with small energy. We will see other versions of the same statement later, and the case of big energy will be treated separately in \autoref{sec: lifts}.

\begin{prop}
\label{res: sc - lifting morphism - prelim}
	Assume that $\mathcal Q$ is tight.
	Let $G = \group{U \mid R}$ be a finitely presented group.
	Denote by $\ell$ the length of the longest relation in $R$ (seen as words over $U$).
	Let $H$ be the subgroup of $G$ generated by a subset $U_0$ of $U$.
	Let $\iota \colon H \to \Gamma$ be a morphism.
	Let $\bar \varphi \colon G \to \bar \Gamma$ be a morphism whose restriction to $H$ coincides with $\pi \circ \iota$.
	Assume that  $\iota(U_0) \cap K \subset \{1\}$ and 
	\begin{equation*}
		E_\infty( \iota, U_0, \dot X ) < \frac \rho{100}
		\quad \text{and}\quad
		E_\infty(\bar \varphi, U, \bar X) < \frac \rho{100\ell}.
	\end{equation*}
	Then there exists a morphism $\varphi \colon G \to \Gamma$ such that $\bar \varphi = \pi \circ \varphi$ and $\varphi$ restricted to $H$ coincides with $\iota$.
	Moreover, $E_\infty(\varphi, U, X)\leq \pi \sinh E_\infty(\bar \varphi, U, \bar X)$.
\end{prop}

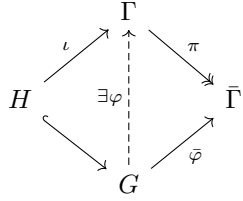
\[\begin{tikzcd}
	& \Gamma \\
	H && {\bar{\Gamma}} \\
	& G
	\arrow["\pi", two heads, from=1-2, to=2-3]
	\arrow["\iota", from=2-1, to=1-2]
	\arrow[hook, from=2-1, to=3-2]
	\arrow["\exists\varphi", dashed, from=3-2, to=1-2]
	\arrow["{\bar{\varphi}}"', from=3-2, to=2-3]
\end{tikzcd}\]

\begin{proof}
	By definition of the $\ell^\infty$-energy, there exist $x_0 \in \dot X$ and $\bar x \in \bar X$ such that
	\begin{equation}
	\label{eqn: sc - lifting morphism}
		\dist[\dot X]{\iota(g) x_0}{x_0} \leq \frac \rho{100}, \quad \forall g \in U_0, 
		\quad \text{and}\quad
		\dist{\bar \varphi(g)\bar x}{\bar x} \leq \frac \rho{100\ell}, \quad \forall \bar g \in \bar U.
	\end{equation}
	Set $W_0 = \iota(U_0) \cup\{1\}$.
	Denote by $\bar W$ the image under $\bar \varphi$ of all elements of $G$ whose word length with respect to $U$ is at most $\ell$ and observe that $\pi$ maps $W_0$ to $\bar W$. 
	It follows from \eqref{eqn: sc - lifting morphism} that
	\begin{equation*}
		\dist[\dot X]{\gamma x_0}{x_0} \leq \frac \rho{100}, \quad \forall \gamma \in W_0, 
		\quad \text{and}\quad
		\dist{\bar \gamma \bar x}{\bar x} \leq \frac \rho{100}, \quad \forall \bar \gamma \in \bar W.
	\end{equation*}
	By \autoref{res: sc - lifting isometries}, there exists a finite subset $W \subset \Gamma$ containing $W_0$ such that  
	\begin{enumerate}
		\item \label{enu: sc - lifting morphism - isom}
		the epimorphism $\pi \colon \Gamma \onto \bar \Gamma$ induces a bijection from $W$ onto $\bar W$;
		\item \label{enu: sc - lifting morphism - morphism}
		for every $\gamma, \gamma' \in W$, if $\pi(\gamma\gamma')$ belongs to $\bar W$, then $\gamma\gamma' \in W$.
	\end{enumerate}
	Let $F$ be the free group generated by $U$, and $q \colon F \onto G$ be the canonical projection associated to the presentation of $G$.
	We define a morphism $\tilde \varphi \colon F \to \Gamma$ by sending $g \in U$ to the unique pre-image of $\bar \varphi(g)$ in $W$.
	Hence, $\pi \circ \tilde \varphi = \bar \varphi \circ q$.
	It follows from \ref{enu: sc - lifting morphism - morphism} that for every $g \in F$ whose length (for the word metric with respect to $U$) is at most $\ell$, the element $\tilde \varphi(g)$ is the unique pre-image of $\bar \varphi \circ q(g)$ in $W$.
	In particular, $\tilde\varphi(g) = 1$ for every $g \in R$.
	Thus, $\tilde \varphi$ factors through the projection $q$.
	We denote by $\varphi \colon G \to \Gamma$ the resulting morphism.
	By construction, it satisfies $\pi \circ \varphi = \bar \varphi$.
	Moreover, $\varphi$ restricted to $H$ coincides with $\iota$.
	The relation between energies follows from \autoref{res: sc - lifting isometries}~\ref{enu: sc - lifting isometries - energy}. 
\end{proof}

The next statement uses equational noetherianity to remove the assumption that $G$ is finitely presented.
Recall that a group $\Gamma$ is \emph{equationally noetherian} if for every finitely generated group $G$, there is a finitely presented group $\hat G$ and a projection $q \colon \hat G \to G$ such that every morphism $\hat G \to \Gamma$ factors through $q$.

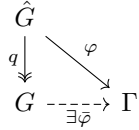
\[\begin{tikzcd}
	{\hat{G}} \\
	G & \Gamma
	\arrow["q"', two heads, from=1-1, to=2-1]
	\arrow["\varphi", from=1-1, to=2-2]
	\arrow["{\exists\bar{\varphi}}"', dashed, from=2-1, to=2-2]
\end{tikzcd}\]

\begin{coro}
\label{res: sc - lifting morphism - prelim en}
	Assume that $\mathcal Q$ is tight.
	Let $G$ be a group generated by a finite set $U$.
	There is $\xi \in (0,1)$ which only depends on $G$, $U$, and $\Gamma$ with the following property.
	Let $H$ be a subgroup of $G$, generated by a subset $U_0$ of $U$.
	Let $\iota \colon H \to \Gamma$ be a morphism.
	Let $\bar \varphi \colon G \to \bar \Gamma$ be a morphism whose restriction to $H$ coincides with $\pi \circ \iota$.
	Assume that  $\iota(U_0) \cap K \subset \{1\}$ and 
	\begin{equation*}
		E_\infty( \iota, U_0, \dot X ) < \frac \rho{100}
		\quad \text{and}\quad
		E_\infty(\bar \varphi, U, \bar X) <\xi \rho.
	\end{equation*}
	Then there exists a morphism $\varphi \colon G \to \Gamma$ such that $\bar \varphi = \pi \circ \varphi$ and $\varphi$ restricted to $H$ coincides with $\iota$.
	Moreover, $E_\infty(\varphi, U, X)\leq \pi \sinh E_\infty(\bar \varphi, U, \bar X)$.
\end{coro}

\begin{proof}
	Recall that every hyperbolic group is equationally noetherian \cite{positivehyp1, Weidmann:2019ue}.
	Hence, there is a finitely presented cover $q \colon \hat G \to G$ such that every morphism $\hat G \to \Gamma$ factors through $q$.
	Without loss of generality, we can assume that $\hat G$ is generated by a finite set $\hat U$ whose image in $G$ is exactly $U$.
	Denote by $\ell$ the length of the longest relation (with respect to the word metric induced by $\hat U$) defining $\hat G$ and set $\xi = 1/100\ell$.
	Note that this data only depends on $G$, $U$, and $\Gamma$.
	
	We choose a subset  $\hat U_0$ of $\hat U$ whose image in $G$ is $U_0$.
	Let $\hat H$ be the subgroup of $\hat G$ generated by $\hat U_0$ and $\hat \iota \colon \hat H \to \Gamma$ be the morphism $\iota$ composed with $q$ restricted to $\hat H$.
	Note that $E_\infty(\hat \iota, \hat U_0, \dot X)  = E_\infty(\iota, U_0, \dot X)$ is less than $\rho / 100$.
	For simplicity, we let $\hat \varphi = \bar \varphi \circ q$.
	In particular, $\hat \varphi$ coincides with $\hat \iota$ on $\hat H$.
	Moreover,
	\begin{equation*}
		E_\infty(\hat \varphi, \hat U, \bar X)  = E_\infty(\bar \varphi, U, \bar X)
	\end{equation*}
	is less than $\xi\rho$.
	Applying \autoref{res: sc - lifting morphism - prelim} with $\hat \iota$ and $\hat \varphi$, we see that there is a morphism $\tilde \varphi \colon \hat G \to \Gamma$ that coincides with $\hat \iota$ on $\hat H$ and such that $\pi \circ \tilde \varphi = \hat \varphi$, while $E_\infty(\tilde \varphi, \hat U, X) \leq \pi \sinh E_\infty(\hat \varphi, \hat U, \bar X)$.
	It follows from our choice of $\hat G$ that $\tilde \varphi$ factors through $q$.
	We denote by $\varphi \colon G \to \Gamma$ the resulting morphism.
	Unwrapping all definitions, we observe that $\varphi$ coincides with $\iota$ on $H$.
	Moreover, $\pi \circ \varphi = \bar \varphi$ and $E_\infty(\varphi, U, X) = E_\infty(\tilde \varphi, \hat U, X)$.
	Hence the result.
\end{proof}

%%%%%%%%%%%%%%%%%%%%%%%%%%%%%%%%%%%%%%%%%%%%%%%%%%%%%%%%%%%%%%%%%%%%%%%%%%%%%%%%%%%%%
%
\subsection{The Greendlinger Lemma revisited}
%
%%%%%%%%%%%%%%%%%%%%%%%%%%%%%%%%%%%%%%%%%%%%%%%%%%%%%%%%%%%%%%%%%%%%%%%%%%%%%%%%%%%%%

The hypotheses on $\iota \colon H \to \Gamma$ in \autoref{res: sc - lifting morphism - prelim en} are stated using the metric of $\dot X$: one requires that $E_\infty(\iota, U_0, \dot X) < \rho/100$.
Since the embedding $X \to \dot X$ is $1$-Lipschitz, the conclusion of  \autoref{res: sc - lifting morphism - prelim en} holds whenever $E_\infty(\iota, U_0, X) < \rho/100$.
This assumption is very crude, though.
The reason is that the embedding $X \into \dot X$ is heavily distorted: for every $(R, Y) \in \mathcal Q$, the image of $Y$ in $\dot X$ is bounded!
The goal of this section is to prove that \autoref{res: sc - lifting morphism - prelim en} still holds with a weaker assumption on $E_\infty(\iota, U_0, X)$, see \autoref{res: sc - lifting morphism}.

\medskip
The Greendlinger Lemma is a key tool in small cancellation theory.
In the geometric point of view that we adopt, it can be conveniently rephrased in terms of angles at apices.
We start with a few useful facts on angles.
The auxiliary parameter
\begin{equation*}
	\vartheta = \frac {\pi \sinh(99 \rho / 100)}{\sinh \rho}
\end{equation*}
will be somehow the angle counterpart of the hyperbolicity constant.
Note that $\vartheta$ converges to zero as $\rho$ tends to infinity.
If $c$ is the apex of $D(Y)$ for some $(R, Y) \in \mathcal Q$, we denote by $\Omega_c$ the quantity
\begin{equation*}
	\Omega_c = \frac 1 {\sinh \rho} \inf_{h \in R \setminus \{1\}} \norm[X]h.
\end{equation*}
It follows from the small cancellation assumption that $\Omega_c \geq 10 \pi$.
In particular, it is much larger than $\vartheta$.

\begin{lemm}
\label{res: small angle at cone point}
	Let $r \in(0,\rho)$.
	Let $c \in \mathcal C$. Let $x,x' \in \dot X \setminus \{c\}$.
	Let $\nu$ and $\nu'$ be two geodesics joining $c$ to $x$ and $x'$ respectively.
	If $\gro x{x'}c \geq r$ (where the Gromov product is computed in $\dot X$) then 
	\begin{equation*}
		\angle_c(\nu, \nu') \leq \frac {\pi\sinh(10\dot \delta)}{\sinh r}.
	\end{equation*}
\end{lemm}

\begin{proof}
	Let $(R, Y) \in \mathcal Q$ be such that $c$ is the apex of $D(Y)$.
	We write $\nu \colon \intval 0\ell \to \dot X$ and $\nu' \colon \intval 0{\ell'} \to \dot X$ for the parametrization by arc length of $\nu$ and $\nu'$.
	Let $z = \nu(r)$ and $z' = \nu'(r)$.
	The hypothesis that $\gro x{x'}c \geq r$ implies that $\dist[\dot X]z{z'} \leq 4\dot \delta$, see for instance \cite[Lemma~2.2(3)]{Coulon:2014fr}.
	By construction, $z$ and $z'$ belong to $D(Y)$.
	More precisely, they can be written as $z = (y,r)$ and $z' = (y',r)$ where $y,y' \in Y$.
	According to \cite[Lemma~2.2]{Coulon:2018ac}, we have
	\begin{equation*}
		\dist[Y] y{y'} \leq \frac {\pi \sinh \rho}{\sinh r} \sinh\left(\frac {\dist[D(Y)] z{z'}}2\right).
	\end{equation*}
	However, the embedding of $D(Y) \into \dot X$ is a $(1, 8\dot \delta)$-quasi-isometric embedding (\autoref{res: qi cone in cone-off}).
	Therefore 
	\begin{equation*}
		\angle_c(\nu, \nu') 
		\leq \frac {\dist[Y] y{y'}}{\sinh \rho} 
		\leq \frac {\pi\sinh(10\dot \delta)}{\sinh r}. \qedhere
	\end{equation*}
\end{proof}

\begin{rema}
\label{rem: small angle at cone point}
    When applying this lemma later, we combine it with the orders of magnitude discussed in \autoref{rem: order of magnitude}. More precisely, we bound the Gromov product $\gro x{x'}c$ to be at least, say, $\gro x{x'}c \geq \rho/50$ (for instance, by showing that $x$ and $x'$ are close to each other while far away from $c$). Together with $10\dot \delta \ll \rho$ from \autoref{rem: order of magnitude}, this gives that 
 	\begin{equation*}
		\angle_c(\nu, \nu') \leq \frac {\pi\sinh(10\dot \delta)}{\sinh \rho/50} \leq \frac{\pi\sinh(99\rho/100)}{2\sinh\rho}  \leq\frac\vartheta2.\qedhere
	\end{equation*}
\end{rema}

\begin{lemm}
\label{res: large angle at cone point}
	Let $c \in \mathcal C$. Let $x,x' \in \dot X \setminus \{c\}$.
	Let $\nu$ and $\nu'$ be two geodesics joining $c$ to $x$ and $x'$ respectively.
	If $\angle_c(\nu, \nu') \geq \pi + \vartheta$, then any geodesic from $x$ to $x'$ passes through $c$.
\end{lemm}

\begin{proof}
	Let $(R, Y)\in \mathcal Q$ be such that $c$ is the apex of the cone $D(Y)$.
	We write $\nu \colon \intval 0\ell \to \dot X$ and $\nu' \colon \intval 0{\ell'} \to \dot X$ for the parametrization by arc length of $\nu$ and $\nu'$.
	Let $\ell_0 = \min \{\ell, \rho/3\}$ and $\ell'_0 = \min \{\ell', \rho/3\}$.
	Denote by $\nu_0$ and $\nu'_0$ the restrictions of $\nu$ and $\nu'$ to $\intval 0{\ell_0}$ and $\intval 0{\ell'_0}$ respectively.
	
	We first claim that if $\angle_c(\nu, \nu') \geq \pi$, then the path obtained by concatenating $\nu_0^{-1}$ and $\nu'_0$ is the unique geodesic joining its endpoints. 
	In particular, it goes through $c$.
	Let $x_0 = \nu(\ell_0)$ and $x'_0 = \nu'(\ell_0)$.
	They are two points of $D(Y)$ which can also be written $x_0 = (y_0, \ell_0)$ and $x'_0 = (y'_0, \ell'_0)$.
	By definition 
	\begin{equation*}
		\dist[Y] {y_0}{y'_0} = \angle_c(\nu, \nu') \sinh \rho.
	\end{equation*}
	Thus, if $\angle_c(\nu, \nu') \geq \pi$, we get $\dist[Y] {y_0}{y'_0} \geq \pi \sinh \rho$.
	Using the metric of $D(Y)$ -- see \eqref{eqn: sc - metric cone} -- we observe that there is a unique geodesic in $D(Y)$ joining $x_0$ to $x'_0$, namely the one connecting radially $x_0$ to $c$ and then $c$ to $x'_0$, i.e., the concatenation $\nu_0^{-1} \nu'_0$.
	Since the embedding $D(Y) \into \dot X$ induces an isometry between the balls of radius $\rho/2$ centered at $\bar c$, which are $2\dot \delta$-quasi-convex in $X$, the path $\nu_0^{-1} \nu'_0$ is also the unique geodesic of $\dot X$ from $x_0$ to $x'_0$, which completes the proof of our claim.
	
	We now focus on the general case and assume that $\angle_c(\nu, \nu') \geq \pi + \vartheta$.
	In particular, the concatenation $\nu^{-1} \nu'$ is a $\rho/3$-local geodesic of $\dot X$ going through $c$.
	Therefore, $\gro x{x'}c \leq 2\dot \delta$ (where the Gromov product is computed in $\dot X$).
	Let $\sigma$ be a geodesic from $x$ to $x'$.
	We denote by $x_1$ and $x'_1$ the point on $\sigma$ such that $\dist[\dot X] x{x_0} = \dist[\dot X] x{x_1}$ and $\dist[\dot X] x{x'_0} = \dist[\dot X] x{x'_1}$.
	In particular, $\dist[\dot X]{x_0}{x_1} \leq 4 \dot \delta$ and $\dist[\dot X]{x'_0}{x'_1} \leq 4 \dot \delta$, see for instance \cite[Lemma~2.2(3)]{Coulon:2014fr}.
	Observe also that $x_1 = x_0$ (\resp, $x'_1 = x'_0$) whenever $\ell \leq \rho/3$ (\resp, $\ell' \leq \rho/3$).
	
	We write $\nu_1$ and $\nu'_1$ for geodesics joining $c$ to $x_1$ and $x'_1$ respectively.
	If $x_1 = x_0$ (\resp, $x'_1 = x'_0$) we simply choose $\nu_1 = \nu_0$ (\resp, $\nu'_1 = \nu'_0$).
	It follows from \autoref{res: small angle at cone point} (see also \autoref{rem: small angle at cone point}) that $\angle_c(\nu_0, \nu_1)$ is equal to zero (if $\nu_0 = \nu_1$) or is most $\vartheta/2$.
	The same holds for $\angle_c(\nu'_0, \nu'_1)$.
	Consequently, $\angle_c(\nu_1, \nu'_1) \geq \pi$.
	According to our previous claim, the concatenation $\nu_1^{-1} \nu'_1$ is the unique geodesic of $\dot X$ from $x_1$ to $x'_1$.
	Moreover, it passes through $c$.
	However, $x_1$ and $x'_1$ lie on the geodesic $\sigma$ from $x$ to $x'$, hence $c$ belongs to $\sigma$.
\end{proof}

\begin{lemm}
\label{res: angle vs distX}
	Let $c \in \mathcal C$.
	Let $x,x' \in X$.
	If $\nu$ and $\nu'$ are geodesics of $\dot X$ joining $c$ to $x$ and $x'$ respectively, then
	\begin{equation*}	
		\angle_c( \nu, \nu') \leq \frac {\dist[X] x{x'}}{\sinh \rho} + \vartheta.
	\end{equation*}
\end{lemm}

\begin{proof}
	Let $(R, Y)\in \mathcal Q$ be such that $c$ is the apex of the cone $D(Y)$.
	Let $y$ and $y'$ be respective projections of $x$ and $x'$ onto $Y$.
	In particular, $\dist[X] y{y'} \leq \dist[X] x{x'} + 10\delta$, see for instance \cite[Lemma~2.12]{Coulon:2014fr}.
	By construction, $z = \nu(\rho)$ and $z' = \nu'(\rho)$ are respective projections of $x$ and $x'$ on the closed ball $\bar B(c, \rho)$ for the metric of $\dot X$.
	They also belong to $Y$.
	For simplicity, we write $P$ and $\dot P$ for the Gromov product $\gro xyz$ computed in $X$ and $\dot X$ respectively.
	In particular, $\dist[X]yz \leq P + 5\delta$, while $\dot P \leq 5 \dot \delta$, see again \cite[Lemma~2.12]{Coulon:2014fr}.
	We claim that 
	\begin{equation*}
		P \leq \frac{\pi \sinh \rho}{\sinh (\rho - 10\dot \delta)} \sinh(10\dot \delta)
		\leq \pi e^{20\dot \delta}.
	\end{equation*}
	Indeed, the first inequality is an application of \cite[Proposition~2.16]{Coulon:2018ac} and the second one follows from our choice of $\rho_0$, see \eqref{eqn: rho for comparing projection}.
	The same analysis works for $y'$ and $z'$.
	Combining the previous discussion with the triangle inequality, we get
	\begin{equation*}
		\dist[X] z{z'} \leq 
		\dist[X]y{y'} + 2P + 10\delta
		\leq \dist[X]x{x'} + 2\pi e^{20\dot \delta} + 20\delta.
	\end{equation*}
	However, $Y$ is strongly quasi-convex, hence
	\begin{equation*}
		\dist[Y] z{z'} \leq \dist[X]x{x'} + 2\pi e^{20\dot \delta} + 28\delta.
	\end{equation*}
	Recall that $z$ and $z'$ lie on $\nu$ and $\nu'$ respectively.
	Hence,
	\begin{equation*}
		\angle_c(\nu, \nu')
		\leq \frac{\dist[Y] z{z'}}{\sinh \rho}
		\leq \frac{\dist[X] x{x'}}{\sinh \rho} + \frac {2\pi e^{20\dot \delta} + 28\delta}{\sinh \rho}
		\leq \frac{\dist[X] x{x'}}{\sinh \rho} + \vartheta. \qedhere
	\end{equation*}
\end{proof}

\begin{lemm}
\label{res: displacement vs angle}
	Let $(R,Y) \in \mathcal Q$ and $c$ be the apex of $D(Y)$.
	Let $x \in \dot X \setminus \{c\}$.
	Let $\nu$ be a geodesic from $c$ to $x$.
	For every $h \in R \setminus \{1\}$, we have $\angle_c(h\nu, \nu) \geq \Omega_c$.
\end{lemm}

\begin{proof}
	Let $r \in (0, \rho)$ be in the domain of the definition of $\nu$.
	In particular, it belongs to $D(Y)$, hence it can be written as $\nu(r) = (y,r)$ for some $y \in Y$.
	By definition, $\angle_c(h\nu, \nu) = \dist[Y]{hy}y / \sinh\rho$.
	However, it follows from the construction that 
	\begin{equation*}
		\dist[Y] {hy}y \geq \dist[X] {hy}y \geq \norm[X] h \geq \Omega_c \sinh\rho. \qedhere
	\end{equation*}
\end{proof}

In order to study the kernel $K(\mathcal Q)$, Dahmani, Guirardel, and Osin studied in \cite{Dahmani:2017ef} its action on $\dot X$ using the formalism of \emph{rotation families}.
Their work provides an analog of the Greendlinger Lemma.

\begin{prop}[Greendlinger Lemma, {\cite[Lemma~5.1.6]{Dahmani:2017ef}}]
\label{res: greendlinger lemma in the cone-off}
	Let $x \in  \dot X$.
	Let $g \in K\setminus\{1\}$.
	There exists $(R, Y) \in \mathcal Q$ such that $c$ lies on a geodesic between $x$ and $gx$ (in $\dot X$), where $c$ is the apex of the cone $D(Y)$.
	Moreover, one of the following holds.
	\begin{enumerate}
		\item $g \in R$ and $\dist xc \leq \rho/10$, 
		\item There exist two points $p,q \in \dot X$ and $h \in H \setminus\{1\}$ such that $p$ (\resp, $q$) lies on a geodesic from $c$ to $x$ (\resp, $gx$) at distance $\rho/20$ from $c$ and $\dist {hq}p \leq 100 \dot \delta$.
	\end{enumerate}
\end{prop}

\begin{rema}
	The statement given in \cite{Dahmani:2017ef} has slightly different numerical parameters (given in terms of $\dot \delta$ instead of $\rho$). 
	However, the proof works verbatim for our statement.
\end{rema}

\medskip
The notion of angle that we defined provides a way to reformulate the Greendlinger lemma in a concise way.

\begin{coro}[Greendlinger Lemma revisited]
\label{res: greendlinger lemma with angle}
	Let $\bar x \in \bar X$.
	Let
	\begin{displaymath}
		\epsilon(\bar x) =  2\min \left\{ \rho/30, d\left(\bar x, \bar {\mathcal C}\setminus\{\bar x\}\right) \right\}.
	\end{displaymath}
	Let $x,x' \in \dot X$ be two distinct preimages of $\bar x$.
	There exists $(R,Y) \in \mathcal Q$ and $h \in R \setminus\{1\}$ satisfying the following properties.
	The apex $c$ of the cone $D(Y)$ is distinct from $x$ and $x'$.
	Moreover, there are geodesics $\nu$ and $\nu'$ of $\dot X$ joining $c$ to $x$ and $x'$ respectively, such that 
	\begin{enumerate}
		\item $\angle_c(h \nu,\nu') \leq \vartheta$.
		\item $\dist[\dot X]{hx}{x'} \leq \dist[\dot X]x{x'} - \epsilon(\bar x)$.
	\end{enumerate}
\end{coro}

\begin{proof}
	In this proof, all the distances are measured either in $\dot X$ or $\bar X$.
	By assumption, there exists $g \in K\setminus\{1\}$ such that $x' = gx$.
	Let $(R, Y)$ be the pair of $\mathcal Q$ given by  \autoref{res: greendlinger lemma in the cone-off}.
	In particular, the apex $c$ of $D(Y)$ is distinct from $x$ and $x'$ and $\gro x{x'}c = 0$.
	We now follow the dichotomy provided by \autoref{res: greendlinger lemma in the cone-off}.
	Assume first that $g \in R$ and $\dist xc \leq \rho/10$. 
	We let $h = g$.
	We choose for $\nu$ any geodesic of $\dot X$ from $c$ to $x$ and let $\nu' = h\nu$.
	Hence, $\angle_c (h\nu, \nu')  = 0$.
	Note also that $2\dist xc \geq 2\epsilon(\bar x)$, hence $\dist{hx}{x'} = 0 \leq \dist x{x'} - \epsilon(\bar x)$.

	Let us now focus on the second case given by \autoref{res: greendlinger lemma in the cone-off}.
	There exist two points $p,p' \in \dot X$ and $h \in R \setminus\{1\}$ with the following properties.
	The point $p$ (\resp, $p'$) lies on a geodesic $\nu$ (\resp, $\nu'$) joining $c$ to $x$ (\resp, $x'$) at distance $\rho/20$ from $c$.
	Moreover, $\dist {hp}{p'} \leq 100 \dot \delta$.
	Hence, $\angle_c(h\nu, \nu') \leq \vartheta$ (\autoref{res: small angle at cone point}).
	Recall that $c$ lies on a geodesic between $x$ and $x'$ (\autoref{res: greendlinger lemma in the cone-off}).
	Hence, the triangle inequality yields
	\begin{align*}
		\dist{hx}{x'} 
		\leq \dist {hx}{hp} + \dist {hp}{p'} + \dist{p'}{x'}
		& \leq \dist xc + \dist c{x'} - \rho/10 + 100 \dot \delta \\
		& \leq \dist x{x'} - \epsilon(\bar x),
	\end{align*}
	which completes the proof of the second case.
\end{proof}

\begin{defi}
	Let $\omega \in \R_+$.
	Let $Z$ be a subset of $\dot X$.
	Let $(R,Y) \in \mathcal Q$ and $c$ be the apex of $D(Y)$.
	We say that $Z$ has \emph{$\omega$-small sector at $c$} if for every $z,z' \in Z \setminus \{c\}$,  for every geodesic $\nu$ and $\nu'$ joining $c$ to $z$ and $z'$ respectively,
	\begin{equation*}
		\angle_c(h\nu, \nu') \geq \pi + \omega, \quad \forall h \in H \setminus\{1\}.
	\end{equation*}	
	The set $Z$ has \emph{$\omega$-small sectors} if it has $\omega$-small sectors at every cone point $c \in \mathcal C$.
\end{defi}

\begin{lemm}
\label{res: embedding very small sector}
	Let $Z$ be a subset of $\dot X$.
	If $Z$ has $2\vartheta$-small sectors, then the map $\dot X \to \bar X$ induces an isometry from $Z$ onto its image.
\end{lemm}

\begin{proof}
	In this proof, all distances are measured either in $\dot X$ or $\bar X$.
	Any subset of $Z$ has $2\vartheta$-small sectors.
	Hence, it suffices to prove the statement when $Z$ is reduced to two points, say $z_1$ and $z_2$.
	Let $\epsilon = \epsilon(\bar z_2)$ be the parameter given by \autoref{res: greendlinger lemma with angle}.
	According to \autoref{res: small cancellation}\ref{enu: small cancellation - translation kernel}, there is a pre-image $z'_2$ of $\bar z_2$ such that $\dist{z_1}{z'_2} = \dist{\bar z_1}{\bar z_2}$.
	Denote by $Z'_2$ the set of all such pre-images.
	We now choose $z'_2 \in Z'_2$ such that 
	\begin{equation}
	\label{enu: embedding very small sector - minimality}
		 \dist{z_2}{z'_2} < \dist{z_2}{z''_2} + \epsilon, \quad \forall z''_2 \in Z'_2
	\end{equation}
	It suffices to prove that $z_2 = z'_2$.
	Suppose, on the contrary, that it is not the case.
	According to \autoref{res: greendlinger lemma with angle}, we can find a cone point $c \in \mathcal C$, a non-trivial element $h \in K \cap \stab c$ and geodesics $\nu_2$ and $\nu_2'$ joining $c$ to $z_2$ and $z'_2$ such that $\angle_c(\nu_2, h\nu_2') \leq \vartheta$ and $\dist{z_2}{hz'_2} \leq \dist{z_2}{z'_2} - \epsilon$.
	We assume first that $c$ lies on a geodesic from $z_1$ to $z'_2$.
	We estimate:
	\begin{align*}
		\dist{\bar z_1}{\bar z_2}
		\leq \dist{z_1}{hz'_2}
		& \leq \dist{z_1}c + \dist c{hz'_2} \\
	 	& \leq \dist{z_1}c + \dist c{z'_2}
		\leq \dist{z_1}{z'_2}
		\leq \dist{\bar z_1}{\bar z_2}.
	\end{align*}
    
    For the first inequality, we used that $h \in K$, for the second one the triangle inequality, and for the third one that $h$ fixes $c$.
	In particular, $\dist{z_1}{hz'_2} =  \dist{\bar z_1}{\bar z_2}$.
	Consequently, $hz'_2$ belongs to $Z'_2$ but is closer to $z_2$ than $z'_2$.
	This contradicts (\ref{enu: embedding very small sector - minimality}).
	Suppose now that $c$ does not lie on a geodesic from $z_1$ to $z'_2$.
	In particular, $c$ and $z_1$ are distinct.
	Moreover, if $\nu_1$ stands for a geodesic from $c$ to $z_1$, then $\angle_c(\nu_1, \nu'_2) < \pi + \vartheta$ (\autoref{res: large angle at cone point}).
	Using the triangle inequality for angles, we have
	\begin{equation*}
		\angle_c(\nu_1, h^{-1}\nu_2)
		\leq \angle_c(\nu_1, \nu'_2) + \angle_c(\nu'_2, h^{-1}\nu_2)
		< \pi + 2\vartheta.
	\end{equation*}
	This contradicts the fact that $\{z_1,z_2\}$ has $2 \vartheta$-small sectors and completes the proof of the statement.
\end{proof}

\begin{lemm}
\label{res: small diam yields small sectors}
	Let $\omega \in \R_+$.
	Let $Z$ be a subset of $X$.
	Suppose that the diameter of $Z$ (measured in $X$) is at most $T(\mathcal Q, X) - (\pi + \omega) \sinh \rho$. Then $Z$ (seen as a subset of $\dot X$) has $(\omega- \vartheta)$-small sectors.
\end{lemm}

\begin{proof}
	Let $(R,Y) \in \mathcal Q$ and $c$ be the apex of $D(Y)$.
	Let $z, z' \in X$.
	Let $\nu$ and $\nu'$ be geodesics joining $c$ to $z$ and $z'$ respectively. \autoref{res: angle vs distX} yields
    \[\angle_c(\nu, \nu') \leq \frac{d(x, x')}{\sinh \rho} + \vartheta \leq \frac{T(\mathcal Q, X)}{\sinh \rho} - (\pi + \omega) + \vartheta \leq \Omega_c - (\pi + \omega - \vartheta).\]
	Combined with \autoref{res: displacement vs angle}, we get for every $h \in R \setminus \{1\}$, 
	\begin{equation*}
		\angle_c(h\nu, \nu') \geq \angle_c(h\nu,\nu) - \angle_c(\nu, \nu') \geq \pi + \omega - \vartheta. \qedhere
	\end{equation*}
\end{proof}

The next statement strengthens \autoref{res: small cancellation}\ref{enu: small cancellation - translation kernel}.

\begin{lemm}
\label{res: proto greendlinger X}
	For every $\gamma \in K\setminus\{1\}$, we have
	\begin{equation*}
		\norm[X]\gamma \geq T(\mathcal Q, X) - (\pi +3\vartheta) \sinh \rho.
	\end{equation*}
\end{lemm}

\begin{proof}
	Suppose that $\norm[X]\gamma < T(\mathcal Q, X) - (\pi +3\vartheta) \sinh \rho$.
	It follows from \autoref{res: small diam yields small sectors} that there is $x \in X$ such that $\{x, \gamma x\}$ has $2\vartheta$-small sectors.
	Since $\gamma$ belongs to $K$, we get from \autoref{res: embedding very small sector} that $\gamma x = x$.
	Hence, $\gamma = 1$, by \autoref{res: small cancellation}\ref{enu: small cancellation - translation kernel}.
\end{proof}

\begin{prop}
\label{res: same small energy}
	Assume that $\mathcal Q$ is tight.
	Let $U \subset \Gamma$ be a non-empty subset of $\Gamma$ and let $\bar U$ be its image in $\bar \Gamma$.
	Assume that 
	\begin{equation*}
		E_\infty(U, X) < \frac 13 \left[ T(\mathcal Q, X) - (\pi + 5\vartheta) \sinh \rho\right]
		\quad \text{and} \quad
		E_\infty(\bar U, \bar X) < \frac\rho{100}.
	\end{equation*}
	Then $U \cap K \subset \{1\}$.
	Moreover, $E_\infty(U, \dot X) = E_\infty(\bar U, \bar X)$.
\end{prop}

\begin{proof}
	Without loss of generality, we can assume that $U$ contains the identity.
	According to \autoref{res: proto greendlinger X}, $U \cap K = \{1\}$.
	Hence, we can assume that $\bar U$ contains a non-trivial element.
	Otherwise, all energies vanish, and the statement holds.
	We now fix $E, \bar E \in \R_+$ such that 
	\begin{equation*}
		E_\infty(U, X)< E < \frac 13 \left[ T(\mathcal Q, X) - (\pi + 5\vartheta) \sinh \rho\right]
	\end{equation*}
	and
	\begin{equation*}
		\max \{E_\infty(\bar U, \bar X) , 7\bar \delta\} < \bar E < \rho/100.
	\end{equation*}
	Let $x$ be a point in $\fix{U, E} \subset X$.
	Denote by $\bar x$ its image in $\bar X$.
	Observe that the set $Ux$ has diameter (measured in $X$) at most $T(\mathcal Q, X) - (\pi + 5\vartheta)\sinh \rho$.
	Hence, it has $2\vartheta$-small sectors (\autoref{res: small diam yields small sectors}).
	In particular, the map $\dot X \to \bar X$ induces an isometry from $Ux$ onto its image (\autoref{res: embedding very small sector}). 

	We write $\bar Z$ for the set $\fix{\bar U, \bar E} \subset \bar X$ and choose a projection $\bar z$ of $\bar x$ onto $\bar Z$.
	We are going to prove that there is a pre-image $z \in \dot X$ of $\bar z$ such that for every $\gamma \in U$, we have  $\dist[\dot X] z{\gamma z} = \dist[\bar X]{\bar z}{\bar \gamma \bar z}$.
	If $\bar z = \bar x$ we can simply take $z = x$.
	Thus, we can assume without loss of generality that $\bar x$ does not belong to $\bar Z$.
	
	According to \autoref{res: set of almost fixed points}, there is $\gamma_0 \in U$ such that $\gro {\bar x}{\bar \gamma_0 \bar x}{\bar z} \leq 5\bar \delta$.
	Fix a geodesic $\sigma$ of $\dot X$ from $x$ to $\gamma_0 x$.
	Since $\dist[\dot X]{\gamma_0 x}x = \dist[\bar X]{\bar \gamma_0 \bar x}{\bar x}$, the image $\bar \sigma$ of $\sigma$ in $\bar X$ is a geodesic from $\bar x$ to $\bar \gamma_0 \bar x$.
	Moreover, the point $\bar z$ is $9\bar \delta$-close to $\bar \sigma$, see for instance \cite[Chapitre~3, Lemme~2.7]{Coornaert:1990tj}.
	We choose for $z \in \dot X$, a pre-image of $\bar z$ that is $9\dot \delta$-close to $\sigma$ (remember that we assumed that $\dot \delta = \bar \delta$).
	
	We now claim that for every $c \in \mathcal C$, if $\nu$, $\nu_1$ and $\nu_2$ are geodesics joining $c$ to $z$, $x$, and $\gamma_0x$ respectively, then either $\angle_c(\nu, \nu_1) \leq \vartheta$ or $\angle_c(\nu, \nu_2) \leq \vartheta$.
	Using the hyperbolicity of $\dot X$, we have
	\begin{equation}
	\label{eqn: same small energy}
		\min\left\{ \gro xcz, \gro {\gamma_0x}cz \right\} 
		\leq \gro x{\gamma_0x}z + \dot \delta 
		\leq d(z, \sigma) + \dot \delta
		\leq 10 \dot \delta
	\end{equation}
	(where all Gromov products are computed in $\dot X$).
	Suppose first that $\gro xcz \leq 10\dot \delta$.
	Since $\bar U$ contains a non-identity element, $d(\bar c, \bar Z) \geq \rho - \bar E/2$, for every $\bar c \in \bar{\mathcal C}$ (\autoref{res: subset far from apices}).
	In particular, $\dist[\dot X] cz \geq \rho - \bar E/2$ and thus
	\begin{equation*}
		\gro xzc \geq \dist[\dot X] cz - \gro xcz \geq 9\rho / 10.
	\end{equation*}
	It follows from \autoref{res: small angle at cone point} that $\angle_c(\nu, \nu_1) \leq \vartheta$.
	If the minimum in \eqref{eqn: same small energy} is achieved by $\gro {\gamma_0x}cz$, we prove in the same way that  $\angle_c(\nu, \nu_2) \leq \vartheta$, which completes the proof of our claim.
	
	Consider now any element $\gamma \in U$.
	The set $\{ x, \gamma_0x , \gamma x, \gamma \gamma_0 x\}$ has diameter at most $T(\mathcal Q, X) - (\pi + 5\vartheta) \sinh \rho$, hence it has $4\vartheta$-small sectors (\autoref{res: small diam yields small sectors}).
	It follows from our previous claim that $\{z, \gamma z\}$ has $2\vartheta$-small sectors.
	Consequently, $\dist[\dot X] z{\gamma z} = \dist[\bar X]{\bar z}{\bar \gamma \bar z}$ (\autoref{res: embedding very small sector}) as we announced.
	This proves in particular that $E_\infty(U, \dot X) \leq E < \rho /100$.
	We already observed that $U \cap K \subset \{1\}$.
	\autoref{res: sc - lifting isometries} applied with $U_0 = U$ and $\bar U$ shows that $E_\infty(U, \dot X) = E_\infty(\bar U, \bar X)$.
\end{proof}

\begin{prop}
\label{res: sc - lifting morphism}
	Assume that $\mathcal Q$ is tight.
	Let $G$ be a group generated by a finite set $U$.
	There is $\xi \in (0,1)$ which only depends on $G$, $U$, and $\Gamma$ with the following property.
	Let $H$ be a subgroup of $G$, generated by a subset $U_0$ of $U$.
	Let $\iota \colon H \to \Gamma$ be a morphism.
	Let $\bar \varphi \colon G \to \bar \Gamma$ be a morphism whose restriction to $H$ coincides with $\pi \circ \iota$.
	Assume that  
	\begin{equation*}
		E_\infty( \iota, U_0, X ) < \frac 13 \left[ T(\mathcal Q, X) - (\pi+ 5\vartheta) \sinh \rho\right]
		\quad \text{and}\quad
		E_\infty(\bar \varphi, U, \bar X) < \xi \rho.
	\end{equation*}
	Then there exists a morphism $\varphi \colon G \to \Gamma$ such that $\bar \varphi = \pi \circ \varphi$ and $\varphi$ restricted to $H$ coincides with $\iota$.
	Moreover, $E_\infty(\varphi, U, X)\leq \pi \sinh E_\infty(\bar \varphi, U, \bar X)$.
\end{prop}

\begin{proof}
	The statement is a direct combination of \autoref{res: same small energy} applied with $\iota(U_0)$ and \autoref{res: sc - lifting morphism - prelim en}.
\end{proof}

%%%%%%%%%%%%%%%%%%%%%%%%%%%%%%%%%%%%%%%%%%%%%%%%%%%%%%%%%%%%%%%%%%%%%%%%%%%%%%%%%%%%%
%
\subsection{Isometries of $\bar X$}
%
%%%%%%%%%%%%%%%%%%%%%%%%%%%%%%%%%%%%%%%%%%%%%%%%%%%%%%%%%%%%%%%%%%%%%%%%%%%%%%%%%%%%%

We now review the properties of isometries of $\bar X$ and elementary subgroups of $\bar \Gamma$.
Recall that $X$ is assumed to be a Cayley graph of $\Gamma$, so that the action of $\Gamma$ on $X$ is proper and co-compact.
Moreover, $\Gamma$ is torsion-free.
In this section, we also assume that $\mathcal Q$ is tight.

We say that an elementary subgroup $\bar E$ of $\bar \Gamma$ (for its action on $\bar X$) \emph{lifts} if there is an elementary subgroup $E$ of $\Gamma$ (for its action on $X$) such that the projection $\pi \colon \Gamma \onto \bar \Gamma$ induces an isomorphism from $E$ onto $\bar E$.

\begin{lemm}[{\cite[Proposition~5.18]{Coulon:2016if}}]
\label{res: lifting elliptic subgroups}
	Every elliptic subgroup of $\bar \Gamma$ lifts to an elliptic subgroup of $\Gamma$.
	Hence $\bar \Gamma$ is torsion-free.
\end{lemm}

\begin{lemm}[{\cite[Proposition~5.26]{Coulon:2016if}}]
\label{res: sc - lifting loxo sbgp}
	Every loxodromic subgroup of $\bar \Gamma$ lifts to a loxodromic subgroup of $\Gamma$.
\end{lemm}

\begin{lemm}[{\cite[Lemma~A.18]{Coulon:2021wg}}]
\label{res: approx - thin loxodromic}
	Assume that there exists $\beta \in \R_+$ such that every loxodromic element of $\Gamma$ is $\beta$-thin (for its action on $X$).
	Then every loxodromic element of $\bar \Gamma$ is $\bar \beta$-thin, where $\bar \beta = \beta + 10\pi \sinh (100\bar \delta)$.
\end{lemm}

\begin{prop}[{\cite[Proposition~6.9]{Coulon:2014fr}}]
\label{res: proper co-compact action quotient}
	The action of $\bar \Gamma$ on $\bar X$ is proper and co-compact.
\end{prop}

\begin{prop}
\label{res: approx - A inv}
	The acylindricity $A(\bar \Gamma, \bar X)$ satisfies
	\begin{equation*}
		A(\bar \Gamma, \bar X) \leq A(\Gamma, X) + 5\pi \sinh( 1000\bar \delta).
	\end{equation*}
\end{prop}

\begin{proof}
	Although the definition of $A(\Gamma, X)$ is slightly different, the proof works verbatim as in \cite[Proposition~5.30]{Coulon:2014fr}.
	See also \cite[Proposition~4.47]{Coulon:2018vp}.
\end{proof}

%%%%%%%%%%%%%%%%%%%%%%%%%%%%%%%%%%%%%%%%%%%%%%%%%%%%%%%%%%%%%%%%%%%%%%%%%%%%%%%%%%%%%
%
\subsection{A rescaling procedure}
%
%%%%%%%%%%%%%%%%%%%%%%%%%%%%%%%%%%%%%%%%%%%%%%%%%%%%%%%%%%%%%%%%%%%%%%%%%%%%%%%%%%%%%

We now sum up the previous study. This will involve a rescaling procedure, so let us remind our convention that graphs are endowed with a path metric where all edges have the same positive length, but not necessarily length $1$.

\begin{theo}
\label{res: recap sc}
	Let $\Gamma$ be a non-elementary, torsion-free, hyperbolic group and $X$ a Cayley graph of $\Gamma$.
	Then there are $\lambda_0, \epsilon_0, \delta \in \R_+^*$, together with two maps
	\begin{equation*}
		a \colon (0,1) \times (0, 1)  \to \R_+^*
		\quad \text{and} \quad
		\rho \colon (0, 1) \times (0, 1)  \to \R_+^*
	\end{equation*}
	such that 
	\begin{equation*}
		\lim_{\lambda, \epsilon \to 0}\rho(\lambda, \epsilon) = \infty
		\quad \text{and} \quad
		\lim_{\lambda, \epsilon \to 0} \frac {a(\lambda, \epsilon)}{\sinh \rho(\lambda, \epsilon)} = \infty,
	\end{equation*}
	satisfying the following properties.

	Let $\lambda \in (0, \lambda_0)$ and $\epsilon \in (0, \epsilon_0)$.
	Let $\mathcal Q$ be a tight relation family satisfying the $C'(\lambda, \epsilon)$ strengthened small cancellation condition.
	Let $\pi \colon \Gamma \onto \bar \Gamma$ be the projection onto the associated quotient.
	Then there exists a $\delta$-hyperbolic, geodesic space $\bar X$ endowed with an action by isometries of $\bar \Gamma$, together with a $\pi$-equivariant map $f \colon X \to \bar X$ such that the following hold.
	\begin{enumerate}
		\item \label{enu: recap sc - tf + action} 
		The group $\bar \Gamma$ is torsion-free.
		Its action on $\bar X$ is proper, co-compact, and non-elementary.		
		\item \label{enu: recap sc - acyl} 
		Every non-trivial element of $\bar \Gamma$ is $\delta$-thin. Moreover, $A(\bar \Gamma, \bar X) \leq \delta$.
		\item \label{enu: recap sc - lip} 
		The map $f \colon X \to \bar X$ is $\kappa$-Lipschitz, where
		\begin{equation*}
			\kappa = \frac{a(\lambda, \epsilon)}{T(\mathcal Q, X)}.
		\end{equation*}
		\item \label{enu: recap sc - injectivity} 
		For every $x \in X$, the intersection of $K(\mathcal Q)$ with the set
		\begin{equation*}
			\set{ \gamma \in \Gamma}{ \dist {\gamma x}x < \left[ 1 - \frac {2\pi \sinh \rho(\lambda, \epsilon)}{a(\lambda, \epsilon)} \right]T(\mathcal Q, X)}
		\end{equation*}
		is reduced to the identity element.
		\item \label{enu: recap sc - lift} 
		Let $G$ be a group generated by a finite set $U$.
		There is $\xi \in (0,1)$ which only depends on $G$, $U$, and $\Gamma$ with the following property.
		Let $H$ be the subgroup of $G$ generated by a subset $U_0$ of $U$.
		Let $\iota \colon H \to \Gamma$ be a morphism.
		Let $\bar \varphi \colon G \to \bar \Gamma$ be a morphism whose restriction to $H$ coincides with $\pi \circ \iota$.
		Assume that  
		\begin{equation*}
			E_\infty\left( \iota, U_0, X \right) < \frac 14 T(\mathcal Q, X)
			\quad \text{and}\quad
			E_\infty\left(\bar \varphi, U, \bar X \right) <  \xi \rho(\lambda, \epsilon).
		\end{equation*}
		Then there exists a map $\varphi \colon G \to \Gamma$ such that $\bar \varphi = \pi \circ \varphi$ and $\varphi$ restricted to $H$ coincides with $\iota$.
		Moreover, $\kappa E_\infty(\varphi, U, X) \leq  \pi \sinh E_\infty(\bar \varphi, U, \bar X)$.
	\end{enumerate}
\end{theo}

\begin{proof}
	We begin the proof by fixing some useful parameters.
	Let $\delta_0$, $\Delta_0$, $\delta_1$, and $\rho_0$ be those given by \autoref{res: small cancellation}.
	We define two maps 
	\begin{equation*}
		a \colon (0,1) \times (0, 1)  \to \R_+^*
		\quad \text{and} \quad
		\rho \colon (0, 1) \times (0, 1)  \to \R_+^*.
	\end{equation*}
	The map $a$ is defined by 
	\begin{equation*}
		a(\lambda, \epsilon) = \min \left\{ \frac {\Delta_0} \lambda, \frac {\delta_0}\epsilon\right\},
	\end{equation*}
	while $\rho$ is chosen so that 
	\begin{equation*}
		\lim_{\lambda, \epsilon \to 0} \rho(\lambda, \epsilon) = \infty
		\quad \text{and} \quad
		\lim_{\lambda, \epsilon \to 0} \frac {a(\lambda, \epsilon)}{\sinh \rho(\lambda, \epsilon)} = \infty.
	\end{equation*}
	
	Since the action of $\Gamma$ is non-elementary, proper, and co-compact, $\inj \Gamma X$ is positive.
    The derived subgroup $[\Gamma, \Gamma]$ contains a non-trivial element, say $\gamma_0$.
	We can find $\lambda_0, \epsilon_0 \in \R_+^*$, such that for every $(\lambda, \epsilon) \in (0, \lambda_0) \times (0, \epsilon_0)$, the following hold:
	\begin{align}
		\label{enu: recap - rho}
		\rho(\lambda, \epsilon) & \geq \rho_0, \\
		\label{enu: recap - a vs rho}
		a(\lambda, \epsilon)  & \geq 10 \pi \sinh \rho(\lambda, \epsilon), \\
		\label{enu: recap - non-elem}
		\norm{\gamma_0} & < \frac 1\epsilon \left[1 - \frac{2\pi \sinh \rho(\lambda, \epsilon)}{a(\lambda, \epsilon)}\right] \inj \Gamma X.
	\end{align}
	We now fix $(\lambda, \epsilon) \in (0, \lambda_0) \times (0, \epsilon_0)$.
	We make an abuse of notations and write $a$ and $\rho$ instead of $a(\lambda, \epsilon)$ and $\rho(\lambda, \epsilon)$, respectively.
	Consider a tight relation family $\mathcal Q$ satisfying the $C'(\lambda, \epsilon)$ strengthened small cancellation assumption.
	We let 
	\begin{equation*}
		\kappa = \frac {a}{T(\mathcal Q, X)}
	\end{equation*}
	and consider the rescaled space $\kappa X$, i.e., 
	\begin{equation*}
		\dist[\kappa X] x{x'} = \kappa \dist[X] x{x'}, \quad \forall x, x' \in X.
	\end{equation*}
	It follows from the small cancellation condition, combined with \eqref{enu: recap - rho} and \eqref{enu: recap - a vs rho} that $\kappa X$ is $\delta_0$-hyperbolic, $\Delta(\mathcal Q, \kappa X) \leq \Delta_0$, and $T(\mathcal Q, \kappa X) \geq 10 \pi \sinh \rho$, with $\rho \geq \rho_0$.
	Moreover, its strengthened version yields 
	\begin{equation}
	\label{eqn: recap - kappa}
		\kappa \leq \frac{\delta_0}{\inj \Gamma X}.
	\end{equation}
	Therefore, we can apply our study of small cancellation theory for the group $\Gamma$ acting on $\kappa X$ and the relation family $\mathcal Q$.
	Denote by $\pi \colon \Gamma \onto \bar \Gamma$ the projection from $\Gamma$ onto $\bar \Gamma = \Gamma / K(\mathcal Q)$.
	
	According to \autoref{res: small cancellation}, there is a $\delta_1$-hyperbolic length space $\bar X$ endowed with an action by isometries of $\bar \Gamma$ as well as a $\pi$-equivariant $1$-Lipschitz map $\kappa X \to \bar X$.
	According to \autoref{res: proper co-compact action quotient}, the action of $\bar \Gamma$ on $\bar X$ is proper and co-compact, hence $\bar X$ is actually a geodesic space, by the Hopf--Rinow Theorem \cite[Chapter~I.3, Proposition~3.7]{Bridson:1999ky}.
	It follows from \autoref{res: lifting elliptic subgroups} that $\bar \Gamma$ is torsion-free.
	Since the map $\kappa X \to \bar X$ is $1$-Lipschitz, the map $f \colon X \to \bar X$ is $\kappa$-Lipschitz, which proves \ref{enu: recap sc - lip}.
	Point~\ref{enu: recap sc - injectivity}  is a direct application of \autoref{res: proto greendlinger X}.
	In particular, the strengthening of the small cancellation assumption combined with \eqref{enu: recap - non-elem} shows that $\gamma_0$ has a non-trivial image in $\bar \Gamma$.
	Consequently, $\bar \Gamma$ is non-elementary, for otherwise $\bar \Gamma$ would be abelian, which completes the proof of \ref{enu: recap sc - tf + action}.
	Point~\ref{enu: recap sc - lift}  follows from \autoref{res: sc - lifting morphism}.	
	The action of $\Gamma$ on $X$ is proper and co-compact, hence there is $\beta$ such that any non-trivial element of $\Gamma$ is $\beta$-thin.
	Thus, every non-trivial element in $\bar \Gamma$ is $\bar \beta$-thin where 
	\begin{equation*}
		\bar \beta 
		\leq \kappa \beta +  10\pi \sinh (100\delta_1)
		\leq \frac{\beta \delta_0}{\inj \Gamma X} +  10\pi \sinh (100\delta_1).
	\end{equation*}
	The first inequality is indeed an application of \autoref{res: approx - thin loxodromic}, while the second one follows from (\ref{eqn: recap - kappa}).
	Combining \autoref{res: approx - A inv} with (\ref{eqn: recap - kappa}), we get
	\begin{equation*}
		A(\bar \Gamma, \bar X) 
		\leq A(\Gamma, \kappa X) + 5\pi \sinh(1000\delta_1)
		\leq \frac {A(\Gamma, X) \delta_0}{\inj \Gamma X} + 5\pi \sinh(1000\delta_1).
	\end{equation*}
	Finally we choose 
	\begin{equation*}
		\delta = \max \left\{ \delta_1, \frac{\beta \delta_0}{\inj \Gamma X} +  10\pi \sinh (100\delta_1),  \frac {A(\Gamma, X) \delta_0}{\inj \Gamma X} + 5\pi \sinh(1000\delta_1)  \right\}. 
	\end{equation*}
	Note that $\delta$ does not depend on $\lambda$, nor on $\epsilon$.
	Point \ref{enu: recap sc - acyl} follows from the previous discussion.
\end{proof}

%
%%%%%%%%%%%%%%%%%%%%%%%%%%%%%%%%%%%%%%%%%%%%%%%%%%%%%%%%%%%%%%%%%%%%%%%%%%%%%%%%%%%%%
%%%%%%%%%%%%%%%%%%%%%%%%%%%%%%%%%%%%%%%%%%%%%%%%%%%%%%%%%%%%%%%%%%%%%%%%%%%%%%%%%%%%%

%%%%%%%%%%%%%%%%%%%%%%%%%%%%%%%%%%%%%%%%%%%%%%%%%%%%%%%%%%%%%%%%%%%%%%%%%%%%%%%%%%%%%
%%%%%%%%%%%%%%%%%%%%%%%%%%%%%%%%%%%%%%%%%%%%%%%%%%%%%%%%%%%%%%%%%%%%%%%%%%%%%%%%%%%%%
%
\section{Graphs of groups and graphs of actions}
\label{sec: graph of groups}

%
%%%%%%%%%%%%%%%%%%%%%%%%%%%%%%%%%%%%%%%%%%%%%%%%%%%%%%%%%%%%%%%%%%%%%%%%%%%%%%%%%%%%%
%%%%%%%%%%%%%%%%%%%%%%%%%%%%%%%%%%%%%%%%%%%%%%%%%%%%%%%%%%%%%%%%%%%%%%%%%%%%%%%%%%%%%

Now that we have proved \autoref{intro:thm:lifting} for morphisms with ``small energy'' (\autoref{res: recap sc}), we start to work towards the proof of the general case. In this section, we cover the first set of tools that we will need: graphs of actions and the Rips--Sela machine for groups acting on $\R$-trees.

\subsection{$G$-trees}

Let $G$ be a group.
A \emph{$G$-tree} is a simplicial tree $S$ endowed with a simplicial action of $G$ without inversion.
In this context, a subgroup of $G$ is \emph{elliptic} if it fixes a point.
The action of $G$ is \emph{minimal} if $S$ does not contain any proper $G$-invariant subtree.
Given a $G$-tree, one of the following holds: $G$ has a global fixed point, $G$ fixes an end, or $G$ contains loxodromic elements and there exists a minimal $G$-invariant subtree.
A \emph{splitting} of $G$ is a non-trivial minimal $G$-tree.
Bass--Serre theory builds a correspondence between minimal $G$-trees and \emph{graph of groups decompositions of $G$} \cite{Serre:1980aa}.
In this article, we favor the point of view of $G$-trees.

\subsection{JSJ decomposition}

Let $G$ be a group and $\mathcal H$ a collection of subgroups of $G$. Recall that $G$ is \emph{finitely generated relative to $\mathcal H$} if it is generated by the union of the groups in $\mathcal H$ and finitely many additional elements. Moreover, $G$ is \emph{freely indecomposable} (respectively, \emph{one-ended}) \emph{relative to $\mathcal H$} if it does not act on a tree with trivial (respectively, finite) edge groups, such that the groups in $\mathcal H$ are elliptic. We recall that, unless $G = \Z$ and $\mathcal H$ is trivial, $G$ being freely indecomposable relative to $\mathcal H$ is equivalent to $G$ not splitting as $G = G_1 \ast G_2$, where $G_i \neq \{1\}$ and each element of $\mathcal H$ is conjugated into $G_1$ or $G_2$.

We will always use $\mathcal A$ to denote the collection of abelian subgroups of $G$ (in the literature on JSJ decompositions, more general families are allowed).

An \emph{$(\mathcal A,\mathcal H)$-tree} is a simplicial $G$-tree such that each edge stabilizer belongs to $\mathcal A$ (i.e., is abelian) while the subgroups in $\mathcal H$ are elliptic.
A subgroup $E$ of $G$ is $(\mathcal{A}, \mathcal{H})$-\emph{universally elliptic}, if it is elliptic in every $(\mathcal A, \mathcal H)$-tree. In accordance with the literature \cite{Guirardel:2017te}, we will just talk about \emph{universally elliptic} subgroups when $\mathcal{A}$ and $\mathcal{H}$ are clear from the context.
An $(\mathcal A, \mathcal H)$-tree is \emph{universally elliptic} if its edge stabilizers are universally elliptic.

Given two $(\mathcal A, \mathcal H)$-trees $S_1$ and $S_2$, we say that $S_1$ \emph{dominates} $S_2$ if there is a $G$-equivariant map $S_1 \to S_2$.
Such a map is a \emph{collapse map} if the pre-image of any subtree is still a subtree.
In this case, we also say that $S_1$ is a \emph{refinement} of $S_2$.
Two trees are \emph{compatible} if they have a common refinement.

\begin{defi}
\label{def: JSJ tree}
	The \emph{JSJ deformation space of $G$ over $\mathcal A$ relative to $\mathcal H$} is the collection of universally elliptic $(\mathcal A,\mathcal H)$-trees which dominate every other universally elliptic $(\mathcal A, \mathcal H)$-tree.
	A point of the JSJ deformation space is called a \emph{JSJ tree}.
	The corresponding graph of groups decomposition of $G$ is called a \emph{JSJ decomposition}.
	A vertex stabilizer of a JSJ tree is \emph{rigid} if it is universally elliptic and \emph{flexible} otherwise.
    (Note that edge stabilizers are universally elliptic by assumption, but vertex stabilizers need not.)
\end{defi}

\begin{defi}
\label{defi:TQH}
	A subgroup $Q$ of $G$ is \emph{quadratically hanging}, or simply \emph{QH} (over $\mathcal A$ relative to $\mathcal H$) if 
	\begin{enumerate}
		\item $Q$ is the stabilizer of a vertex $v$ of an $(\mathcal A, \mathcal H)$-tree,
		\item $Q \cong \pi_1(\Sigma)$, where $\Sigma$ is a compact hyperbolic $2$-orbifold;
		\item each edge group incident to $v$ is an \emph{extended boundary subgroup}: by definition, this means that it is either finite or contained in a boundary subgroup $B$ of $\pi_1(\Sigma)$.
	\end{enumerate}
	Although it is not explicit in the terminology, the orbifold $\Sigma$ is part of the structure of a QH subgroup.
\end{defi}

QH subgroups are usually defined by allowing a fiber, that is, $Q$ fits in a short exact sequence $1 \to F \to Q \to \pi_1(\Sigma) \to 1$. In this paper, the fiber is always trivial, hence our more restrictive definition.

Recall that a group $G$ is \emph{Conjugately Separated Abelian} (\emph{CSA}) if every maximal abelian subgroup of $G$ is malnormal.
For instance, torsion-free hyperbolic groups are CSA \cite[Proposition~12]{Myasnikov:1996aa}.

\begin{theo}[See Guirardel-Levitt {\cite[Theorem~9.5]{Guirardel:2017te}}]
\label{res: jsj - existence}	
	Let $G$ be a finitely generated, torsion-free, CSA group.
	Let $\mathcal A$ be the class of all abelian subgroups of $G$.
	Let $\mathcal H$ be a collection of subgroups of $G$ such that $G$ is one-ended relative to $\mathcal H$.
	Then the JSJ deformation space of $G$ over $\mathcal A$ relative to $\mathcal H$ exists.
	Moreover, every flexible vertex group is either abelian or QH.
	
	The JSJ deformation space contains a \emph{canonical JSJ tree} that is invariant under all automorphisms of $G$ preserving $\mathcal H$, and is compatible with every $(\mathcal A, \mathcal H)$-tree.
\end{theo}

A model of the canonical JSJ tree is the collapse of the tree of cylinders. We will only need that a canonical JSJ tree exists, so the specific model does not matter.

%%%%%%%%%%%%%%%%%%%%%%%%%%%%%%%%%%%%%%%%%%%%%%%%%%%%%%%%%%%%%%%%%%%%%%%%%%%%%%%%%%%%%
%%%%%%%%%%%%%%%%%%%%%%%%%%%%%%%%%%%%%%%%%%%%%%%%%%%%%%%%%%%%%%%%%%%%%%%%%%%%%%%%%%%%%
%
\subsection{Modular group}
%
%%%%%%%%%%%%%%%%%%%%%%%%%%%%%%%%%%%%%%%%%%%%%%%%%%%%%%%%%%%%%%%%%%%%%%%%%%%%%%%%%%%%%%%%%%%%%%%%%%%%%%%%%%%%%%%%%%%%%%%%%%%%%%%%%%%%%%%%%%%%%%%%%%%%%%%%%%%%%%%%%%%%%%%%%%

\label{sec: modular group}
Let $G$ be a finitely generated, torsion-free, CSA group.
Let $\mathcal A$ stand for the class of all abelian subgroups of $G$, while $\mathcal H$ is an arbitrary class of subgroups of $G$.
We assume that $G$ is one-ended relative to $\mathcal H$.

\paragraph{Dehn twist.}
Consider a one-edge splitting of $G$ over abelian groups of the form $G = A \ast_C B$ or $G = A \ast_C$.
Let $u \in G\setminus\{1\}$ be an element centralizing $C$.
A \emph{Dehn twist} by $u$ is an automorphism which fixes $A$ and 
\begin{itemize}
	\item conjugates $B$ by $u$, in the amalgamated product case;
	\item sends the stable letter $t$ to $tu$, in the HNN case.
\end{itemize}
We say that such a Dehn twist is \emph{relative to $\mathcal H$} if every subgroup $H \in \mathcal H$ is conjugated into $A$ or $B$.

\begin{defi}
\label{def: modular group}
	The \emph{modular group of $G$ relative to $\mathcal H$} (or simply the \emph{modular group of $G$}) is the subgroup of $\aut G$ generated by all the Dehn twists relative to $\mathcal H$.
	We denote it by $\mcg{G, \mathcal H}$ or simply $\mcg G$ if $\mathcal H$ is empty.
\end{defi}

\begin{rema}
	By definition, for every $H \in \mathcal H$, the restriction to $H$ of any automorphism $\alpha \in \mcg{G, \mathcal H}$ coincides with the conjugation by an element of $G$.
\end{rema}

Let $S$ be an $(\mathcal A, \mathcal H)$-tree.
Given an edge $e$ of $S$, a \emph{Dehn twist over $e$} is a Dehn twist in the one-edge splitting obtained by collapsing all the edges of $S$ which are not in the orbit of $e$.
Let $v$ be a vertex of $S$.
Any automorphism $\alpha_v$ of the vertex group $G_v$ which acts by conjugacy on the adjacent edge groups can be extended to an automorphism $\alpha$ of $G$.
We call $\alpha$ the \emph{standard extension} of $\alpha_v$; this is defined up to inner automorphism, see for instance Perin \cite[Section~3.1]{Perin:2008aa}.
Assume now that $G_v$ is abelian.
Let $E_v$ be the subgroup of $G_v$ generated by the stabilizer of every edge in $S$ adjacent to $v$.
The \emph{peripheral subgroup} of $v$ is the subgroup
\begin{equation}
\label{eq - peripheral subgroup}
	K_v = \bigcap_{\phi} \ker \phi,
\end{equation}
where $\phi$ runs over all morphisms $\phi \colon G_v \to \Z$ whose kernel contains $E_v$.

Assume now that $S$ is the canonical JSJ tree of $G$ over $\mathcal A$ relative to $\mathcal H$ (as in \autoref{res: jsj - existence}).
It follows from the fact that $S$ is compatible with every $(\mathcal A, \mathcal H)$-tree that $\mcg{G, \mathcal H}$ is equal to the subgroup of $\aut G$ generated by
\begin{itemize}
	\item inner automorphisms,
	\item Dehn twists over edges of $S$,
	\item standard extensions of automorphisms of QH vertex groups (called \emph{surface type automorphisms}),
	\item standard extensions of automorphisms of abelian vertex groups $G_v$ that fix the peripheral subgroup $K_v$ (called \emph{generalized Dehn twists}).
\end{itemize}

%%%%%%%%%%%%%%%%%%%%%%%%%%%%%%%%%%%%%%%%%%%%%%%%%%%%%%%%%%%%%%%%%%%%%%%%%%%%%%%%%%%%%
%%%%%%%%%%%%%%%%%%%%%%%%%%%%%%%%%%%%%%%%%%%%%%%%%%%%%%%%%%%%%%%%%%%%%%%%%%%%%%%%%%%%%
%
\subsection{Graphs of actions}
%
%%%%%%%%%%%%%%%%%%%%%%%%%%%%%%%%%%%%%%%%%%%%%%%%%%%%%%%%%%%%%%%%%%%%%%%%%%%%%%%%%%%%%
%%%%%%%%%%%%%%%%%%%%%%%%%%%%%%%%%%%%%%%%%%%%%%%%%%%%%%%%%%%%%%%%%%%%%%%%%%%%%%%%%%%%%

Recall that an $\R$-tree is a $0$-hyperbolic geodesic metric space. Graphs of actions were formalized by Levitt to decompose the action of a given group on an $\R$-tree \cite{Levitt:1994aa}.
We follow here \cite[Definition~4.3]{Guirardel:2004li}.

\begin{defi}
\label{def: graph of actions}
	A \emph{graph of actions on $\R$-trees} $\Lambda$ consists of the following data.
	\begin{enumerate}
		\item A group $G$ and a $G$-tree $S$ (called the \emph{skeleton} of the graph of actions).
		\item An $\R$-tree $Y_v$ (called the \emph{vertex tree}) for each vertex $v$ of $S$.
		\item An \emph{attaching point} $p_e \in Y_{v}$ for each oriented edge $e$ of $S$ with endpoint $v$.
	\end{enumerate}
	Moreover, all these data should be invariant under $G$, i.e.,
	\begin{itemize}
		\item $G$ acts on the disjoint union of the vertex trees so that the projection $Y_v \mapsto v$ is equivariant,
		\item for every edge $e$ of $S$, for every $g \in G$, we have $p_{ge}=gp_e$. 
	\end{itemize}
    If all the edge groups of $S$ are abelian, we speak of a \emph{graph of actions over abelian groups}.
\end{defi}

Given a graph of actions one builds an $\R$-tree $T_\Lambda$ endowed with an action by isometries of $G$: $T_\Lambda$ is obtained from the disjoint union of all trees $Y_v$ by identifying, for every edge $e$ of $S$, the attaching points $p_e \in Y_{v_1}$ and $p_{\bar e} \in Y_{v_2}$ where $v_1$ (\resp, $v_2$) is the terminal (\resp, initial) vertex of $e$.
We say that the action of $G$ on an $\R$-tree $T$ decomposes as a graph of actions if there is a graph of actions $\Lambda$ and a $G$-equivariant isometry from $T$ onto $T_\Lambda$.

\begin{rema}
	If no confusion can arise, we write $\Lambda$ for both the graph of actions and its underlying graph of groups decomposition.
\end{rema}

\begin{defi}
\label{def: action type on R-tree}
	Let $H$ be a group acting by isometries on an $\R$-tree $Y$.
	We say that the action is 
	\begin{enumerate}
		\item \emph{simplicial} if $Y$ is simplicial and the action of $H$ on $Y$ is simplicial;
		\item \emph{axial} if $Y$ is a line and the image of $H$ in $\isom Y$ is a finitely generated group acting with dense orbits on $Y$;
		\item of \emph{Seifert type} if the action has a kernel $N$ and the faithful action of $H/N$ on $Y$ is dual to an arational measured foliation on a closed $2$-orbifold with boundary.
		\end{enumerate}
\end{defi}

We refer the reader to \cite[Exposé~11]{Fathi:2012wo} for the definition of arational measured foliations.

The existence of a decomposition as a graph of actions is provided by the next statement, originally due to Rips and Sela \cite{Rips:1994jg}. Recall that the action of a group on an $\R$-tree is \emph{super-stable} if it satisfies the following property: given any two arcs $ I \subset J$, if the pointwise stabilizer of $J$ is non-trivial, then it coincides with the pointwise stabilizer of $I$.

\begin{theo}[{\cite[Theorem~5.1]{Guirardel:2008ik}}]
\label{res: splitting relative version}
	Let $G$ be a group and $\mathcal H$ a finite collection of subgroups of $G$ such that $G$ is finitely generated relative to $\mathcal H$.
	Let $T$ be an $\R$-tree endowed with a minimal, super-stable action of $G$ so that each subgroup in $\mathcal H$ is elliptic.
	Then one of the following holds
	\begin{itemize}
		\item The group $G$ splits as a free product relative to $\mathcal H$.
		\item The action of $G$ on $T$ splits as a graph of actions where each vertex action is either simplicial, axial, or of Seifert type.
	\end{itemize}	
\end{theo}

%%%%%%%%%%%%%%%%%%%%%%%%%%%%%%%%%%%%%%%%%%%%%%%%%%%%%%%%%%%%%%%%%%%%%%%%%%%%%%%%%%%%%
%%%%%%%%%%%%%%%%%%%%%%%%%%%%%%%%%%%%%%%%%%%%%%%%%%%%%%%%%%%%%%%%%%%%%%%%%%%%%%%%%%%%%
%
\section{Limit groups}
\label{sec: limit}
%
%%%%%%%%%%%%%%%%%%%%%%%%%%%%%%%%%%%%%%%%%%%%%%%%%%%%%%%%%%%%%%%%%%%%%%%%%%%%%%%%%%%%%
%%%%%%%%%%%%%%%%%%%%%%%%%%%%%%%%%%%%%%%%%%%%%%%%%%%%%%%%%%%%%%%%%%%%%%%%%%%%%%%%%%%%%
\label{sec: space marked groups}

After the Rips--Sela machine, the second set of tools that we will need for the proof of \autoref{intro:thm:lifting} is limit groups and the shortening argument. We will use the topological approach that views limit groups as limits of certain groups in the compact space of marked groups, following \cite{Champetier:2000jx, Champetier:2005ic}.

%%%%%%%%%%%%%%%%%%%%%%%%%%%%%%%%%%%%%%%%%%%%%%%%%%%%%%%%%%%%%%%%%%%%%%%%%%%%%%%%%%%%%
%
\subsection{The space of marked groups}
%
%%%%%%%%%%%%%%%%%%%%%%%%%%%%%%%%%%%%%%%%%%%%%%%%%%%%%%%%%%%%%%%%%%%%%%%%%%%%%%%%%%%%%

\paragraph{Definition.}
Let $G$ be a finitely generated group.
A \emph{group marked by $G$} is a pair $(H,\varphi)$ where $H$ is a group and $\varphi \colon G \onto H$ is an epimorphism from $G$ onto $H$.
If there is no ambiguity, we omit the group $G$ and simply say that $(H,\varphi)$ is a \emph{marked group}.
Given two groups $(H_1,\varphi_1)$ and $(H_2,\varphi_2)$ marked by $G$, we say that $(H_1,\varphi_1)$ is a \emph{cover} of $(H_2,\varphi_2)$ or $(H_2, \varphi_2)$ is \emph{quotient} of $(H_1, \varphi_1)$ and write $(H_2,\varphi_2) \prec (H_1,\varphi_1)$ if $\ker \varphi_1 \subset \ker \varphi_2$.
The relation $\prec$ defines a pre-order on the set of marked groups (note that equality is allowed).
Two marked groups $(H_1,\varphi_1)$ and $(H_2,\varphi_2)$ are \emph{equivalent}, if $(H_1,\varphi_1)$ is a cover of $(H_2,\varphi_2)$ and vice versa.

\begin{defi}
\label{def: space of marked groups}
	The \emph{space of groups marked by $G$} (or simply the \emph{space of marked groups}) is the set of equivalence classes of groups marked by $G$.
We denote it by $\mathfrak G(G)$.
\end{defi}

\begin{nota}
	In the remainder of the paper, we make the following abuse of notation: given a group $(H,\varphi)$ marked by $G$, we still denote by $(H,\varphi)$ its equivalence class in $\mathfrak G(G)$.

\end{nota}

\paragraph{Topology.}
We endow the space $\mathfrak G(G)$ with a topology defined as follows.
Let $(H,\varphi)$ be a group marked by $G$.
Let $U$ be a finite subset of $G$.
We let 
\begin{equation*}
	\mathfrak V_U(H,\varphi) = \set{(H',\varphi') \in \mathfrak G(G)}{U \cap \ker \varphi = U \cap \ker \varphi'}.
\end{equation*}
When $U$ runs over all finite subsets of $G$, the collection $\{\mathfrak V_U(H,\varphi)\}_{U \subset G}$ describes a basis of open neighborhoods of $(H,\varphi)$.
The space $\mathfrak G(G)$ endowed with this topology is metrizable and compact.
Note that if $(L, \eta)$ is the limit of a sequence  $(H_k, \varphi_k)$ of marked groups, then $L$ is the quotient of $G$ by the \emph{stable kernel} of $(\varphi_k)$, i.e., the normal subgroup
\begin{equation}
\label{eqn: stable kernel}
	K = \set{\gamma \in G}{\varphi_k(\gamma) = 1,\ \text{for all but finitely many}\ k \in \N}.
\end{equation}

\begin{rema}
	Sometimes it is convenient to consider pairs $(H,\varphi)$, where the morphism $\varphi \colon G \to H$ is not necessarily onto.
	In this situation, we make an abuse of notation and write $(H, \varphi)$ to mean $(\varphi(G), \varphi)$.
	In particular, we say that a sequence $(H_k, \varphi_k)$ of such pairs converges to $(L, \eta)$ and write 
	\begin{equation*}
		(L, \eta) = \lim_{k \to \infty} (H_k,\varphi_k),
	\end{equation*}
	if $(\varphi_k(G), \varphi_k)$ converges to $(L, \eta)$ in $\mathfrak G(G)$.
\end{rema}

By construction, the pre-order $\prec$ induces an order on $\mathfrak G(G)$.
It is compatible with the topology in the sense that the set 
\begin{equation*}
	\set{\left((H_1,\varphi_1), (H_2, \varphi_2)\right) \in \mathfrak G(G) \times\mathfrak G(G)}{ (H_1, \varphi_1)\prec (H_2, \varphi_2) }
\end{equation*}
is closed for the product topology.

\paragraph{Changing the marker.}
Given an epimorphism $\pi \colon G \onto G'$, we define a map $\pi_\ast \colon \mathfrak G(G') \to \mathfrak G(G)$ by sending $(H,\varphi)$ to $(H,\varphi \circ \pi)$.
This is an order-preserving homeomorphism from $\mathfrak G(G')$ onto its image, which is automatically closed.
In addition, if the kernel of $\pi \colon G \to G'$ is the normal closure of a \emph{finite} subset of $G$, then the image of $\pi_\ast$ is an \emph{open} subset of $\mathfrak G(G)$.

\paragraph{Strong cover of a marked group.}
In general, limit groups (\autoref{def: limit group - general}) need not be finitely presented, which creates some technical challenges.
In order to bypass these, it is common to replace limit groups with finitely presented covers.
For our purpose, we need these covers to reflect a given graph of groups decomposition.
This motivates the next definitions.

\begin{defi}
\label{def: strong cover}
	Let $G$ be a finitely generated group.
	Let $(H,\varphi)$ and $(\hat H, \hat \varphi)$ be two groups marked by $G$.
	Let $S$ (\resp, $\hat S$) be a splitting of $H$ (\resp, $\hat H$) over abelian groups.
	We say that the triple $(\hat H,\hat \varphi,\hat S)$ is a \emph{strong cover} of $(H,\varphi, S)$ if there exist an epimorphism $\zeta \colon \hat H \onto H$ and a 
	$\zeta$-equivariant map $f \colon \hat S \to S$ such that the following hold.
	\begin{enumerate}
		\item $\zeta \circ \hat \varphi = \varphi$. In particular, $(H, \varphi) \prec (\hat H, \hat \varphi)$.
		\item The map $f$ induces an isomorphism from $\hat S/\hat H$ onto $S/H$.
		\item The morphism $\zeta$ is injective when restricted to any edge group of $\hat S$.
		\item Let $\hat v$ be a vertex of $\hat S$ and $v = f(\hat v)$ its image in $S$.
		If $H_v$ is finitely presented (\resp, abelian), then $\zeta$ induces an isomorphism from $\hat H_{\hat v}$ onto $H_v$ (\resp, an embedding from $\hat H_{\hat v}$ into $H_v$).
	\end{enumerate}
\end{defi}

\begin{defi}
	Let $G$ be a finitely generated group.
	Let $(H,\varphi)$ be a group marked by $G$ and $S$ a splitting of $H$ over abelian groups.
	Let  $(H_i, \varphi_i, S_i)$ be a sequence of strong covers of $(H, \varphi, S)$.
	Denote by $\zeta_i \colon H_i \to H$ and $f_i \colon S_i \to S$ the corresponding underlying maps.	
	We say that $(H_i, \varphi_i, S_i)$ is a  \emph{directed system of strong covers} for $(H, \varphi, S)$ if for every $i,j \in \N$ with $i \leq j$, there is a $G$-equivariant map $h_{j,i} \colon S_i \to S_j$ such that
	\begin{equation*}
		h_{k,j}, \circ h_{j,i} = h_{k,i}
		\quad \text{and} \quad
		f_j \circ h_{j,i} = f_i, \quad \forall i,j,k \in \N, \ \text{with } i \leq j\leq k.
	\end{equation*}
\end{defi}

\begin{defi}
\label{def: conv strong cover}
	Let $G$ be a finitely generated group.
	Let $(H,\varphi)$ be a group marked by $G$ and $S$ a splitting of $H$ over abelian groups.
	Let $(H_i, \varphi_i, S_i)$ be a directed systems of strong covers for $(H,\varphi, S)$.
	We write $\zeta_i \colon H_i \to H$, $h_{j,i} \colon S_i \to S_j$, and $f_i \colon S_i \to S$ for the corresponding underlying maps.
	We say that the sequence $(H_i, \varphi_i, S_i)$ \emph{converges to} $(H, \varphi, S)$ if the following hold.
	\begin{itemize}
		\item $(H_i, \varphi_i)$ converges to $(H, \varphi)$ in the space of marked groups.
		\item For every $i \in \N$, for every edge $e$ in $S_i$, we have
		\begin{equation*}
			H_{f_i(e)} = \bigcup_{j \geq i} \zeta_j\left(H_{j, h_{j,i}(e)}\right).
		\end{equation*}
		\item For every $i \in \N$, for every vertex $v$ in $S_i$, we have
		\begin{equation*}
			H_{f_i(v)} = \bigcup_{j \geq i} \zeta_j\left(H_{j, h_{j,i}(v)}\right).
		\end{equation*}
	\end{itemize}
\end{defi}

The next lemma ensures the existence of a converging directed system of \emph{finitely presented} strong covers.
The proof works verbatim as in \cite[Lemma~7.1]{Weidmann:2019ue}.

\begin{lemm}
\label{res: graph of groups cover}
	Let $G$ be a finitely presented group.
	Let $(H,\varphi)$ be a group marked by $G$ and $S$ a splitting of $H$ over abelian groups.
	Then there exists a directed system of strong covers $(H_i, \varphi_i, S_i)$ which converges to $(H, \varphi, S)$ such that for every $i \in \N$, the group $H_i$ is finitely presented.
\end{lemm}

\begin{lemm}
\label{res: lifting modular automorphism}
	Let $G$ be a finitely generated group.
	Let $(H, \varphi)$ be a marked group, where $H$ is torsion-free, one-ended, and CSA.
	Let $S$ be the canonical JSJ tree of $H$.
	Let $(H_i, \varphi_i, S_i)$ be a directed system of strong covers that converges to $(H, \varphi, S)$.
	Denote by $\zeta_i \colon H_i \to H$ the corresponding underlying morphism.
	For every $\alpha \in \mcg{H}$, there is $i_0$ such that for every $i \geq i_0$, we can find $\alpha_i \in \aut{H_i}$ lifting $\alpha$, that is, $\zeta_i \circ \alpha_i = \alpha \circ \zeta_i$.
	Moreover, if $v$ is a vertex of $S_i$ whose image in $S$ is a rigid vertex, then $\alpha_i$ restricted to $H_{i,v}$ coincides with a conjugation in $H_i$.
\end{lemm}

\begin{proof}
	The proof is a variation on \cite[Lemma~7.47]{Coulon:2021wg}.
	It suffices to prove the statement for every element in the generating set of $\mcg H$.
	Inner automorphisms, standard extensions of automorphisms of QH vertex groups and abelian vertex groups are handled as in \cite[Lemma~7.47]{Coulon:2021wg}.
	The only (minor) difference concerns Dehn twists.
	Let $e$ be an edge of $S$.
	Let $u$ be an element centralizing $H_e$ and $\alpha$ the Dehn twist of $H$ over $e$ by $u$.
	Since $S$ is the canonical JSJ tree of $L$, one of the endpoints $v$ of $e$ is such that $H_v$ is the maximal abelian group containing $H_e$.
	In particular, $u$ belongs to $H_v$.
	Let $\tilde u \in G$ be a pre-image of $u$.
	Let $e_i$ be a pre-image of $e$ in some $S_i$ and $v_i$ the endpoint of $e_i$ lifting $v$.
	For every $j \geq i$, we write $v_j$ (\resp, $e_j$) for the image of $v_i$ (\resp, $e_i$) in $S_j$.
	Note that $u_j = \varphi_j(\tilde u)$ belongs to $H_{j, v_j}$ for all but finitely many $j$.
	Since $H_v$ is abelian, so is $H_{j, v_j}$.
	In particular, $u_j$ centralizes the stabilizer of $e_j$.
	We choose for $\alpha_j$ the Dehn twist of $H_j$ over $e_j$ by $u_j$. 
	It is a lift of $\alpha$, provided $j$ is sufficiently large.
	Note that $\alpha_j$ acts by conjugation on every vertex group of $S_j$, in particular, on the ones whose images in $H$ are rigid.
\end{proof}

%%%%%%%%%%%%%%%%%%%%%%%%%%%%%%%%%%%%%%%%%%%%%%%%%%%%%%%%%%%%%%%%%%%%%%%%%%%%%%%%%%%%%
%
\subsection{Definition of limit groups and first properties}
%
%%%%%%%%%%%%%%%%%%%%%%%%%%%%%%%%%%%%%%%%%%%%%%%%%%%%%%%%%%%%%%%%%%%%%%%%%%%%%%%%%%%%%

Let $\Gamma$ be a non-elementary, torsion-free, hyperbolic group and $X$ a Cayley graph of $\Gamma$.
Given $\lambda, \epsilon \in (0,1)$, we denote by $\mathfrak F(\lambda, \epsilon)$ the collection of all tight $C'(\lambda, \epsilon)$ strengthened small cancellation quotients of $\Gamma$ and their subgroups; recall that the small cancellation parameters depend on $X$, which we fixed once and for all. 
Let $G$ be a finitely generated group.
We write $\mathfrak F(G, \lambda, \epsilon)$ for the set of all groups $(H,\varphi)$ marked by $G$, where $H$ belongs to $\mathfrak F(\lambda, \epsilon)$.

\begin{defi}
\label{def: limit group - general}
	A \emph{limit group over the small cancellation quotients of $\Gamma$} is a marked group in the set
	\begin{equation*}
		\mathfrak L(G) = \bigcap_{\lambda, \epsilon \in (0,1)} \overline{\mathfrak F(G, \lambda, \epsilon)}.
	\end{equation*}
	where $ \overline{\mathfrak F(G, \lambda, \epsilon)}$ stands for the closure of $ \mathfrak F(G, \lambda, \epsilon)$ in $\mathfrak G(G)$.
\end{defi}

\begin{rema}
    The torsion-free Tarski monsters $\Gamma_\infty$ that are the object of our main results arise as direct limits of sequences $\Gamma_1 \onto \Gamma_2 \onto \Gamma_3 \onto \cdots$
    where each $\Gamma_i$ is a torsion-free non-elementary hyperbolic group, and each step is a small cancellation quotient of the previous group. 
    However, the $\Gamma_i$ are \emph{not} small cancellation quotients of a single hyperbolic group, and therefore $\Gamma_\infty$ is not a limit group in the sense of \autoref{def: limit group - general}. 
    Instead, limit groups will only appear in proofs by contradiction in the next section.
\end{rema}

\begin{voca*}
	If there is no ambiguity, we will simply refer to the marked groups in $\mathfrak L(G)$ as \emph{limit groups}.
	A sequence $(H_k, \varphi_k)$ of marked groups is an \emph{$\mathfrak F$-sequence} if there are sequences of positive numbers $(\lambda_k)$ and $(\epsilon_k)$ converging to zero, such that $(H_k, \varphi_k)$ belongs to $\mathfrak F(G, \lambda_k, \epsilon_k)$ for every $k \in \N$.
	Note that any subsequence of an $\mathfrak F$-sequence is still an $\mathfrak F$-sequence.
	Concretely, $(L, \eta)$ is a limit group if and only if it is the limit of an $\mathfrak F$-sequence.
\end{voca*}

Let $(H_k, \varphi_k)$ be an $\mathfrak F$-sequence.
According to \autoref{res: recap sc}, for all but finitely many $k \in \N$, the group $H_k$ is torsion-free and CSA.
These properties are closed in the space of marked groups.
Hence:

\begin{lemm}
\label{res: limit group torsion free and CSA}
    Every limit group in $\mathfrak L(G)$ is torsion-free and CSA. \qed
\end{lemm}

\begin{defi}
\label{def: factorization property}
	We say that $(Q, \pi) \in \mathfrak G(G)$ has the \emph{factorization property} if there exist a finite subset $W \subset \ker \pi$ and $\lambda, \epsilon \in (0,1)$, with the following property:
	for every marked group $(H, \varphi) \in \mathfrak F(G, \lambda, \epsilon)$, if $W \subset \ker \varphi$, then $\varphi$ factors through $\pi$.
\end{defi}

\[\begin{tikzcd}
	{W \subset G} & Q \\
	& H
	\arrow["\pi", two heads, from=1-1, to=1-2]
	\arrow["\varphi"', two heads, from=1-1, to=2-2]
	\arrow["{\exists\bar{\varphi}}", dashed, from=1-2, to=2-2]
\end{tikzcd}\]

Note that if $(Q, \pi) \in \mathfrak{G}(G)$ is such that $\ker \pi < G$ is the normal closure of finitely many elements, then $(Q, \pi)$ has the factorization property. In particular, this holds if $Q$ is finitely presented.

The next lemma is a direct consequence of the definitions, see for instance {\cite[Lemma~4.8]{Coulon:2021wg}.

\begin{lemm}
\label{res: factorization prop implies almost open}
	Let $(Q, \pi) \in \mathfrak G(G)$ with the factorization property.
	There exists a finite subset $W \subset \ker \pi$ such that for every limit group $(L, \eta) \in \mathfrak L(G)$, if $W \subset \ker \eta$, then $(L, \eta) \prec (Q, \pi)$.
\end{lemm}

%%%%%%%%%%%%%%%%%%%%%%%%%%%%%%%%%%%%%%%%%%%%%%%%%%%%%%%%%%%%%%%%%%%%%%%%%%%%%%%%%%%%%
%
\subsection{Actions on $\R$-trees}
\label{sec: actions on R-trees}
%
%%%%%%%%%%%%%%%%%%%%%%%%%%%%%%%%%%%%%%%%%%%%%%%%%%%%%%%%%%%%%%%%%%%%%%%%%%%%%%%%%%%%%

Let $\delta \in \R_+^*$.
In this section, we denote by $\mathfrak H(\delta)$ the collection of all pairs $(\Gamma, X)$ where $X$ is a $\delta$-hyperbolic geodesic metric space endowed with a proper action by isometries of a torsion-free group $\Gamma$ satisfying the following properties:
\begin{enumerate}
	\item every non-trivial element of $\Gamma$ is $\delta$-thin (\autoref{def: thin isom});
	\item $A(\Gamma, X) \leq \delta$ (\autoref{def: acylindricity}).
\end{enumerate}
It follows directly from our assumptions that every non-trivial element of $\Gamma$ is loxodromic (for its action on $X$).
Moreover, every non-trivial, elementary subgroup of $\Gamma$ is isomorphic to $\Z$.
Note also that $\Gamma$ is CSA.
Observe that $X$ is not necessarily the Cayley graph of $\Gamma$, though. 
Indeed, in practice, $\Gamma$ will be a (subgroup of a) small cancellation quotient of a hyperbolic group and $X$ the space obtained by the small cancellation theorem (\autoref{res: recap sc}).

The goal of this section is to prove two shortening statements (Corollaries \ref{res: shortening argument - relative} and \ref{res: shortening argument - absolute}).
These two statements are shown in a very similar way.
The next theorem provides a common framework to handle them both.

\begin{theo}
\label{res: shortening argument - general}
	Let $\delta \in \R_+^*$.
	Let $G$ be a finitely generated group and $U$ a finite generating set of $G$.
	Let $H$ be a subgroup of $G$.
    Let $\varphi_k \colon G \to \Gamma_k$ be a sequence of marked groups where $(\Gamma_k, X_k) \in \mathfrak H(\delta)$.
	We make the following assumptions.
	\begin{itemize}
		\item The $\ell^\infty$-energy $E_\infty(\varphi_k, U)$ diverges to infinity.
		\item For every $h \in H$, we have $\lim_{k \to \infty} \norm{\varphi_k(h)} / E_\infty(\varphi_k, U) = 0$.
		\item The sequence $(\Gamma_k, \varphi_k)$ converges to a group $(L, \eta)$.
		\item The group $L$ is freely indecomposable relative to $M = \eta(H)$.
	\end{itemize}
	Let $(\hat L_k, \hat \eta_k)$ be a sequence of covers of $(L, \eta)$ and let $\zeta_k \colon \hat L_k \onto L$ be the corresponding covering map.
	Additionally, we assume that the following properties hold.
	\begin{itemize}
		\item $\varphi_k$ factors through $\hat \eta_k$, and we write $\hat \varphi_k \colon \hat L_k \to \Gamma_k$ for the resulting morphism so that $\varphi_k = \hat \varphi_k \circ \hat \eta_k$.
		\item for every $\alpha \in \mcg{L, M}$, there is $k_0 \in \N$ such that for every $k \geq k_0$, the automorphism $\alpha$ lifts to an automoprhism $\hat \alpha_k \in \aut{\hat L_k}$, i.e., $\zeta_k \circ \hat \alpha_k = \alpha \circ \zeta_k$.
	\end{itemize}
	
	Then there exists $\tau > 0$ with the following property: for infinitely many $k \in \N$, there exist automorphisms $\alpha_k \in \mcg{L, M}$ and $\hat \alpha_k \in \aut{\hat L_k}$ such that $\zeta_k \circ \hat \alpha_k = \alpha_k \circ \zeta_k$ and 
	\begin{equation*}
		E_1\left(\hat \varphi_k \circ \hat \alpha_k \circ \hat \eta_k, U\right) \leq (1- \tau) E_1\left(\varphi_k,U\right).
	\end{equation*}
\end{theo}

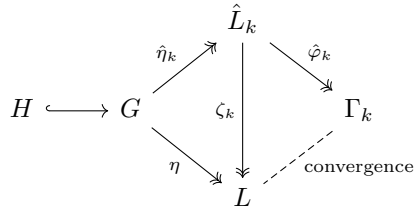
\[\begin{tikzcd}
	&& \hat L_k \\
	H & G && \Gamma_k \\
	&& L
	\arrow["\hat \varphi_k", two heads, from=1-3, to=2-4]
	\arrow["\zeta_k"', two heads, from=1-3, to=3-3]
	\arrow[hook, from=2-1, to=2-2]
	\arrow["\hat \eta_k", two heads, from=2-2, to=1-3]
	\arrow["\eta"', two heads, from=2-2, to=3-3]
	\arrow["\footnotesize{\text{convergence}}", dashed, no head, from=2-4, to=3-3]
\end{tikzcd}\]

In words, we can decrease the $\ell^1$-energy of $\varphi_k$ by a definite amount, up to twisting by an automorphism of $\hat L_k$, which lifts a modular automorphism of $L$.
The strategy for the proof, which is now well understood, goes back to the work of Rips and Sela \cite{Rips:1994jg}.
Similar arguments appear also in \cite{Groves:2004uy,Perin:2008aa,Weidmann:2019ue,Coulon:2021wg}.
We only sketch the main arguments and highlight the differences with the standard proof.

\paragraph{Action on an $\R$-tree.}
Since $L$ is a limit of a sequence of torsion-free, CSA groups, it is also torsion-free and CSA.
The first step is to produce an action of $L$ on a limit $\R$-tree.
To that end, we fix a non-principal ultra-filter $\omega \colon \mathcal P(\N) \to \{0, 1\}$.
Recall that a property $P_k$ holds \oas if 
\begin{equation*}
	\omega(\set{k \in \N}{P_k\ \text{holds}}) = 1.
\end{equation*}
A real valued sequence $(u_k)$ is \oeb if there exists $M \in \R$, such that $\abs{u_k} \leq M$, \oas.
It \emph{converges to $\ell$ along $\omega$}, if for every $\epsilon \in \R_+^*$, we have $\abs{u_k - \ell} < \epsilon$, \oas.
For every $k \in \N$, we choose a base point $o_k \in X_k$ such that 
\begin{equation*}
	\max_{g \in U} \dist{\varphi_k(g)o_k}{o_k} \leq E_\infty(\varphi_k, U) + \delta.
\end{equation*}
For simplicity, we let 
\begin{equation*}
	\epsilon_k = \frac 1 {E_\infty(\varphi_k, U)}, \quad \forall k \in \N.
\end{equation*}
It follows from our assumptions that $(\epsilon_k)$ converges to zero.
In the remainder of this section, unless mentioned otherwise, we work with the rescaled space $\epsilon_k X_k$, i.e., for every $x,x' \in X_k$
\begin{equation*}
	\dist[\epsilon_k X_k] x{x'} = \epsilon_k \dist[X_k] x{x'}.
\end{equation*}
We endow $\epsilon_k X_k$ with the action by isometries of $G$ induced by $\varphi_k$.
This space is $\delta_k$-hyperbolic, where $\delta_k = \epsilon_k \delta$ converges to zero.
We consider now an ultra-limit of metric spaces. 
We refer the reader to Dru\c tu and Kapovich \cite[Chapter~10]{Drutu:2018aa} for a detailed exposition of this construction.
As the sequence $(\delta_k)$ converges to zero, the limit space
\begin{equation*}
	(X_\omega, o) = \limo \left(\epsilon_k X_k, o_k\right)
\end{equation*}
is an $\R$-tree.

\begin{nota}
	Let $(x_k)$ be a sequence of points such that $x_k \in X_k$, for every $k \in \N$.
	If $\dist{o_k}{x_k}$ is \oeb, then we write $\limo x_k$ for the corresponding point of $X_\omega$.
	We let $o = \limo o_k$.
\end{nota}

The action of $G$ on $X_k$ induces an action without global fixed points of $G$ on $X_\omega$ \cite{Paulin:1991fx}.
Moreover, the stable kernel of $(\varphi_k)$ acts trivially on $X_\omega$.
Consequently, the limit group $L$ acts on $X_\omega$ without global fixed points.
We denote by $T$ the minimal $L$-invariant subtree of $X_\omega$.
A standard exercise shows that $o$ belongs to $T$.

\begin{lemm}
\label{res: elt with slowing growing trans. length}
	Let $g \in G$ be such that 
	\begin{equation*}
		\lim_{k \to \infty} \norm{\varphi_k(g)} = 0,
	\end{equation*}
	(where the translation length is computed in $\epsilon_kX_k$).
	Then $\eta(g)$ fixes a point in $T$.
\end{lemm}

\begin{proof}
	We fix a sequence $(r_k)$ converging to zero such that 
	\begin{equation*}
		r_k > \max \{\norm{\varphi_k(g)} , 7 \delta_k\}.
	\end{equation*}
	For every $k \in \N$, we write $x_k$ for a projection of $o_k$ on $\fix{\varphi_k(g), r_k}$.
	It follows from  \autoref{res: set of almost fixed points} that either $o_k$ belongs to $\fix{\varphi_k(g), r_k}$, in which case $x_k = o_k$, or
	\begin{equation*}
		\dist {x_k}{o_k} \leq \frac 12 \dist{o_k}{\varphi_k(g)o_k} + 5\delta_k -\frac{r_k}2.
	\end{equation*}
	In both cases, $\dist {x_k}{o_k}$ is \oeb, thus $x = \limo x_k$ is a well-defined point of $X_\omega$.
	By construction, $\dist{\varphi_k(g)x_k}{x_k}$ converges to zero.
	Thus, $\eta(g)$ fixes $x$.
	Recall that $T$ is an $L$-invariant subtree.
	Therefore, $\eta(g)$ also fixes a point in $T$.
\end{proof}

\begin{lemm}
\label{res: shortening - transverse simplicial arc}
	Let $x = \limo x_k$ and $y = \limo y_k$ be two distinct points of $X_\omega$.
	Let $g \in G$ be such that $\eta(g)$ is a non-trivial element fixing pointwise $[x,y]$.
	There exists a sequence $(g_k)$ such that
	\begin{enumerate}
		\item $g_k \in \langle g \rangle$, and so $\eta(g_k)$ fixes $\geo xy$ pointwise;
		\item $\limo\dist{x_k}{\varphi_k(g_k)y_k} = 0$.
	\end{enumerate}
\end{lemm}

\begin{proof}
	Since $\eta(g)$ is non-trivial, $\varphi_k(g)$ is loxodromic \oas.
	There exists a sequence $(r_k)$ converging to zero such that $x_k, y_k \in \fix{ \varphi_k(g),r_k}$ \oas.
	In particular, $\norm{\varphi_k(g)} \leq r_k$.
	Since $\varphi_k(g)$ is $\delta_k$-thin, $x_k$ and $y_k$ lie in the $(r_k/2 + 50\delta_k)$-neighborhood of any $100\delta_k$-local $(1, \delta_k)$-quasi-geodesic of $\epsilon_kX_k$ joining the attractive and repulsive points of $\varphi_k(g)$.
	Recall that $\varphi_k(g)$ approximately acts on this geodesic as a translation of length $\norm{\varphi_k(g)}$.
	Consequently, there exists $g_k \in \group {g}$ such that
	\begin{equation*}
		\dist{x_k}{\varphi_k\left(g_k\right)y_k} \leq 2 r_k +1000 \delta_k\quad  \oas.
	\end{equation*}
	In particular, $\dist{x_k}{\varphi_k(g_k)y_k}$ converges to zero.
	Since $\eta(g_k)$ belongs to $\group{\eta(g)}$, it fixes $\geo xy$ pointwise, whence the result.
\end{proof}

\begin{lemm}
\label{res: abelian arc stabilizers}
	Let $[x,y]$ be an arc in $X_\omega$.
	Its pointwise stabilizer is abelian.
\end{lemm}

\begin{proof}
	Let $g_1, g_2 \in L$ pointwise fixing $[x,y]$.
	Let $\tilde g_1, \tilde g_2 \in G$ be respective pre-images of $g_1$ and $g_2$.
	By definition, there exists a sequence $(r_k)$ converging to zero such that $x_k,y_k \in \fix{W_k,r_k}$, where $W_k = \{ \varphi_k(\tilde g_1), \varphi_k(\tilde g_2)\}$.
	Recall that $A(\Gamma_k, X_k) \leq \delta_k$.
	Thus, if $W_k$ does not generate an elementary subgroup, then
	\begin{equation*}
		\dist {x_k}{y_k} \leq \diam{\fix{W_k, r_k}} \leq 4r_k + \delta_k.
	\end{equation*}
	Passing to the limit, we get $x = y$, which contradicts our assumption.
	Therefore, the set $W_k$ generates an elementary subgroup \oas.
	Recall that every elementary subgroup of $\Gamma_k$ is abelian.
	Consequently, $\varphi_k(\tilde g_1)$ and $\varphi_k(\tilde g_2)$ commute, hence so do $g_1$ and $g_2$.
\end{proof}

\begin{lemm}
\label{res: tree-graded - transverse tripod}
	Let $x$, $y$, $z$ be three points of $X_\omega$.
	If $x$, $y$, and $z$ do not lie on a geodesic, then the pointwise stabilizer of the tripod $[x,y,z]$ is trivial.
\end{lemm}

\begin{proof}
	We write $x = \limo x_k$, $y = \limo y_k$ and $z = \limo z_k$.
	Assume that there exists $g \in L \setminus\{1\}$ fixing $x$, $y$ and $z$.
	Let $\tilde g \in G$ be a pre-image of $g$.
	As usual, there exists a sequence $(r_k)$ converging to zero such that $x_k$, $y_k$ and $z_k$ belong to $\fix{\varphi_k(\tilde g),r_k}$.
	Since $g$ is non-trivial, $\varphi_k(\tilde g)$ is loxodromic.
	We denote by $\xi_k^-, \xi_k^+ \in \partial X_k$ the repulsive and attractive points of $\varphi_k(\tilde g)$.
	It follows from the thinness of $\varphi_k(\tilde g)$ that for every $t \in \{x_k, y_k, z_k\}$ we have $\gro{\xi_k^-}{\xi_k^+}t \leq r_k/2 + \delta_k$.
	Hence, up to permuting $x$, $y$ and $z$, we observe that
	\begin{equation*}
		\gro{x_k}{z_k}{y_k} \leq 3r_k/2 + 100\delta_k,\ \oas,
	\end{equation*}
	and so $y$ lies on the geodesic $\geo xz$.
\end{proof}

Combining the previous two statements, we get the following.

\begin{prop}
	The action of $L$ on $X_\omega$ is super-stable.	
\end{prop}

\begin{proof}
	It works verbatim as in \cite[Lemma~1.3(vi)]{Sela:2001gb}.
\end{proof}

\begin{lemm}
\label{res: fixed point H}
	The group $M = \eta(H)$ has a global fixed point in $T$.
\end{lemm}

\begin{proof}
	It follows from our assumption that
	\begin{equation*}
		\lim_{k \to \infty} \norm{\varphi_k(h)} = 0, \quad \forall h \in H,
	\end{equation*}
	(where the translation length is computed in $\epsilon_kX_k$).
	Consequently, every element in $M$ fixes a point in $T$ (\autoref{res: elt with slowing growing trans. length}).
	Hence, $M$ is either elliptic (and fixes a point in $T$) or parabolic.
	Since the action of $L$ on $T$ is super-stable, it does not admit parabolic subgroups, whence the result.
\end{proof}

\paragraph{Decomposing the action.}
The next step in the proof is to decompose the action of $G$ on the limit tree $T$ into a graph of actions.
By \autoref{res: limit group torsion free and CSA}, $L$ is CSA.
Since $L$ is also freely indecomposable relative to $M$, the JSJ decomposition of $L$ relative to $M$ exists.

According to the previous discussion, we are in a setting where \autoref{res: splitting relative version} applies (recall that we are assuming that $L$ is freely indecomposable relative to $M$).
Hence, the action of $L$ on $T$ splits as a graph of actions where each vertex action is either simplicial, axial, or of Seifert type.
We write $S$ for its skeleton.
We can now take advantage of this decomposition to reduce the energy of the initial morphisms $\varphi_k$.

Some of the following results are certainly well-known to the experts. 
However, we were not able to find the precise statements that we need in the literature, so we go through the arguments for completeness.

\begin{lemm}
\label{res: standard extension mcg in modular group}
	Let $Y$ be a surface type component of the decomposition of $T$ as a graph of actions.
	Let $\Sigma$ be the underlying surface.
	Let $\alpha$ be an automorphism of $\pi_1(\Sigma)$ represented by a mapping class that pointwise fixes the boundary of $\Sigma$.
	Then its standard extension to $L$ belongs to $\mcg{L, M}$.
\end{lemm}

\begin{proof}
	The mapping class group of $\Sigma$ is generated by Dehn twists around essential simple closed curves that are not parallel to a boundary component.
	Hence, we can assume without loss of generality that $\alpha$ is such a Dehn twist around the curve, say $c$.
	Let $S$ be the underlying skeleton of the decomposition of $T$ as a graph of actions.
	Denote by $v$ the vertex of $S$ corresponding to $Y$.
	We build from $S$ a one-edge splitting $S'$ of $L$ over abelian groups as follows.
	First, we refine $S$ by cutting $\Sigma$ around $c$ (hence the stabilizers of new edges are conjugated to the subgroup generated by $c$), and then we collapse in the resulting tree all edges originally coming from $S$.
	By construction, $\alpha$ is a Dehn twist around an edge of $S'$.
	We are left to prove that $M$ is elliptic in $S'$.
	We distinguish two cases. 
	Recall that $M$ fixes a point in $T$, hence a vertex $w \in S$.
	If $w$ is not in the orbit of $v$, then $M$ automatically fixes a vertex in $S'$.
	Suppose now that $w$ is in the orbit of $v$.
	Up to replacing $M$ by a conjugate, we can assume that $M$ is contained in $L_v$.
	In particular, $M$ fixes a point in $Y$.
	By definition, $Y$ is a tree dual to an arational foliation on $\Sigma$.
	Hence, $M$ is necessarily conjugated into a boundary component of $\Sigma$, thus fixes a vertex in $S'$.
\end{proof}

\begin{prop}
\label{res: shortening - surface case}
	There is $\alpha \in \mcg{L,M}$ such that for every $u \in U$,
	\begin{itemize}
		\item if the geodesic $\geo o{uo}$ has a non-degenerate intersection with a Seifert type component of $\Lambda$, then $\dist{\alpha(u)o}o < \dist {uo}o$,
		\item otherwise $\alpha(u) = u$.
	\end{itemize}
\end{prop}

\begin{proof}
	The proof is given in \cite[Theorem~5.12]{Perin:2008aa}.
	However, Perin does not state explicitly that $\alpha$ belongs to $\mcg{L, M}$.
	We explain here why it is the case.
	Denote by $Y_1, \dots, Y_p$ representatives (for the action of $G$) of the surface type components of the above decomposition of $T$ into a graph of actions.
	By definition, each $Y_i$ is an $\R$-tree dual to an arational foliation on the underlying surfaces $\Sigma_i$.
	In the course of the proof, Perin shows that $\alpha$ is a product $\alpha = \alpha_1 \cdots \alpha_p$ where each $\alpha_i$ is (up to conjugacy) the standard extension to $L$ of an automorphism induced by a mapping class of $\Sigma_i$ fixing pointwise the boundary of $\Sigma_i$, see the proof of \cite[Theorem~5.12]{Perin:2008aa}.
	It follows from \autoref{res: standard extension mcg in modular group} that $\alpha \in \mcg{L, M}$.
\end{proof}

\begin{lemm}
\label{res: standard extension axial in modular group}
	Let $Y$ be an axial component in the decomposition of $T$ as a graph of actions.
	Let $v$ be the corresponding vertex in the underlying skeleton $S$ of the decomposition, and suppose that $L_v$ is abelian.
	Let $\alpha$ be an automorphism of the vertex group $L_v$ that is trivial on the peripheral subgroup $K_v$.
	Its standard extension to $L$ belongs to $\mcg{L, M}$.
\end{lemm}

\begin{proof}
	By definition of the peripheral subgroup $K_v$ (see \eqref{eq - peripheral subgroup}), the abelian vertex group $L_v$ splits as $L_v = K_v \oplus \Z^m$ for some $m \in \N$.
	Recall that the image of $L_v$ in $\isom Y$ is a finitely generated group with dense orbits.
	This image is also abelian like $L_v$, hence it is free abelian, acting by non-zero translations on $Y$.
	It follows that $K_v$ is actually the kernel of $L_v \to \isom Y$. 
	Moreover, $m \geq 2$.

	Let $\{t_1, \dots, t_m\}$ be a free basis of the $\Z^m$ factor.
	We build from $S$ a one-edge splitting $S'$ of $L$ above abelian groups as follows.
	First we replace $L_v$ by the HNN decomposition $A \ast_ C$ where 
	\begin{equation*}
		A = C = K_v \oplus \Z t_1 \oplus \dots \oplus \Z t_{m-1},
	\end{equation*}
	and the attaching maps $C \into A$ are both the identity; 
	then we collapse in the resulting tree all edges originally coming from $S$.
	We claim that any Dehn twist $\alpha$ around an edge of $S'$ belongs to $\mcg{L, M}$.
	To that end, it suffices to prove that $M$ is elliptic in $S'$.
	We distinguish two cases. 
	Recall that $M$ fixes a point in $T$, hence a vertex $w \in S$.
	If $w$ is not in the orbit of $v$, then $M$ automatically fixes a vertex in $S'$.
	Suppose now that $w$ is in the orbit of $v$.
	Up to replacing $M$ by a conjugate, we can assume that $M$ is contained in $L_v$.
	In particular, $M$ fixes a point in $Y$.
	Recall that the image of $L_v$ in $\isom Y$ acts without fixed points on $Y$.
	Hence, $M$ lies in the kernel of $L_v \to \isom Y$, i.e., in $K_v$.
	In particular, $M$ fixes an edge in $S'$, which completes the proof of our claim.
	Every standard extension of a generalized Dehn twist in $L_v$ can be obtained as a product of Dehn twists as above, whence the result.
\end{proof}

\begin{prop}
\label{res: shortening - axial case}
	There is $\alpha \in \mcg{L,M}$ such that for every $u \in U$,
	\begin{itemize}
		\item if the geodesic $\geo o{uo}$ has a non-degenerate intersection with an axial type component of $\Lambda$, then $\dist{\alpha(u)o}o < \dist {uo}o$,
		\item otherwise $\alpha(u) = u$.
	\end{itemize}
\end{prop}

\begin{proof}
	The proof is given in \cite[Theorem~5.17]{Perin:2008aa}.
	Like before, Perin does not state explicitly that $\alpha$ belongs to $\mcg{L, M}$.
	We explain here why it is the case.
	Denote by $Y_1, \dots, Y_p$ representatives (for the action of $G$) of the axial components of the graph of groups decomposition of $T$.
	Perin then shows --- see  \cite[Theorem~5.17]{Perin:2008aa} ---  that $\alpha$ is a product $\alpha = \alpha_1 \cdots \alpha_p$ where each $\alpha_i$ is (up to conjugacy) the standard extension to $L$ of generalized Dehn twist as in \autoref{res: standard extension axial in modular group}.
	Whence the result.
\end{proof}

\begin{prop}
\label{res: shortening morphism - transverse arc}
	Assume that $T$ contains a simplicial arc.
	Then there exists $\ell > 0$ such that for every sufficiently large $k \in \N$, there are $\alpha_k \in \mcg{L,M}$ and $\hat \alpha_k \in \aut{\hat L_k}$ such that $\zeta_k \circ \hat \alpha_k = \alpha_k \circ \zeta_k$ and 
		\begin{equation*}
			E_1\left(\hat \varphi_k \circ \hat \alpha_k \circ \hat \eta_k, U\right) < E_1\left(\varphi_k,U\right) - \frac \ell{\epsilon_k}\ \oas.
		\end{equation*}
\end{prop}

In the statement above, the energies are this time computed in the non-rescaled version of $X_k$.

\begin{proof}
	The strategy in this situation has been detailed many times in the literature.
	We refer the reader to Rips--Sela~\cite[Section~6]{Rips:1994jg}, Perin~\cite[Section~5.6]{Perin:2008aa}, or Weidmann--Reinfeldt~\cite[Section~5.2.3]{Weidmann:2019ue}  for a carefully written proof.
	Let us just highlight its main steps.
	Let $x = \limo x_k$ and $y = \limo y_k$ such that  $\geo xy$ is a simplicial arc of $T$.
	Up to replacing $[x,y]$ by a smaller arc, we always assume that $o$ does not belong to its interior.
	Let $\bar T$ be the simplicial tree obtained by collapsing to a point every connected component in the complement of the $L$-orbit of the open segment $(x, y)$.
	It defines a one-edge splitting of $L$ relative to $M = \eta(H)$.
	In particular, the stabilizer of this edge $e$ --- which coincides with the pointwise stabilizer of $\geo xy$ --- is not trivial, for otherwise $L$ would be freely decomposable relative to $M$.
	Fix $g \in G$ such that $\eta(g)$ is a non-trivial element of $G$ fixing pointwise $\geo xy$.
	Using \autoref{res: shortening - transverse simplicial arc} we build a sequence $(g_k)$ of elements of $\group g$ such that 
	\begin{enumerate}
		\item $\eta(g_k)$ pointwise fixes $\geo xy$ \oas, in particular $\eta(g_k)$ belongs to $L_e$;
		\item $\limo\dist{x_k}{\varphi_k(g_k)y_k} = 0$.
	\end{enumerate}
	Let $\alpha$ be the Dehn twist by $\eta(g)$ over $e$.
	Note that $M$ fixes a vertex in the one-edge splitting $\bar T$ of $L$, hence $\alpha \in \mcg{L, M}$.
	Similarly, we write $\alpha_k$ for the Dehn twist by $\eta(g_k)$ over $e$.
	Since $g_k$ belongs to $\group g$, the automorphism $\alpha_k$ is a power of $\alpha$.
	According to our assumption, there is $k_0$ such that for every $k \geq k_0$, the automorphism $\alpha$ can be lifted to an automorphism $\hat \alpha$ of $\hat L_k$.
	It follows, in particular, that up to forgetting the first few terms of the sequence, $\alpha_k$ lifts to an automorphism $\hat \alpha_k$ of $\hat L_k$, for every $k \in \N$.
	Proceeding as in Perin~\cite[Section~5.6]{Perin:2008aa} one can prove that for every $u \in U$,
	\begin{equation*}
		\limo \dist{\hat\varphi_k \circ \hat \alpha_k \circ \hat  \eta_k(u)o_k}{o_k}
		\leq \dist{\eta(u)o}o.
	\end{equation*}
	However, since $U$ generates $L$, there is at least one element $u \in U$ such that the geodesic $\geo o{uo}$ contains a translated copy of $\geo xy$.
	For such an element, we get
	\begin{equation*}
		\limo \dist{\hat\varphi_k \circ \hat \alpha_k \circ \hat  \eta_k(u)o_k}{o_k}
		\leq \dist{\eta(u)o}o - \dist xy.
	\end{equation*}
	It follows from there that
	\begin{equation*}
		E_1\left(\hat \varphi_k \circ \hat \alpha_k \circ \hat \eta_k, U\right) \leq E_1\left(\varphi_k,U\right) - \frac {\dist xy}{2\epsilon_k}, \quad \oas. \qedhere
	\end{equation*}
\end{proof}

\begin{rema}
	In \cite{Rips:1994jg} or \cite{Weidmann:2019ue}, the authors only consider limit groups over a fixed hyperbolic group.
	The proof of Perin \cite{Perin:2008aa} does not make this assumption. 
	However, it requires a uniform control on $L_e$.
	More precisely, she asks that there exists $g \in G$ in the pre-image of $L_e$ such that $\norm {\varphi_k(g)} > 12\delta_k$ \oas.
	This assumption is used to prove the analogue of our \autoref{res: shortening - transverse simplicial arc}, see \cite[Lemma~5.25]{Perin:2008aa}.
	In our context, given a pair $(\Gamma, X) \in \mathfrak H(\delta)$, we have no control on the injectivity radius of the action of $\Gamma$ on $X$.
	Consequently, such a uniform control does not hold in general.
	However, as we have seen in the proof of \autoref{res: shortening - transverse simplicial arc}, it can be advantageously replaced by the fact that hyperbolic isometries of $\Gamma$ are uniformly thin.
\end{rema}

We can now complete the proof of \autoref{res: shortening argument - general}.
We claim that there is $\ell \in \R_+^*$ such that 
\begin{equation*}
	E_1\left(\hat \varphi_k \circ \hat \alpha_k \circ \hat \eta_k, U\right) \leq E_1\left(\varphi_k,U\right) - \frac\ell{2\epsilon_k}, \quad \oas,
\end{equation*}
where the energies are computed in the non-rescaled version of $X_k$.
As we already explained, it follows from \autoref{res: splitting relative version} that the action of $L$ on $T$ decomposes as a graph of actions whose components are either simplicial, axial, or of surface type.
If the decomposition admits a simplicial component, then the claim directly follows from \autoref{res: shortening morphism - transverse arc}.
Suppose now that the decomposition admits a component $Y$ of axial (\resp, surface) type.
According to \autoref{res: shortening - surface case} (\resp, \autoref{res: shortening - axial case}) there is $\alpha \in \mcg{L, M}$ such that
\begin{equation*}
	\dist o{uo} \leq \dist o{\alpha(u)o}, \quad \forall u \in U
\end{equation*}
with a strict inequality if $[o, uo]$ has a non-degenerate intersection with a component of axial (\resp, surface) type.
Here, the distances are computed in $\epsilon_kX_k$.
According to our assumption there is $\hat \alpha_k \in \aut{\hat L_k}$ lifting $\alpha$, that is such that $\zeta_k \circ \hat \alpha_k = \alpha \circ \zeta_k$ for all but finitely many $k \in \N$.
Unwrapping the definition yields
\begin{equation*}
	\limo \dist{o_k}{\hat \varphi_k \circ \hat \alpha_k \circ \hat \eta_k(u) o_k} \leq \limo \dist {o_k}{\varphi_k (u)o_k} \quad \forall u \in U,
\end{equation*}
with a strict inequality if $[o, uo]$ has a non-degenerate intersection with a component of axial (\resp, surface) type.
Summing these inequalities we get that there is $\ell_0 \in \R_+^*$ such that (in $\epsilon_kX_k$)
\begin{equation*}
	E_1\left(\hat \varphi_k \circ \hat \alpha_k \circ \hat \eta_k, U\right) \leq E_1\left(\varphi_k,U\right) - \frac{\ell_0}{2\epsilon_k}, \quad \oas.
\end{equation*}
This completes the proof of our claim.
\autoref{res: shortening argument - general} now follows from the fact that for all $k \in \N$,  
\begin{equation*}
	\epsilon_k = \frac 1 {E_\infty(\varphi_k, U)}, 
	\quad \text{and} \quad
	E_1(\varphi_k, U) \leq \card U E_\infty(\varphi_k, U). \tag*{\qed}
\end{equation*}

\begin{coro}
\label{res: shortening argument - relative}
	Let $\delta \in \R_+^*$.
	Let $G$ be a finitely generated group and $U$ a finite generating set of $G$.
	Let $H$ be a subgroup of $G$ such that $G$ is torsion-free and freely indecomposable relative to $H$.
	Let $\varphi_k \colon G \to \Gamma_k$ be a sequence of marked groups where $(\Gamma_k, X_k) \in \mathfrak H(\delta)$.
	We make the following assumptions.
	\begin{itemize}
		\item The $\ell^\infty$-energy $E_\infty(\varphi_k, U)$ diverges to infinity.
		\item For every $h \in H$, we have $\lim_{k \to \infty} \norm{\varphi_k(h)} / E_\infty(\varphi_k, U) = 0$.
		\item The sequence $(\Gamma_k, \varphi_k)$ converges to $(G, \id)$.
	\end{itemize}
	Then there exists $\tau > 0$ with the following property: for infinitely many $k \in \N$, there exist automorphisms $\alpha_k \in \mcg{G, H}$ such that 
	\begin{equation*}
		E_1\left(\varphi_k \circ \alpha_k, U\right) \leq (1- \tau) E_1\left(\varphi_k,U\right).
	\end{equation*}
\end{coro}

\begin{proof}
	This is a particular case of \autoref{res: shortening argument - general}, where $(L, \eta) = (G, \id)$ and the approximation sequence $(\hat L_k, \hat \eta_k)$ is constant equal to $(G, \id)$.
\end{proof}

\begin{coro}
\label{res: shortening argument - absolute}
	Let $\delta \in \R_+^*$.
	Let $G$ be a finitely generated group and $U$ a finite generating set of $G$.
	Let $\varphi_k \colon G \to \Gamma_k$ be a sequence of marked groups where $(\Gamma_k, X_k) \in \mathfrak H(\delta)$.
	We make the following assumptions.
	\begin{itemize}
		\item The $\ell^\infty$-energy $E_\infty(\varphi_k, U)$ diverges to infinity.
		\item The sequence $(\Gamma_k, \varphi_k)$ converges to a group $(L, \eta)$.
		\item The group $L$ is freely indecomposable, torsion-free, and CSA.
	\end{itemize}
	Let $S$ be the canonical JSJ tree of $L$.
	Let $(\hat L_k, \hat \eta_k, \hat S_k)$ be a directed sequence of strong covers for $(L, \eta, S)$ converging to $(L, \eta, S)$.
	Let $\zeta_k \colon \hat L_k \onto L$ be the covering map associated to $\hat L_k$.
	Assume that for every $k \in \N$, the morphism $\varphi_k \colon G \to \Gamma_k$ factors through $\hat \eta_k \colon G \to \hat L_k$ and write $\hat \varphi_k \colon \hat L_k \to \Gamma_k$ for the resulting morphism, so that $\varphi_k = \hat \varphi_k \circ \hat \eta_k$.
	
	Then there exists $\tau > 0$ with the following property: for infinitely many $k \in \N$, there exist automorphisms $\alpha_k \in \mcg L$ and $\hat \alpha_k \in \aut{\hat L_k}$ preserving $\hat S_k$ such that $\zeta_k \circ \hat \alpha_k = \alpha_k \circ \zeta_k$ and 
	\begin{equation*}
		E_1\left(\hat \varphi_k \circ \hat \alpha_k \circ \hat \eta_k, U\right) \leq (1- \tau) E_1\left(\varphi_k,U\right).
	\end{equation*}
\end{coro}

\begin{proof}
	Let $S$ be the canonical JSJ tree of $L$.
	According to \autoref{res: lifting modular automorphism}, $(\hat L_k, \hat \eta_k, \hat S_k)$ satisfies the assumptions of \autoref{res: shortening argument - general} with $H = \{ 1 \}$ and the result follows.
\end{proof}

%%%%%%%%%%%%%%%%%%%%%%%%%%%%%%%%%%%%%%%%%%%%%%%%%%%%%%%%%%%%%%%%%%%%%%%%%%%%%%%%%%%%%
%
\subsection{Shortening arguments}
%
%%%%%%%%%%%%%%%%%%%%%%%%%%%%%%%%%%%%%%%%%%%%%%%%%%%%%%%%%%%%%%%%%%%%%%%%%%%%%%%%%%%%%
\label{sec: infinite desc  seq}

\begin{defi}
\label{def: shortenable}
	Let $G$ be a finitely generated group.
	Let $(H_k, \varphi_k) \in \mathfrak G(G)$ be an $\mathfrak F$-sequence converging to a limit group $(L, \eta)$.
	We say that the sequence $(H_k, \varphi_k)$ is \emph{shortenable} if there is an $\mathfrak F$-sequence $(H'_k, \varphi'_k) \in \mathfrak G(G)$ such that, up to passing to a subsequence, the following hold.
	\begin{enumerate}
		\item \label{enu: shortening - sbgp}
			 $H'_k=H_k$, for every $k \in \N$.
		\item \label{enu: shortening - cvg}
			$(H'_k, \varphi'_k)$ converges to a limit group $(L', \eta') \prec (L, \eta)$, i.e., $\ker \eta \subset \ker \eta'$.
		\item \label{enu: shortening - dichotomy}
			If $(L', \eta')$ is \emph{not} a proper quotient of $(L, \eta)$, then every abelian subgroup of $L'$ is finitely generated and all but finitely many $\varphi'_k$ factor through $\eta'$.
		\item \label{enu: shortening - fact}
			If every abelian subgroup of $L'$ is finitely generated, then the same is true for $L$.
			If, in addition, all but finitely many $\varphi'_k$ factor through $\eta'$, then all but finitely many $\varphi_k$ factor through $\eta$.
	\end{enumerate}
	In this context, we say that $(H'_k, \varphi'_k)$ is a \emph{shortening sequence} of $(H_k, \varphi_k)$ and $(L', \eta')$ a \emph{shortening quotient} of $(L, \eta)$.
	
	A limit group $(L, \eta)$ is \emph{shortenable} if every $\mathfrak F$-sequence converging to $(L, \eta)$ is shortenable.
\end{defi}

\begin{rema}
\label{rem: fp group are shortenable}
    If $(L, \eta)$ is a limit group with the factorization property (\autoref{def: factorization property}) all of whose abelian subgroups are finitely generated, then it is automatically shortenable.
    Indeed, if $(H_k, \varphi_k)$ is an $\mathfrak F$-sequence converging to $(L, \eta)$ we can simply take as shortening sequence $(H'_k, \varphi'_k) = (H_k, \varphi_k)$.
    In particular, finitely presented limit groups all of whose abelian subgroups are finitely generated are shortenable.
\end{rema}

We start by recalling some stabilities of the shortenability property.
The first statement follows directly from the definition.

\begin{lemm}
\label{res: shortenable changing marker}
	Let $G$ be a finitely generated group.
	Let $W \subset G$ be a finite subset and $\pi \colon G \onto Q$ the projection onto $Q = G / \normal W$.
	Let $(H_k, \varphi_k) \in \mathfrak G(Q)$ be an $\mathfrak F$-sequence converging to a limit group $(L, \eta)$.
	Then $(H_k, \varphi_k)$ is shortenable, as a sequence in $\mathfrak G(Q)$, if and only if $(H_k, \varphi_k \circ \pi)$ is shortenable as a sequence in $\mathfrak G(G)$.
	In particular, $(L, \eta)$ is shortenable if and only if $(L, \eta \circ\pi)$ is. \qed
\end{lemm}

\begin{prop}[{\cite[Proposition~4.14]{Coulon:2021wg}}]
\label{res: shortenable stable under free product}
	Let $G$ be a finitely generated group that splits as a free product $G = G_1 \ast G_2$.
	Let $(L, \eta) \in \mathfrak G(G)$ be a limit group that decomposes as $L = L_1 \ast L_2$ so that $\eta$ maps $G_1$ and $G_2$ onto $L_1$ and $L_2$, respectively.
	We denote by $\eta_i \colon G_i \onto L_i$ the restriction of $\eta$ to $G_i$.
	If for every $i \in \{1, 2\}$, the marked group $(L_i, \eta_i)$ is shortenable --- viewed as an element of $\mathfrak G(G_i)$ --- then so is $(L, \eta)$.
\end{prop}

\begin{prop}
\label{res: shortening group - freely indecomposable}
	Let $G$ be a finitely generated group.
	Let $(L, \eta) \in \mathfrak L(G)$ be a limit group.
	If $L$ is freely indecomposable, then $(L, \eta)$ is shortenable.
\end{prop}

\begin{proof}
	The constant $\delta \in \R^*_+$ and the maps $a, \rho \colon (0, 1) \times (0,1) \to \R^*_+$ are the ones given by \autoref{res: recap sc}.
	We fix once and for all a finite generating set $U$ of $G$.
	We denote by $\xi \in (0, 1)$ the parameter given by  \autoref{res: recap sc}\ref{enu: recap sc - lift}.
	Fix two sequences $(\lambda_k)$ and $(\epsilon_k)$ of positive numbers converging to zero.
	Choose a sequence of morphisms $\bar \varphi_k \colon G \to \bar \Gamma_k$, where $\pi_k \colon \Gamma \onto \bar \Gamma_k$ is a non-elementary, tight $C'(\lambda_k, \epsilon_k)$ strengthened small cancellation quotient of $\Gamma$, and such that $(\bar \Gamma_k, \bar \varphi_k)$ converges to $(L, \eta)$.
	For every $k \in \N$, we write $\mathcal Q_k$ for the underlying relation family and $X_k$ for the $\delta$-hyperbolic space associated to $\bar \Gamma_k$ by \autoref{res: recap sc}.
	
	We denote by $S$ the canonical JSJ tree of $L$ (over the abelian subgroups of $L$).
	In addition, we fix a directed system $(\hat L_k, \hat \eta_k, \hat S_k)$ of finitely presented strong covers of $(L, \eta, S)$ converging to $(L, \eta, S)$ (which exists by \autoref{res: graph of groups cover}); let $\zeta_k \colon \hat L_k \to L$ denote the associated epimorphisms.
    Since $\hat L_k$ is finitely presented, for large enough $k'$, $\bar\varphi_{k'}$ factors through $\hat\eta_k$. Thus, by replacing $(\bar\varphi_k)$ with a subsequence (while leaving $\hat\eta_k$ unchanged), we may assume that $\bar\varphi_k$ factors through $\hat\eta_k$ for every $k$.
	We write $\hat  \varphi_k \colon \hat L_k \to \bar \Gamma_k$ for the resulting morphism.
	Let $\mathcal M_k$ be the set of all automorphisms $\hat \alpha_k \in \aut{\hat L_k}$ that descend to a modular automorphism $\alpha \in \mcg L$, i.e., such that $\zeta_k \circ \hat \alpha_k = \alpha \circ \zeta_k$.

\[\begin{tikzcd}
	& {\hat{L}_k} & {\hat{L}_k} \\
	G &&& {\bar{\Gamma}_k} \\
	& L & L
	\arrow["{\hat{\alpha}_k}", hook, two heads, from=1-2, to=1-3]
	\arrow["{\zeta_k}", two heads, from=1-2, to=3-2]
	\arrow["{\hat{\varphi}_k}", two heads, from=1-3, to=2-4]
	\arrow["{\zeta_k}"', two heads, from=1-3, to=3-3]
	\arrow["{\hat{\eta}_k}", two heads, from=2-1, to=1-2]
	\arrow["\eta"', two heads, from=2-1, to=3-2]
	\arrow["{\text{convergence}}", dashed, no head, from=2-4, to=3-3]
	\arrow["\alpha"', hook, two heads, from=3-2, to=3-3]
\end{tikzcd}\]
    
	In addition, we let 
	\begin{equation*}
		\mathfrak C_k = \set{ \hat  \varphi_k \circ \hat \alpha_k \circ \hat \eta_k}{ \hat \alpha_k \in \mathcal M_k}.
	\end{equation*}
	Let $\sigma \in \R_+^*$.
	For every $k \in \N$, we choose $\hat \alpha_k \in \mathcal M_k$ such that the morphism $\bar \varphi'_k  =  \hat \varphi_k \circ \hat \alpha_k \circ \hat \eta_k$ is $\sigma$-short among $\mathfrak C_k$, that is
	\begin{equation*}
		E_1(\bar \varphi'_k, U) \leq \sup_{\psi_k \in \mathfrak C_k} E_1( \psi_k, U)  + \sigma.
	\end{equation*}
	(Recall that the energies are computed in $X_k$.)
	Up to passing to a subsequence, $(\bar \Gamma_k, \bar \varphi'_k)$ converges to a limit group that we denote by $(L', \eta')$.
	We are going to prove that $(\bar \Gamma_k, \bar \varphi'_k)$ is a shortening sequence of $(\bar \Gamma_k, \bar \varphi_k)$.
	
	Recall that $\hat \eta_k$ is onto, and $\hat \alpha_k$ is an automorphism of $\hat L_k$.
	Consequently, $\bar \varphi_k$ and $\bar \varphi'_k$ have the same image for every $k \in \N$, proving Point~\ref{enu: shortening - sbgp} of \autoref{def: shortenable}.
	By construction, $(L', \eta')$ is a quotient of $(L, \eta)$, which corresponds to Point~\ref{enu: shortening - cvg}.
	We now focus on Point~\ref{enu: shortening - dichotomy}.
	Assume that $(L', \eta') = (L, \eta)$.
	
	\begin{clai}
		There exists $E \in \R_+$ such that we have $E_\infty(\bar \varphi'_k, U) \leq E$.
	\end{clai}

	\begin{proof}[Proof of the claim]
		Assume that our claim fails, that is, up to passing to a subsequence, $E_\infty(\bar \varphi'_k, U)$ diverges to infinity.
		Up to passing to a subsequence, it follows then from \autoref{res: shortening argument - absolute} that there is $\tau \in (0,1)$ as well as automorphisms $\beta_k$ in $\mcg L$ and $\hat \beta_k \in \aut{\hat L_k}$ such that for infinitely many $k \in \N$, we have $\beta_k \circ \zeta_k = \zeta_k \circ \hat \beta_k$ and 
		\begin{equation*}
			E_1(\hat \varphi_k \circ \hat \alpha_k \circ \hat \beta_k \circ \hat \eta_k, U) \leq (1-\tau)E_1(\bar \varphi'_k, U). 
		\end{equation*}
		In particular, $\hat \alpha_k \circ \hat  \beta_k$ belongs to $\mathcal M_k$.
		Since $\bar \varphi'_k$ is $\sigma$-short among $\mathfrak C_k$ we have
		\begin{equation*}
			E_1(\bar \varphi'_k, U) \leq E_1(\hat \varphi_k \circ  \hat \alpha_k \circ \hat \beta_k \circ \hat \eta_k, U) + \sigma \leq (1-\tau)E_1(\bar \varphi'_k, U)  + \sigma.
		\end{equation*}
		Recall that $E_\infty(\bar \varphi_k, U)$, hence $E_1(\bar \varphi_k, U)$, diverges to infinity.
		This leads to a contradiction if $k$ is sufficiently large, and completes the proof of our claim.
	\end{proof}
	
	In particular, since $\rho(\lambda_k, \epsilon_k) \to \infty$, we deduce that $E_\infty(\bar \varphi'_k, U) < \xi \rho(\lambda_k, \epsilon_k)$ for all but finitely many $k$.
	It follows from \autoref{res: recap sc}\ref{enu: recap sc - lift} (applied with $H = \{1\}$), that there are lifts $\varphi'_k \colon G \to \Gamma$ whose $\ell^\infty$-energy $E_\infty(\varphi'_k, U)$ measured in $X$ satisfies
	\begin{equation*}
		E_\infty(\varphi'_k, U) \leq   \frac {\pi \sinh E }{a(\lambda_k, \epsilon_k)}T(\mathcal Q_k, X).
	\end{equation*}
	However, for every $x \in X$ the projection $\pi_k \colon \Gamma \onto \bar \Gamma_k$ is injective on the set 
	\begin{equation*}
		\set{ \gamma \in \Gamma}{ \dist {\gamma x}x < \left[ 1 - \frac {2\pi \sinh \rho(\lambda_k, \epsilon_k)}{a(\lambda_k, \epsilon_k)} \right]T(\mathcal Q_k, X)},
	\end{equation*}
	by \autoref{res: recap sc}\ref{enu: recap sc - injectivity}.
	Recall that $a$ (\resp, $\pi \sinh \rho / a$) diverges to infinity (\resp, converges to zero) as $(\lambda, \epsilon)$ approaches $(0,0)$.
	It follows that $(\Gamma, \varphi'_k)$ converges to $(L, \eta)$.
	In particular, $L$ is a limit group over a \emph{fixed} hyperbolic group.
	Hence, by \cite{positivehyp1}, every abelian subgroup of $L$ is finitely generated and $L$ has the factorization property. Finally, up to passing to a subsequence, $\varphi'_k$ factors through $\eta$, and hence so does $\bar \varphi'_k = \pi_k \circ \varphi'_k$, .
	This completes the proof of Point~\ref{enu: shortening - dichotomy} of \autoref{def: shortenable}.
	
	\medskip
	The proof of Point~\ref{enu: shortening - fact} in \autoref{def: shortenable} follows the standard strategy, see for instance \cite{Sela:2001gb, positivehyp1, Weidmann:2019ue}.
	We quickly go through the main arguments.
	Assume that every abelian subgroup of $L'$ is finitely generated.
	Given $k \in \N$, we partition the vertices of $\hat S_k$ as follows: a vertex $\hat v$ belongs to $V_k^0$ (\resp, $V_k^1$) if the stabilizer of its image $v$ in $S$ is a surface group or an abelian group (\resp, a rigid group).
	If $\hat v$ belongs to $V_k^1$, then we can always assume that $\hat \alpha_k$ acts on the vertex group $\hat L_{k, \hat v}$ by conjugation (see \autoref{res: lifting modular automorphism}).
	It follows that the projection $L \onto L'$ is injective when restricted to every rigid vertex group of $L$.

	We now claim that the edge groups of $S$ are finitely generated.
	Consider an edge $e$ of $S$.
	Since $L$ is CSA, it cannot connect two vertices whose stabilizers are both abelian.
	Consequently, $L_e$ is either contained in a surface group or a rigid group.
	If $L_e$ is contained in a surface subgroup, then $L_e$ is cyclic.
	If $L_e$ is contained in a rigid vertex group, then, according to our previous discussion, $L_e$ embeds in $L'$.
	It follows from our assumption that $L_e$ is finitely generated, which completes the proof of our claim.

	We now show that every abelian subgroup $A$ of $L$ is finitely generated.
	Without loss of generality, we can assume that $A$ is not cyclic.
	Since the action of $L$ on $S$ is acylindrical, $A$ fixes a vertex $v$.
	This vertex cannot be a surface group, otherwise $A$ would be cyclic.
	Assume that the stabilizer of $v$ is abelian.
	Since $L$ is finitely generated, $L_v$ is finitely generated relative to the stabilizers of the edges in the link of $v$.
	We just proved that edge groups are finitely generated.
	Consequently, $L_v$ and thus $A$ are finitely generated.
	Suppose now that the stabilizer of $v$ is rigid.
	As previously, $A$ embeds in $L'$, and thus is finitely generated.
	
	We proved that every edge group and every abelian vertex group of $S$ is finitely generated.
	Consequently, if $k$ is sufficiently large, the strong cover $(\hat L_k, \hat \eta_k, \hat S_k)$ satisfies the following additional properties.
	\begin{itemize}
		\item Let $\hat v$ be a vertex of $V_k^0$ and $v$ its image in $S$.
		The map $\zeta_k \colon \hat L_k \onto L$ induces an isomorphism from $\hat L_{k, \hat v}$ onto $L_v$.
		\item Let $\hat e$ be a an edge of $\hat S_k$ and $e$ its image in $S$.
		The map $\zeta_k \colon \hat L_k \onto L$ induces an isomorphism from $\hat L_{k, \hat e}$ onto $L_e$.
	\end{itemize}
	Hence, the kernel of $\zeta_k$ is generated (as a normal subgroup) by a subset $U_k$ with the following property: for every $g \in U_k$, there exists a vertex $\hat v$ of $V_k^1$ such that $g$ belongs to $\hat L_{k, \hat v}$.	
	
	Suppose now that all but finitely many $\bar \varphi'_k$ factor through $\eta'$.
	Since $(L', \eta') \prec (L, \eta)$ those morphisms also factor through $\eta$.	
	Consequently, if $k$ is sufficiently large, $U_k$ is contained in $\ker \hat \varphi_k \circ \hat \alpha_k$.
	As we observed before, for every $\hat v \in V_k^1$, the automorphism $\hat \alpha_k$ acts on $\hat L_{k, \hat v}$ by conjugation.
	Consequently, $U_k$ is also contained in $\ker \hat \varphi_k$.
	However, $U_k$ generates the kernel of $\zeta_k$ (as a normal subgroup).
	It follows that $\hat \varphi_k$ factors through $\zeta_k$.
	Thus, $\bar \varphi_k$ factors through $\eta$.
	This completes the proof of Point~\ref{enu: shortening - fact} of \autoref{def: shortenable}, and hence the proof of \autoref{res: shortening group - freely indecomposable}.
\end{proof}

The next statement is a variation on \cite[Corollary~8.12]{Coulon:2021wg}.

\begin{coro}
\label{res: shortening group - general}
	Let $F$ be a finitely presented group.
	Every limit group $(L, \eta) \in \mathfrak L(F)$ is shortenable.
\end{coro}

Note that, unlike in other results, we are assuming that the marker is finitely presented. We will only apply this to a free group in \autoref{res: dcc}, hence our notation ($F$ instead of the usual $G$).

\begin{proof}
	In view of \autoref{res: shortening group - freely indecomposable}, we can assume that $L$ is freely decomposable.
	Consider the Grushko decomposition
	\begin{equation*}
		L = F_0 \ast L_1 \ast L_2 \ast \dots \ast L_m,
	\end{equation*}
	where $F_0$ is a finitely generated free group and $L_i$ is a non-cyclic, freely indecomposable subgroup.
	We fix a finitely presented cover $(\hat L, \hat\eta)$ of $(L, \eta)$ with the following properties:
	\begin{itemize}
		\item $\hat L$ splits as $\hat L =  \hat F_0 \ast \hat L_1 \ast \dots \ast \hat L_m$,
		\item the morphism $\zeta \colon \hat L \onto L$ maps each $\hat L_i$ onto $L_i$.
		Moreover, it induces an isomorphism from $\hat F_0$ onto $F_0$.
	\end{itemize}
	For every $i \in \intvald 1m$, we write $\zeta_i \colon \hat L_i \onto L_i$ for the restriction of $\zeta$ to $\hat L_i$.
	Similarly, we write $\zeta_0 \colon \hat F_0 \onto F_0$ for the isomorphism induced by $\zeta$.
    Since $\hat L$ is finitely presented, the marked groups converging to $L$ eventually factor through $\hat L$; restricting the corresponding maps, we obtain that $(L_i, \zeta_i) \in \mathfrak G(\hat L_i)$ and $(F_0, \zeta_0) \in \mathfrak G(\hat F_0)$ are limit groups.
	By \autoref{res: shortening group - freely indecomposable}, the group $(L_i, \zeta_i)$ is shortenable.
	Since $F_0$ is free, it has the factorization property, hence $(F_0, \zeta_0)$ is shortenable as well (see \autoref{rem: fp group are shortenable}).	
	According to \autoref{res: shortenable stable under free product}, the limit group $(L, \zeta) \in \mathfrak G(\hat L)$ is shortenable.
	It follows then from \autoref{res: shortenable changing marker} that $(L, \eta) \in \mathfrak G(F)$ is shortenable.
\end{proof}

%%%%%%%%%%%%%%%%%%%%%%%%%%%%%%%%%%%%%%%%%%%%%%%%%%%%%%%%%%%%%%%%%%%%%%%%%%%%%%%%%%%%%
%
\subsection{Infinite descending sequences}
%
%%%%%%%%%%%%%%%%%%%%%%%%%%%%%%%%%%%%%%%%%%%%%%%%%%%%%%%%%%%%%%%%%%%%%%%%%%%%%%%%%%%%%
\label{sec: infinite descending sequence}

Given $(L, \eta) \in \mathfrak G(G)$, an \emph{infinite descending sequence starting at $(L, \eta)$} is a sequence of marked groups
\begin{equation*}
	(L, \eta) = (L_0, \eta_0) \succ (L_1, \eta_1) \succ (L_2, \eta_2) \succ \cdots
\end{equation*}
such that $(L_{i+1}, \eta_{i+1})$ is a proper quotient of $(L_i, \eta_i)$, for every $i \in \N$.

\begin{prop}
\label{res: dcc}
	Let $G$ be a finitely generated group.
	Let $(L, \eta) \in \mathfrak L(G)$ be a limit group.
	Then there is no infinite descending sequence of limit groups starting at $(L, \eta)$.
	Moreover, $(L, \eta)$ has the factorization property.
\end{prop}

\begin{proof}
	Let $F$ be a finitely generated free group with a projection $\pi \colon F \onto G$ so that $\pi$ induces a natural continuous embedding $\mathfrak G(G) \into \mathfrak G(F)$ sending $\mathfrak L(G)$ into $\mathfrak L(F)$.
	According to \autoref{res: shortening group - general}, every limit group $(L,\eta) \in \mathfrak L(F)$ is shortenable.
	It follows then from \cite[Proposition~4.17]{Coulon:2021wg} that there is no infinite descending sequence of limit groups in $\mathfrak L(F)$ starting at $(L, \eta)$.
	In particular, \cite[Proposition~4.15]{Coulon:2021wg} yields that any limit group $(L, \eta) \in \mathfrak L(F)$ has the factorization property.
	The property now follows for limit groups in $\mathfrak L(G)$.
\end{proof}

%%%%%%%%%%%%%%%%%%%%%%%%%%%%%%%%%%%%%%%%%%%%%%%%%%%%%%%%%%%%%%%%%%%%%%%%%%%%%%%%%%%%%
%%%%%%%%%%%%%%%%%%%%%%%%%%%%%%%%%%%%%%%%%%%%%%%%%%%%%%%%%%%%%%%%%%%%%%%%%%%%%%%%%%%%%
%
\section{Lifting morphisms}
\label{sec: lifts}

%
%%%%%%%%%%%%%%%%%%%%%%%%%%%%%%%%%%%%%%%%%%%%%%%%%%%%%%%%%%%%%%%%%%%%%%%%%%%%%%%%%%%%%
%%%%%%%%%%%%%%%%%%%%%%%%%%%%%%%%%%%%%%%%%%%%%%%%%%%%%%%%%%%%%%%%%%%%%%%%%%%%%%%%%%%%%

In this section, we complete the proof of \autoref{intro:thm:lifting}. This will rely on the version for morphisms with ``small energy'' from Section \ref{sec: sc} and the tools we collected in Sections \ref{sec: graph of groups} and \ref{sec: limit}.

\begin{nota}
	Let $G$ be a group and $H$ a subgroup of $G$.
	We denote by $\aut[H]G$ the set of all automorphisms of $G$ whose restriction to $H$ is the identity.
\end{nota}

%%%%%%%%%%%%%%%%%%%%%%%%%%%%%%%%%%%%%%%%%%%%%%%%%%%%%%%%%%%%%%%%%%%%%%%%%%%%%%%%%%%%%
%
\subsection{Morphisms to a fixed hyperbolic group}
%
%%%%%%%%%%%%%%%%%%%%%%%%%%%%%%%%%%%%%%%%%%%%%%%%%%%%%%%%%%%%%%%%%%%%%%%%%%%%%%%%%%%%%

The goal of this section is to describe certain morphisms from a finitely generated group $G$ to a hyperbolic group $\Gamma$ whose restriction to a subgroup $H \subset G$ is prescribed.
Such results are certainly known to the specialists.
A similar study occurs, for instance, in Perin \cite[Proposition~4.13]{Perin:2011aa} where she assumes that $G$ is hyperbolic and $H$ is a non-abelian subgroup.
For completeness, we recall the main arguments. 
It can be seen as a warm-up for the rest of this section.

\begin{prop}
\label{res: morphism to hyp - first step}
	Let $G$ be a group generated by a finite set $U$.
	Let $H$ be a subgroup of $G$ such that $G$ is freely indecomposable relative to $H$.
	Let $\Gamma$ be a torsion-free group acting properly and co-compactly on a geodesic hyperbolic space $X$.
	Let $\iota \colon H \to \Gamma$ be a morphism.
	There exist $E \in \R_+$, as well as two finite subsets $W \subset G \setminus \{1\}$ and $U_0 \subset H$ with the following properties.
	For every morphism $\varphi \colon G \to \Gamma$, if $\varphi$ and $\iota$ coincide on $H_0 = \group {U_0}$, then there is $\alpha \in \aut[H]G$ such that one of the following holds:
		\begin{itemize}
		\item $W \cap \ker(\varphi \circ \alpha) \neq \emptyset$,
		\item $E_\infty(\varphi \circ \alpha, U) \leq E$ (where the energy is computed in $X$).
	\end{itemize}
\end{prop}

\begin{proof}
	Without loss of generality, we can assume that $\iota$ is injective, for otherwise it suffices to take for $W = U_0$ any finite subset of $\ker \iota \setminus \{1\}$ and set $\alpha$ to be the identity.
	We now fix an exhaustion $(W_k)$ of $G \setminus \{1\}$ by finite subsets and an exhaustion $(U_k)$ of $H$ by finite subsets.
	Suppose that the statement is false.
	For every $k \in \N$, we can find a morphism $\varphi_k \colon G \to \Gamma$ such that 
	\begin{itemize}
		\item $\varphi_k$ and $\iota$ coincide on $H_k = \group {U_k}$, 
		\item for every $\alpha \in \aut[H]G$, we have $W_k \cap \ker(\varphi_k \circ \alpha) = \emptyset$ and 
		\begin{equation}
		\label{eqn: morphism to hyp - first step}
			E_\infty(\varphi_k \circ \alpha, U) \geq k.
		\end{equation}
	\end{itemize}
	Let $\sigma \in \R_+^*$.
	Up to pre-composing $\varphi_k$ by some element of $\aut[H]G$, we can assume without loss of generality that $(\Gamma,\varphi_k)$ is $\sigma$-short among the collection
	\begin{equation*}
		\mathfrak C_k = \set{(\Gamma, \varphi_k \circ \alpha)}{ \alpha \in \aut[H]G},
	\end{equation*}
	that is,
	\begin{equation*}
		E_1(\varphi_k, U) \leq \sup_{\alpha \in \aut[H]G} E_1(\varphi_k \circ \alpha, U)  + \sigma.
	\end{equation*}
	The sequence $(\Gamma, \varphi_k)$ converges to $(G, \id)$.
	Moreover, $E_\infty(\varphi_k, U)$, and thus $E_1(\varphi_k, U)$, diverges to infinity.
	We now claim that
	\begin{equation*}
		\lim_{k \to \infty} \frac{\norm{\varphi_k (h)}}{E_\infty(\varphi_k, U)} = 0, \quad \forall h \in H.
	\end{equation*}
	Consider indeed some element $h \in H$.
	Since $(U_k)$ is an exhaustion of $H$, the element $h$ belongs to $U_k$ for all but finitely many $k \in \N$.
	However, $\varphi_k$ and $\iota$ coincide on $U_k$.
	Consequently, for all but finitely many $k \in \N$, we get $\norm{\varphi_k(h)} = \norm{\iota(h)}$, which does not depend on $k$.
	The claim follows from the fact that $E_\infty(\varphi_k, U)$ diverges to infinity.
	We are now in a position to apply  \autoref{res: shortening argument - relative}.
	In particular, one can shorten $\varphi_k$ by pre-composition with an element of $\aut G$ which acts by conjugation on $H$.
	However, the energy of a morphism is invariant under conjugation (\autoref{rem: energy invariant by inner automorphisms}).
	Hence, this shortening can be obtained by pre-composition with an element of $\aut[H]G$.
	This contradicts our choice of $(\varphi_k)$ and completes the proof.
\end{proof}

\begin{coro}
\label{res: morphism to hyp}
	Let $G$ be a finitely generated group.
	Let $H$ be a subgroup of $G$ such that $G$ is freely indecomposable relative to $H$.
	Let $\Gamma$ be a torsion-free hyperbolic group.
	Let $\iota \colon H \to \Gamma$ be a morphism.
	Then there exists a finite subset $W \subset G \setminus \{1\}$ and a finitely generated subgroup $H' \subset H$ with the following properties.
	For every morphism $\varphi \colon G \to \Gamma$ which coincides with $\iota$ on $H'$, one of the following holds:
		\begin{itemize}
		\item there is an automorphism $\alpha \in \aut[H]G$ such that $W \cap \ker(\varphi \circ \alpha) \neq \emptyset$,
		\item $\varphi$ is a monomorphism and its restriction to $H$ coincides with $\iota$.
	\end{itemize}
\end{coro}

\begin{proof}
	As previously, we can assume without loss of generality that $\iota$ is injective.
	We fix once and for all a finite generating set $U$ of $G$.
	We denote by $X$ a Cayley graph of $\Gamma$.
	Let $E \in \R_+$, $W \subset G \setminus \{1\}$ and $U_0 \subset H$ be the data given by \autoref{res: morphism to hyp - first step}.
	Since the action of $\Gamma$ on $X$ is proper and co-compact, there are, up to conjugacy, only finitely many morphisms, say $\psi_1, \dots, \psi_m \colon G \to \Gamma$ such that $E_\infty(\psi_i, U) \leq E$.
	Up to enlarging the set $W$, we can assume that $W \cap \ker \psi_j = \emptyset$ if and only if $\psi_j$ is injective.
	
	Let us now define $H'$.
	If $H$ is already finitely generated, we simply set $H' = H$.
	Suppose now that $H$ is not finitely generated.
	For every $\gamma \in \Gamma$ we write $\tau_\gamma$ for the conjugation in $\Gamma$ by $\gamma$.
	Recall that $\iota \colon H \to \Gamma$ is an embedding.
	Since $H$ is not finitely generated, $\iota(H)$ is a non-elementary subgroup of $\Gamma$.
	We first fix a finitely generated subgroup $H_1$ of $H$ containing $H_0 = \group{U_0}$ such that $\iota(H_1)$ is non-elementary.
	In particular, the centralizer of $\iota(H_1)$ in $\Gamma$ is trivial.
	Therefore, for every $j \in \intvald 1m$, there is at most one element $\gamma_j \in \Gamma$ such that $\tau_{\gamma_j} \circ \psi_j = \iota$ on $H_1$.
	We now choose for $H'$ a finitely generated subgroup of $H$ containing $H_1$ such that for every $j \in \intvald 1m$, the morphisms $\tau_{\gamma_j} \circ \psi_j$ and $\iota$ coincide on $H'$ if and only if they coincide on $H$.
	
	Consider now a morphism $\varphi \colon G \to \Gamma$ which coincides with $\iota$ on $H'$ and such that $W \cap \ker(\varphi \circ \alpha) = \emptyset$, for every $\alpha \in \aut[H]G$.
	In particular, $\varphi$ coincides with $\iota$ on $H_0 = \group {U_0}$.
	It follows from \autoref{res: morphism to hyp - first step} that there is $\alpha \in \aut[H]G$ such that $E_\infty(\varphi \circ \alpha, U) \leq E$.
	Hence, there is $j \in \intvald 1m$ and $\gamma \in \Gamma$ such that $\varphi \circ \alpha = \tau_\gamma \circ \psi_j$.
	Observe that $\ker (\varphi \circ \alpha) = \ker \psi_j$.
	According to our choice of $W$, the morphism $\psi_j$ is injective, and thus so is $\varphi$.
	By construction, $\varphi \circ \alpha$ coincides with $\iota$ on $H'$.
	If $H$ is finitely generated, we chose $H' = H$, hence $\varphi \circ \alpha$ and $\iota$ already coincide on $H$.
	Suppose now that $H$ is not finitely generated.
	Since $\varphi \circ \alpha$ and $\iota$ coincide on $H_1$, we necessarily have $\gamma = \gamma_j$.
	It follows now from our choice of $H'$ that $\tau_{\gamma_j} \circ \psi_j$ coincides with $\iota$ on $H$, and hence so does $\varphi \circ \alpha$.
\end{proof}

%%%%%%%%%%%%%%%%%%%%%%%%%%%%%%%%%%%%%%%%%%%%%%%%%%%%%%%%%%%%%%%%%%%%%%%%%%%%%%%%%%%%%
%
\subsection{Morphisms to a small cancellation quotient}
%
%%%%%%%%%%%%%%%%%%%%%%%%%%%%%%%%%%%%%%%%%%%%%%%%%%%%%%%%%%%%%%%%%%%%%%%%%%%%%%%%%%%%%

\begin{theo}
\label{res: lifting morphism - killing element for hyperbolic groups - one-ended}
	Let $G$ be a finitely generated group.
	Let $H$ be a subgroup of $G$ such that $G$ is freely indecomposable relative to $H$.
	Let $\iota \colon H \to \Gamma$ be a morphism.
	There exist a finite subset $W \subset G \setminus \{1\}$ and parameters $\lambda, \epsilon \in (0,1)$, with the following properties.
	Let $\pi \colon \Gamma \onto \bar \Gamma$ be a tight $C'(\lambda, \epsilon)$ strengthened small cancellation quotient of $\Gamma$.
	Let $\bar \varphi \colon G \to \bar \Gamma$ be a morphism such that $\bar \varphi$ restricted to $H$ coincides with $\pi \circ \iota$.
	Then one of the following holds:
	\begin{itemize}
		\item there is an automorphism $\alpha \in \aut[H]G$ such that $W \cap \ker(\bar \varphi \circ \alpha) \neq \emptyset$,
		\item there is a monomorphism $\varphi \colon G \to \Gamma$ lifting $\bar \varphi$ whose restriction to $H$ coincides with $\iota$.
	\end{itemize}
\end{theo}

\begin{proof}
	Without loss of generality, we can assume that $\iota$ is injective, for otherwise we can choose for $W$ any finite subset in $\ker \iota \setminus\{1\}$.
	Let us introduce first some useful auxiliary objects.	
	We write $W \subset G \setminus\{1\}$ and $H' \subset H$ 
    for the data given by \autoref{res: morphism to hyp} applied with $\iota \colon H \to G$ and $\Gamma$.
	Let $U$ be a finite generating set of $G$ that contains a finite generating set $U'$ of $H'$.
	The constant $\delta \in \R^*_+$ and the maps $a, \rho \colon (0, 1) \times (0,1) \to \R^*_+$ are the ones given by \autoref{res: recap sc}.
	We also write $\xi \in (0,1)$ for the parameter given by \autoref{res: recap sc}\ref{enu: recap sc - lift}.
	
	We fix a non-decreasing exhaustion $(W_k)$ of $G\setminus\{1\}$ by finite subsets starting with the set $W_0 = W$ defined previously.
	Let $(\lambda_k)$ and $(\epsilon_k)$ be two sequences of positive numbers converging to zero.
	Assume now that the statement is false.
	In particular, for every $k \in \N$, there is a non-elementary, tight $C'(\lambda_k, \epsilon_k)$ strengthened small cancellation quotient $\bar \Gamma_k$ of $\Gamma$ and a morphism $\bar \varphi_k \colon G \to \bar \Gamma_k$ such that 
	\begin{enumerate}
		\item \label{res: lifting morphism - stably one-to-one - kernel}
		$W_k \cap \ker(\bar \varphi_k \circ \alpha) = \emptyset$, for every $\alpha \in \aut[H]G$,
		\item \label{res: lifting morphism - stably one-to-one - lift}
		if $\bar \varphi_k$ admits a lift $\varphi_k \colon G \to \Gamma$ whose restriction to $H$ coincides with $\iota$, then $\varphi_k$ is not injective.
	\end{enumerate}
	For every $k \in \N$, we write $\mathcal Q_k$ for the underlying relation family, $\pi_k \colon \Gamma \onto \bar \Gamma_k$ for the canonical projection, and $X_k$ for the $\delta$-hyperbolic space associated to $\bar \Gamma_k$ by \autoref{res: recap sc}.
	Let $\sigma \in \R_+^*$.
	Up to replacing $\bar \varphi_k$ by $\bar \varphi_k \circ \alpha$ for some $\alpha \in \aut[H]G$, we can assume without loss of generality that $(\bar \Gamma_k, \bar \varphi_k)$ is $\sigma$-short among the collection
	\begin{equation*}
		\mathfrak C_k = \set{(\bar \Gamma_k, \bar \varphi_k \circ \alpha)}{ \alpha \in \aut[H]G},
	\end{equation*}
	that is,
	\begin{equation*}
		E_1(\bar \varphi_k, U) \leq \inf_{\alpha \in \aut[H]G} E_1( \bar \varphi_k \circ \alpha, U)  + \sigma.
	\end{equation*}
	
	We first claim that the $\ell^\infty$-energy $E_\infty(\bar \varphi_k, U)$ diverges to infinity.
	Suppose that, contrary to our claim, $E_\infty(\bar \varphi_k, U)$ is bounded.
	Since $\mathcal Q_k$ satisfies the strengthened small cancellation condition $C'(\lambda_k, \epsilon_k)$ the quantity $T(\mathcal Q_k, X)$ diverges to infinity.
	In particular, it is much larger than $E_\infty(\iota, U')$, which does not depend on $k$.
	Up to forgetting the first terms of the sequence, every $\bar \varphi_k \colon G \to \bar \Gamma_k$ lifts to a morphism $\varphi_k\colon G \to \Gamma$, which coincides with $\iota$ on $H'$, see \autoref{res: recap sc}\ref{enu: recap sc - lift}.
	Observe that $\ker(\varphi_k \circ \alpha) \subset \ker(\bar \varphi_k \circ \alpha)$ for every $\alpha \in \aut[H]G$.
	It follows from our hypothesis that the intersection of those kernels with $W$ is empty.
	According to \autoref{res: morphism to hyp}, the lift $\varphi_k$ is injective and its restriction to $H$ coincides with $\iota$.
	This contradicts our assumption on $\bar \varphi_k$ and completes the proof of the claim.

	It follows from our construction that $(\bar \Gamma_k, \bar \varphi_k)$ converges to $(G, \id)$.
	Recall that $\bar \varphi_k$ and $\pi_k \circ \iota$ coincide on $H$.
	Therefore, for every $h \in H$, the  $\ell^\infty$-energy $E_\infty(\bar \varphi_k, \{h\})$ is bounded.
	It then follows from \autoref{res: shortening argument - relative} that there are $\tau \in \R_+^*$ and a sequence of automorphisms $\alpha_k \in \mcg{G,H}$ such that
	\begin{equation*}
		E_1\left(\bar \varphi_k \circ \alpha_k, U\right) \leq (1- \tau) E_1\left(\bar \varphi_k,U\right).
	\end{equation*}
	Note that $\alpha_k$ may not be the identity when restricted to $H$.
	However, it acts on $H$ as conjugation by an element of $G$.
	In other words, there is $\beta_k \in \aut[H]G$ such that $\alpha_k$ and $\beta_k$ differ by an inner automorphism.
	Consequently,
	\begin{equation*}
		E_1\left(\bar \varphi_k \circ \beta_k, U\right) =  E_1\left(\bar \varphi_k \circ \alpha_k, U\right),
	\end{equation*}
	see \autoref{rem: energy invariant by inner automorphisms}.
	Recall that $E_\infty(\bar \varphi_k, U)$, hence $E_1(\bar \varphi_k, U)$, diverges to infinity.
	Hence, the last inequality contradicts the fact that $(\bar \Gamma_k, \bar \varphi_k)$ is $\sigma$-short among $\mathfrak C_k$ and completes the proof.
\end{proof}

Our next step is to remove the assumption that $G$ is freely indecomposable relative to $H$.
The drawback is that the possible lift may not be injective.

\begin{coro}
\label{res: lifting morphism - killing element for hyperbolic groups - general}
	Let $G$ be a finitely generated group.
	Let $H$ be a subgroup of $G$.
	Let $\iota \colon H \to \Gamma$ be a morphism.
	There exist a finite subset $W \subset G \setminus \{1\}$ and parameters $\lambda, \epsilon \in (0,1)$, with the following properties.
	Let $\pi \colon \Gamma \onto \bar \Gamma$ be a tight $C'(\lambda, \epsilon)$ strengthened small cancellation quotient of $\Gamma$.
	Let $\bar \varphi \colon G \to \bar \Gamma$ be a morphism such that $\bar \varphi$ restricted to $H$ coincides with $\pi \circ \iota$.
	Then one of the following holds:
	\begin{itemize}
		\item there is an automorphism $\alpha \in \aut[H]G$ such that $W \cap \ker(\bar \varphi \circ \alpha) \neq \emptyset$,
		\item there is a morphism $\varphi \colon G \to \Gamma$ lifting $\bar \varphi$ whose restriction to $H$ coincides with $\iota$.
	\end{itemize}
\end{coro}

\begin{proof}
	Consider the Grushko decomposition of $G$ relative to $H$:
	\begin{equation*}
		G = G_0 \ast G_1 \ast \dots \ast G_m \ast F,
	\end{equation*}
	where $H$ is contained in $G_0$ which is freely indecomposable relative to $H$; for every $j \in \intvald 1m$, the factor $G_j$ is non-cyclic and freely indecomposable; and $F$ is a finitely generated free group.
	To simplify the exposition, we let $H_0 = H$ and $H_j = \{1\}$ for every $j \in \intvald 1m$ so that each $G_j$ is freely indecomposable relative to $H_j$.
	Similarly, we consider the morphisms $\iota_j \colon H_j \to \Gamma$ that equals $\iota$ if $j = 0$ and is trivial otherwise.
	For every $j \in \intvald 0m$, we denote by $W_j \subset G_j \setminus\{1\}$ and $\lambda_j, \epsilon_j \in (0,1)$ the data provided by \autoref{res: lifting morphism - killing element for hyperbolic groups - one-ended} applied with $G_j$ and $H_j$.
	We now let 
	\begin{equation*}
		\lambda = \min_{0 \leq j \leq m} \lambda_j, \quad
		\epsilon = \min_{0 \leq j \leq m} \epsilon_j, 
		\quad \text{and} \quad
		W = \bigcup_{j = 0}^m W_j.
	\end{equation*}
	where $W$ is seen as a finite subset of $G$.
	
	Consider now a tight $C'(\lambda, \epsilon)$ strengthened small cancellation quotient $\bar \Gamma$ of $\Gamma$ and denote by $\pi \colon \Gamma \onto \bar \Gamma$ the corresponding projection.
	Note that $\bar \Gamma$ is a $C'(\lambda_j, \epsilon_j)$ tight small cancellation quotient of $\Gamma$ for every $j \in \intvald 0m$.
	Let $\bar \varphi \colon G \to \bar \Gamma$ be a morphism such that $\bar \varphi$ restricted to $H$ coincides with $\pi \circ \iota$.
	In addition we assume that $W \cap \ker(\varphi \circ \alpha) = \emptyset$, for every $\alpha \in \aut[H]G$.
	
	Let $j \in \intvald 0m$.
	We write $\bar \varphi_j \colon G_j \to \bar \Gamma$ for the restriction of $\bar \varphi$ to $G_j$.
	Note that every element of $\aut[H_j]{G_j}$ extends to an element of $\aut[H]G$.
	Therefore $W_j \cap \ker (\bar \varphi_j \circ \alpha) = \emptyset$, for every $\alpha \in \aut[H_j]{G_j}$.
	It follows from \autoref{res: lifting morphism - killing element for hyperbolic groups - one-ended} and our choice of $W_j$ and $(\lambda_j, \epsilon_j)$ that there is a morphism $\varphi_j \colon G_j \to \Gamma$ lifting $\bar \varphi_j$ whose restriction to $H_j$ coincides with $\iota_j$.
	Now, denote by $\bar \sigma \colon F \to \bar \Gamma$ the restriction of $\bar \varphi$ to $F$.
	Since $F$ is free, there is also a morphism $\sigma \colon F \to \Gamma$ lifting $\bar \sigma$.
	
	Using the universal property of the free product, we now define a morphism $\varphi \colon G \to \Gamma$ whose restriction to $G_j$ (\resp, $F$) is $\varphi_j$ (\resp, $\sigma$).
	It follows from the construction that $\varphi$ lifts $\bar \varphi$.
	Moreover, its restriction to $H$ coincides with $\iota$.
\end{proof}

We will now use \autoref{res: lifting morphism - killing element for hyperbolic groups - general} as a stepping stone for proving a stronger statement.

\begin{coro}
\label{res: lifting quotient0}
	Let $G$ be a finitely generated group.
	Let $H$ be a subgroup of $G$ and $\iota \colon H \to \Gamma$ a morphism.
	There are parameters $\lambda, \epsilon \in (0,1)$, with the following properties.
	Let $\pi \colon \Gamma \onto \bar \Gamma$ be a tight $C'(\lambda, \epsilon)$ strengthened small cancellation quotient of $\Gamma$.
	Let $\bar \varphi \colon G \to \bar \Gamma$ be a morphism such that $\bar \varphi$ restricted to $H$ coincides with $\pi \circ \iota$.
	Then there is a morphism $\varphi \colon G \to \Gamma$ lifting $\bar \varphi$ whose restriction to $H$ coincides with $\iota$.
\end{coro}

\begin{proof}
	Assume that the statement is false.
	Then there is a sequence of morphisms $\bar \varphi_k \colon G \to \bar \Gamma_k$ with the following properties:
	\begin{itemize}
		\item $\pi_k \colon \Gamma \onto \bar \Gamma_k$ is a non-elementary, tight, $C'(\lambda_k, \epsilon_k)$ strengthened small cancellation quotient of $\Gamma$, where the sequences $(\lambda_k)$ and $(\epsilon_k)$ converge to zero;
		\item $\bar \varphi_k$ restricted to $H$ coincides with $\pi_k \circ \iota$;
		\item no $\bar \varphi_k$ can be lifted to a morphism $G \to \Gamma$ whose restriction to $H$ coincides with $\iota$.
	\end{itemize}
	We are now going to produce an infinite (strictly) descending sequence of limit groups
	\begin{equation*}
		(L_0, \eta_0) \succ (L_1, \eta_1) \succ (L_2, \eta_2) \succ \cdots
	\end{equation*}
	and for every $i \in \N$, a sequence of morphisms $\bar \mu_k^i \colon L_i \to \bar \Gamma_k$ with the following properties: up to passing to a subsequence, for every $k \in \N$,
	\begin{enumerate}
		\item \label{enu: lifting quotient - H}
		the morphisms $\bar \mu^i_k \circ \eta_i$ and $\pi_k \circ \iota$ coincide on $H$;
		\item \label{enu: lifting quotient - no lifting}
		the morphism $\bar \mu^i_k$ cannot be lifted into a morphism $\mu^i_k \colon L_i \to \Gamma$ such that $\mu^i_k \circ \eta_i$ restricted to $H$ coincides with $\iota$.
	\end{enumerate}
	Recall that there is no infinite descending sequence of limit groups (\autoref{res: dcc}).
	Thus, this construction will provide the desired contradiction.
	
	\paragraph{Basis of the induction.}
	Up to passing to a subsequence, we can assume that $(\bar \Gamma_k, \bar \varphi_k)$ converges to a limit group, say $(L_0, \eta_0)$.
	Recall that $(L_0, \eta_0)$ has the factorization property (\autoref{res: dcc}).
	Up to passing again to a subsequence, every $\bar \varphi_k$ factors through $\eta$.
	We write $\bar \mu^0_k \colon L_0 \to \bar \Gamma_k$ for the resulting morphism.
	Note that $\bar \mu^0_k \circ \eta_0 = \bar \varphi_k$ and $\pi_k \circ \iota$ coincide on $H$.
	Moreover, it follows from the construction that for every $k \in \N$, the morphism $\bar \mu^0_k$ cannot be lifted to a morphism $\mu^0_k \colon L \to \Gamma$ such that $\mu^0_k \circ \eta$ restricted to $H$ coincides with $\iota$ (otherwise $\mu^0_k \circ \eta_0$ would be a lift of $\bar \varphi_k$ which coincides on $H$ with $\iota$).

	\paragraph{Induction step.}
	Let $i \in \N$ be such that $(L_i, \eta_i)$ and $(\bar \mu^i_k)$ have been already defined, satisfying \ref{enu: lifting quotient - H} and \ref{enu: lifting quotient - no lifting}.
	Denote by $H_i$ the image of $H$ in $L_i$.
	We claim that $\iota \colon H \to \Gamma$ factors through $\eta_i \colon H \onto H_i$.
	Consider indeed an element $h \in H$ such that $\eta_i(h) = 1$.
	In particular, for every $k \in \N$, we have
	\begin{equation*}
		\pi_k \circ \iota(h) = \bar \mu^i_k \circ \eta_i(h) = 1.
	\end{equation*}
	However, the stable kernel of  $(\bar \Gamma_k , \pi_k)$ is trivial, see \autoref{res: recap sc}\ref{enu: recap sc - injectivity}.
	Hence, $\iota(h) = 1$, whence the claim.
	We now write $\iota_i \colon H_i \to \Gamma$ for the resulting morphism.
	It is a reformulation of \ref{enu: lifting quotient - no lifting} that, up to passing to a subsequence, there is no morphism $\mu^i_k \colon L_i \to \Gamma$ lifting $\bar \mu^i_k$ such that $\mu^i_k$ and $\iota_i$ coincide on $H_i$.
	
	If follows from \autoref{res: lifting morphism - killing element for hyperbolic groups - general} that, up to passing to a subsequence, there is an element $g_i \in L_i \setminus\{1\}$ and, for every $k \in \N$, an automorphism $\alpha_k^i \in \aut[H_i]{L_i}$ such that $\bar \mu_k^i \circ \alpha_k^i (g_i) = 1$.
	For every $k \in \N$, we let 
	\begin{equation*}
		\bar \varphi_k^{i+1} = \bar \mu_k^i \circ \alpha_k^i \circ \eta_i.
	\end{equation*}
	Up to passing again to a subsequence, we can assume that $(\bar \Gamma_k, \bar \varphi_k^{i+1})$ converges to a limit group $(L_{i+1}, \eta_{i+1}) \prec (L_i, \eta_i)$.
	Note that $L_{i+1}$ is a proper quotient of $L_i$.
	Indeed, by construction, the element $g_i$ lies in the kernel of the projection $\zeta_i \colon L_i \onto L_{i+1}$.
	As a limit group, $(L_{i+1}, \eta_{i+1})$ has the factorization property (\autoref{res: dcc}).
	Up to passing to a subsequence, every $\bar \varphi_k^{i+1}$ factors through $\eta_{i+1}$.
	We write $\bar \mu_k^{i+1} \colon L_{i+1} \to \bar \Gamma_k$ for the resulting morphism.

\[\begin{tikzcd}
	H & G & {L_i} & {L_i} \\
	&& {L_{i+1}} & {\bar \Gamma_k}
	\arrow[hook, from=1-1, to=1-2]
	\arrow["{\eta_i}", two heads, from=1-2, to=1-3]
	\arrow["{\alpha^i_k}", hook, two heads, from=1-3, to=1-4]
	\arrow["{\zeta_i}"', two heads, from=1-3, to=2-3]
	\arrow["{\bar \mu^i_k}", from=1-4, to=2-4]
	\arrow["{\bar \mu^{i+1}_k}"', from=2-3, to=2-4]
\end{tikzcd}\]
    
	Let $k \in \N$.
	Recall that $\alpha^i_k$ induces the identity on $H_i$.
	Hence, for every $h \in H$, we have
	\begin{equation*}
		 \bar \mu_k^{i+1} \circ \eta_{i+1}(h) 
		 = \bar \varphi_k^{i+1}(h)
		 = \bar \mu_k^i \circ \alpha_k^i \circ \eta_i(h)
		 = \bar \mu_k^i \circ \eta_i(h)
		 = \pi_k \circ \iota(h).
	\end{equation*}
	Suppose now that $\bar \mu^{i+1}_k$ lifts to a morphism $\mu^{i+1}_k \colon L_{i+1} \to \Gamma$ such that $\mu^{i+1}_k \circ \eta_{i+1}$ coincides on $H$ with $\iota$.
	By construction, $\bar \mu^{i+1}_k$ satisfies the relation
	\begin{equation*}
		\bar \mu_k^i \circ  \alpha_k^i  = \bar \mu_k^{i+1} \circ \zeta_i,
	\end{equation*}
	hence 
	\begin{equation*}
		\mu^{i+1}_k \circ \zeta_i \circ \left( \alpha^i_k\right)^{-1} \colon L_i \to \Gamma
	\end{equation*}
	lifts $\bar \mu^i_k$.
	Using again the fact that $\alpha^i_k$ induces the identity on $H_i$ we have for every $h \in H$, 
	\begin{equation*}
		\mu^{i+1}_k \circ \zeta_i \circ \left( \alpha^i_k\right)^{-1} \circ \eta_i(h)
		= \mu^{i+1}_k \circ \zeta_i \circ \eta_i(h)
		= \mu^{i+1}_k \circ \eta_{i+1}(h)
		= \iota(h)
	\end{equation*}
	This contradicts the definition of $\bar \mu^i_k$ and completes the induction step, as well as the proof of the corollary.
\end{proof}

We slightly improve again the previous statement by removing the assumption that $H$ is a subgroup of $G$.
This could have been done already in the previous statements, but we thought it would help the reader to reduce a little the number of objects in the proofs.

\begin{coro}[\autoref{intro:thm:lifting}]
\label{res: lifting quotient}
	Let $G$ be a finitely generated group and $H$ a group.
	Let $\jmath \colon H \to G$ and $\iota \colon H \to \Gamma$ be two morphisms.
	There are parameters $\lambda, \epsilon \in (0,1)$, with the following properties.
	Let $\pi \colon \Gamma \onto \bar \Gamma$ be a tight $C'(\lambda, \epsilon)$ strengthened small cancellation quotient of $\Gamma$.
    Let $\bar \varphi \colon G \to \bar \Gamma$ be a morphism such that $\bar \varphi \circ \jmath = \pi \circ \iota$.
	Then there is a morphism $\varphi \colon G \to \Gamma$ lifting $\bar \varphi$ such that $\varphi \circ \jmath = \iota$.
\end{coro}

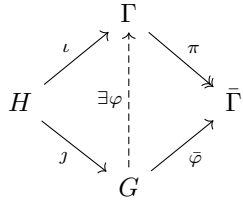
\[\begin{tikzcd}
	& \Gamma \\
	H && {\bar{\Gamma}} \\
	& G
	\arrow["\pi", two heads, from=1-2, to=2-3]
	\arrow["\iota", from=2-1, to=1-2]
	\arrow["\jmath"', from=2-1, to=3-2]
	\arrow["\exists\varphi", dashed, from=3-2, to=1-2]
	\arrow["{\bar{\varphi}}"', from=3-2, to=2-3]
\end{tikzcd}\]

\begin{proof}
    Note that if $\iota$ factors through $\jmath$, then the statement follows from \autoref{res: lifting quotient0} applied with the subgroup $\jmath(H)$.
    We now assume that the statement does not hold. 
    We are going to prove that $\iota$ actually factors through $\jmath$, thus reaching a contradiction.
    Let $(\lambda_k)$ and $(\epsilon_k)$ be a sequence of real numbers converging to zero.
    If the statement is false, then for every $k \in \N$, there is a tight $C'(\lambda_k, \epsilon_k)$ strengthened small cancellation quotient of $\Gamma$ that we denote by $\pi_k \colon \Gamma \onto \bar \Gamma_k$ and a morphism $\bar \varphi_k \colon G \to \bar \Gamma_k$ such that $\bar \varphi_k \circ \jmath = \pi_k \circ \iota$ (all the other properties will not be useful here, just the existence of $\bar \varphi_k$).
    We write $\mathcal Q_k$ for the relation family used to define $\bar \Gamma_k$.
    Consider now $h \in H$ in the kernel of $\jmath \colon H \to G$.
    In particular, $\pi_k \circ \iota (h) = \bar \varphi_k \circ \jmath(h) = 1$, that is, $\iota(h)$ belongs to $K(\mathcal Q_k)$ (the normal subgroup generated by the relation family).
    Since $\epsilon_k$ converges to zero, the strengthening of the small cancellation assumption implies that $T(\mathcal Q_k, X)$ diverges to infinity.
    It follows from \autoref{res: recap sc}\ref{enu: recap sc - injectivity} that $\iota(h) = 1$.
    This holds for every $h \in \ker \jmath$.
    Hence, $\iota \colon H \to \Gamma$ factors through $\jmath \colon H \to G$, a contradiction.
\end{proof}

The following remark will be useful when dealing at the same time with $\Gamma$ and $\Gamma \ast \Z$, in the applications to elementary embeddings: see \autoref{res: induction step toward tarski monsters}.

\begin{rema}
\label{rem: lifting free product}
	Let $\Gamma$ be a torsion-free non-elementary hyperbolic group, and let $\pi \colon \Gamma \to \bar \Gamma$ be a quotient. 
	We consider the corresponding quotient $\pi_\ast \colon \Gamma \ast \Z \to \bar \Gamma \ast \Z$. Let $\jmath \colon H \to G, \iota \colon H \to \Gamma$ be morphisms and let $\iota_\ast \colon H \to \Gamma \ast \Z$ be the corresponding morphism.

	Let $\lambda, \epsilon$ be the constants given by \autoref{res: lifting quotient} for the morphisms $\jmath \colon H \to G, \iota_\ast \colon H \to \Gamma \ast \Z$. 
	Suppose that $\pi_* \colon \Gamma \ast \Z \to \bar \Gamma \ast \Z$ is a tight $C'(\lambda, \epsilon)$ strengthened small cancellation quotient of $\Gamma \ast \Z$.
	We claim that then the conclusion of \autoref{res: lifting quotient} holds also for the quotient $\pi \colon \Gamma \to \bar \Gamma$ and the morphisms $\jmath \colon H \to G, \iota \colon H \to \Gamma$.

	Indeed, let $\bar \varphi \colon G \to \bar \Gamma$ be a morphism such that $\bar \varphi \circ \jmath = \pi \circ \iota$. Consider the corresponding morphism $\bar \varphi_\ast \colon G \to \bar \Gamma \ast \Z$, which satisfies $\bar \varphi_\ast \circ \jmath = \pi_* \circ \iota_\ast$.
	\autoref{res: lifting quotient} applies and gives a morphism $\varphi_\ast \colon G \to \Gamma \ast \Z$ such that $\iota_\ast = \varphi_\ast \circ \jmath$ and $\bar \varphi_\ast = \pi_* \circ \varphi_\ast$.
	We define $\varphi$ to be the composition of $\varphi_\ast$ with the natural retraction $\Gamma \ast \Z \to \Gamma$. 
	Then $\iota = \varphi \circ \jmath$ and $\bar \varphi =  \pi \circ \varphi$, as desired.
\end{rema}

%%%%%%%%%%%%%%%%%%%%%%%%%%%%%%%%%%%%%%%%%%%%%%%%%%%%%%%%%%%%%%%%%%%%%%%%%%%%%%%%%%%%%
%%%%%%%%%%%%%%%%%%%%%%%%%%%%%%%%%%%%%%%%%%%%%%%%%%%%%%%%%%%%%%%%%%%%%%%%%%%%%%%%%%%%%
%
\section{Applications: the one-quantifier Knight \mbox{Conjecture}}
\label{sec: random}
%
%%%%%%%%%%%%%%%%%%%%%%%%%%%%%%%%%%%%%%%%%%%%%%%%%%%%%%%%%%%%%%%%%%%%%%%%%%%%%%%%%%%%%
%%%%%%%%%%%%%%%%%%%%%%%%%%%%%%%%%%%%%%%%%%%%%%%%%%%%%%%%%%%%%%%%%%%%%%%%%%%%%%%%%%%%%

In this section and the next, we apply our lifting results to the first-order theory of groups.  
We start by introducing a general framework to make this translation, and apply it here to small cancellation quotients, solving the one-quantifier Knight Conjecture for random quotients of torsion-free hyperbolic groups (\autoref{intro:thm:knight}). For the rest of this section, we fix a non-elementary torsion-free hyperbolic group $\Gamma$.

\begin{nota}
	Let $G$, $H$, and $\Gamma$ be three groups.
	Given two morphisms $\jmath \colon H \to G$ and $\iota \colon H \to \Gamma$, we denote by $\hom[\jmath, \iota] G\Gamma$ the set of all morphisms $\varphi \colon G \to \Gamma$ such that $\varphi \circ \jmath$ and $\iota$ coincide on $H$.
\end{nota}

\[\begin{tikzcd}
	& \Gamma \\
	H \\
	& G
	\arrow["\iota", from=2-1, to=1-2]
	\arrow["\jmath"', from=2-1, to=3-2]
	\arrow["{\varphi \in \hom[\jmath, \iota] G\Gamma}"', from=3-2, to=1-2]
\end{tikzcd}\]

In what follows, we use the notation $\vec x$ to denote a finite tuple $x_1, \dots, x_r$ of variables.
In particular, $\forall \vec x$ and $\exists \vec x$ stand for $\forall x_1 \dots \forall x_r$ and $\exists x_1 \dots \exists x_r$, respectively.
The number $r$ is called the \emph{arity} or $\vec x$.

\subsection{$\forall\exists$-formulas with coefficients}

We denote by $\mathfrak R$ the collection of all triples $(H, C, \mathcal G)$ where 
\begin{enumerate}
	\item $H$ and $C$ are two finitely generated free groups,
	\item $\mathcal G$ is a finite collection of triples $(G, V, \jmath)$ with $G$ a finitely presented group, $V$ a finite subset of $G$, and $\jmath \colon H \ast C \to G$ a morphism.
\end{enumerate}

\begin{defi}
	Let $(H, C, \mathcal G) \in \mathfrak R$.
	Let $\Gamma$ be a group and $\sigma \colon C \to \Gamma$ a morphism.
	Given another morphism $\iota \colon H \to \Gamma$ we say that 
	\begin{itemize}
		\item $(H,C,\mathcal G)$ is \emph{positively realized in $(\Gamma, \iota, \sigma)$} if there exists $(G,V,\jmath) \in \mathcal G$ and a morphism 
		\begin{equation*}
			\varphi \in  \hom[\jmath, \iota \ast \sigma] G\Gamma
			\quad \text{such that} \quad
			V \cap \ker \varphi = \emptyset.
		\end{equation*}
		Such a morphism $\varphi$ is called a \emph{witness of $(H,C,\mathcal G)$ in $(\Gamma, \iota, \sigma)$};
		\item  $(H,C,\mathcal G)$ is \emph{negatively realized in $(\Gamma, \iota, \sigma)$} otherwise.
	\end{itemize}
	We say that  $(H,C,\mathcal G)$ is \emph{positively realized in $(\Gamma, \sigma)$} if for every morphism $\iota \colon H \to \Gamma$, the tuple $(H,C,\mathcal G)$ is positively realized in $(\Gamma, \iota, \sigma)$.
\end{defi}

\medskip
These notions relate to the first-order theory as follows.
Consider an $\forall \exists$-formula $\Sigma(\vec z)$ with free variables $\vec z$, whose disjunctive normal form is 
\begin{equation}
\label{eqn: AE first-order sentence}
	\forall \vec x, \exists \vec y, \quad \bigvee_{i = 1}^p\left(\left(\bigwedge_{j = 1}^q v_{i,j}(\vec x, \vec y, \vec z) \neq 1\right) \wedge \left(\bigwedge_{j = 1}^q w_{i,j}(\vec x, \vec y, \vec z)= 1\right)\right).
\end{equation}
Let $H$ (\resp, $C$) be the free group generated by $\vec x$ (\resp, $\vec z$).
For every $i \in \intvald 1p$, we denote by $G_i$  the group whose finite presentation is
\begin{equation*}
	G_i = \left< \vec x,  \vec y, \vec z \mid w_{i,1}(\vec x, \vec y, \vec z), \dots, w_{i,q}(\vec x, \vec y, \vec z) \right>.
\end{equation*}
It comes with a natural morphism $\jmath_i \colon H \ast C \to G_i$.
In addition, we denote by $V_i$ the set of all words $v_{i,j}(\vec x, \vec y, \vec z)$ seen as elements of $G_i$. 
We let 
\begin{equation*}
	\mathcal G = \set{(G_i, V_i, \jmath_i)}{i = 1, \ldots, p}.
\end{equation*}
We say that $(H, C, \mathcal G) \in \mathfrak R$ is the \emph{group translation of $\Sigma(\vec z)$}.

\medskip
Consider now a group $\Gamma$ and a tuple $\vec \gamma$ in $\Gamma$ with the same arity as $\vec z$.
Denote by $\sigma \colon C \to \Gamma$ the morphism sending $\vec z$ to $\vec \gamma$.
We observe that $\Sigma(\vec \gamma)$ holds in $\Gamma$, if and only if $(H, C, \mathcal G)$ is positively realized in $(\Gamma, \sigma)$.
\footnote{The example explains our choice of notations: we see the group $C$ and the morphism $\sigma \colon C \to \Gamma$ as an assignment of the constants $\vec z$ in $\Sigma(\vec z)$. The letter $C$ stands for ``constant'' or ``coefficient'' and the letter $\sigma$ for ``{\greektext stajer'a}'' (constant in Greek).}
We can now reinterpret \autoref{res: lifting quotient} in terms of first-order formulas.

\begin{theo}
\label{res: theory quotients twoquant}
	Let $\Sigma(\vec{z})$ be an $\exists \forall$-formula with free variables $\vec z$.
    Let $\Gamma$ be a torsion-free non-elementary hyperbolic group.
    Let $\vec{\gamma}$ be a tuple of elements of $\Gamma$ with the same arity as $\vec z$.
    Suppose that $\Sigma(\vec{\gamma})$ holds in $\Gamma$.
    Then there are parameters $\lambda, \epsilon \in (0, 1)$ such that if $\pi \colon \Gamma \to \bar \Gamma$ is a tight $C'(\lambda, \epsilon)$ strengthened small cancellation quotient of $\Gamma$, then $\Sigma(\pi(\vec{\gamma}))$ holds in $\bar \Gamma$.
\end{theo}

\begin{proof}
    The negation $\neg \Sigma(\vec{z})$ is an $\forall \exists$-formula, whose disjunctive normal form is
    \[\forall \vec x, \exists \vec y, \quad \bigvee_{i = 1}^p\left(\left(\bigwedge_{j = 1}^q v_{i,j}(\vec x, \vec y, \vec z) \neq 1\right) \wedge \left(\bigwedge_{j = 1}^q w_{i,j}(\vec x, \vec y, \vec z)= 1\right)\right).\]
    By assumption, $\neg \Sigma(\vec{\gamma})$ does not hold in $\Gamma$.
    Following the notation above, we let $(H, C, \mathcal{G})$ be the group translation of $\neg \Sigma(\vec{z})$, where $\mathcal{G} = \{ (G_i, V_i, \jmath_i) : i = 1, \ldots, p\}$, $C$ is the free group on the tuple $\vec z$, and $\sigma \colon C \to \Gamma$ is the morphism sending $\vec z$ to $\vec \gamma$.
    By the above discussion, there exists a morphism $\iota \colon H \to \Gamma$ such that $(H, C, \mathcal G)$ is negatively realized in $(\Gamma, \iota, \sigma)$, that is, for each $i = 1, \ldots, p$ and every morphism $\varphi_i \in \hom[\jmath, \iota \ast \sigma] {G_i}\Gamma$, it holds that $V_i \cap \ker \varphi_i \neq \emptyset$.

    Now let $\lambda_i, \epsilon_i$ be the constants given by \autoref{res: lifting quotient} relative to the morphisms $\jmath_i \colon H \ast C \to G_i$ and $\iota \ast \sigma \colon H \ast C \to \Gamma$. Let $\lambda = \min \lambda_i$ and $\epsilon = \min \epsilon_i$. We claim that if $\pi \colon \Gamma \to \bar \Gamma$ is a tight $C'(\lambda, \epsilon)$ strengthened small cancellation quotient of $\Gamma$, then $\Sigma(\pi(\vec \gamma))$ holds in $\bar \Gamma$.

    Indeed, suppose otherwise: there exists a quotient $\pi \colon \Gamma \to \bar \Gamma$ such that $\neg \Sigma(\pi(\vec \gamma))$ holds in $\bar \Gamma$. Then $(H, C, \mathcal G)$ is positively realized in $(\bar \Gamma, \pi \circ \sigma)$, in particular, $(H, C, \mathcal G)$ is positively realized in $(\bar \Gamma, \pi \circ \iota, \pi \circ \sigma)$. This implies that there exists $i = 1, \ldots, p$ and a witness $\bar \varphi_i \in \hom[\jmath_i, \pi \circ \iota \ast \pi \circ \sigma] {G_i}{\bar \Gamma}$ such that $V_i \cap \ker \bar \varphi_i = \emptyset$. By \autoref{res: lifting quotient}, the morphism $\bar \varphi_i$ lifts to a morphism $\varphi_i \in \hom[\jmath_i, \iota \ast \sigma] {G_i}\Gamma$. Because $V_i \cap \ker \bar \varphi_i = \emptyset$, we deduce that $V_i \cap \ker \varphi_i = \emptyset$, a contradiction.
\end{proof}

\begin{coro}
\label{res: theory quotients onequant}
	Let $\Sigma(\vec{z})$ be a one-quantifier formula with free variables $\vec z$.
	Let $\Gamma$ be a torsion-free non-elementary hyperbolic group.
	Let $\vec{\gamma}$ be a tuple of elements of $\Gamma$ with the same arity as $\vec z$.
	Then there are parameters $\lambda, \epsilon \in (0, 1)$ such that if $\pi \colon \Gamma \to \bar \Gamma$ is a tight $C'(\lambda, \epsilon)$ strengthened small cancellation quotient of $\Gamma$, then $\Sigma(\vec \gamma)$ holds in $\Gamma$ if and only if $\Sigma(\pi(\vec \gamma))$ holds in $\bar \Gamma$.
\end{coro}

\begin{proof}
    Both $\Sigma(\vec z)$ and its negation are $\exists \forall$-formulas, and so we conclude by \autoref{res: theory quotients twoquant}.
\end{proof}

\subsection{Random quotients}

We now apply \autoref{res: theory quotients onequant} to random quotients of a hyperbolic group. 
Let $\Gamma$ be a non-elementary, torsion-free, hyperbolic group and $X$ its Cayley graph.
A \emph{model for random quotients of $\Gamma$} is a sequence of probability measures $(\mathbb P_\ell)$ on the set of quotients of $\Gamma$.
A random quotient of $\Gamma$ (for this model) satisfies a property $\mathcal P$ \emph{with overwhelming probability} if 
\begin{equation*}
	\lim_{\ell \to \infty} \mathbb P_\ell( \bar \Gamma\ \text{has}\ \mathcal P) = 1.
\end{equation*}
We say that a model has the \emph{Small Cancellation Property} if for every $\lambda, \epsilon \in (0, 1)$, a random quotient of $\Gamma$ is a tight $C'(\lambda, \epsilon)$ strengthened small cancellation quotient of $\Gamma$ with overwhelming probability (the small cancellation condition is understood with respect to $X$ here).

\begin{exam}
	Fix a finite generating set $S$ of $\Gamma$ and $k \geq 1$.
	Denote by $\mu$ the uniform probability measure on $S \cup S^{-1}$ and $\mu^{\ast \ell}$ its $\ell$-th convolution product.
	Given $\ell \geq 1$, we define a probability measure on $\Gamma^k$ by choosing the element in each factor independently at random according to $\mu^{\ast \ell}$.
	Pushing this measure via the map
	\begin{equation*}
		(\gamma_1, \dots, \gamma_k) \mapsto \Gamma / \normal{\gamma_1, \dots \gamma_k}.
	\end{equation*}
	yields a probability measure $\mathbb P_\ell$ on the set of quotients of $\Gamma$.
	In other words, each $\gamma_i$ is the result of an $\ell$-step random walk according to $\mu$.
	The resulting model has the Small Cancellation Property \cite[Proposition 5.9]{arzhantseva:delzant}.
\end{exam}

In practice, many models for random quotients of $\Gamma$ have the Small Cancellation Property as long as they are obtained by adding a fixed number of additional relations to $\Gamma$,  see e.g.\ \cite{olshanskii:random, ollivier:random}.
The next statement is a direct consequence of \autoref{res: theory quotients onequant}.

\begin{coro}[\autoref{intro:thm:knight}]
\label{res: knight}
    Let $\Gamma$ be a torsion-free non-elementary hyperbolic group, and let $\vec \gamma$ be a tuple of elements of $\Gamma$.
    Consider a model for random quotients of $\Gamma$ with the Small Cancellation Property.
    Let $\Sigma(\vec z)$ be a one-quantifier formula, where $\vec z$ has the same arity as $\vec \gamma$. 
    Then $\Sigma(\vec \gamma)$ holds in $\Gamma$ if and only if $\Sigma(\pi(\vec \gamma))$ holds in a random quotient $\pi \colon \Gamma \to \bar \Gamma$ with overwhelming probability. \qed
\end{coro}

\begin{rema}
    In \autoref{res: theory quotients twoquant} we were able to carry all $\exists \forall$-formulas from $\Gamma$ to $\bar \Gamma$. In order to obtain an ``if and only if'' result analogous to \autoref{res: theory quotients onequant} for two-quantifier formulas, we would need to do the same with $\forall \exists$-formulas. It is possible that, at least for $\forall \exists$-sentences (i.e., with no coefficients), this could be achieved analogously to \cite{kharlampovich:sklinos2}, by appealing to Merzlyakov's Theorem, which holds for all torsion-free non-elementary hyperbolic groups \cite{positivehyp1}.

    In the next section, we are able to tackle all two-quantifier formulas for torsion-free Tarski monsters. This is because the limiting nature of the group makes it possible to control the whole theory by controlling only finite fragments at any given finite stage. It is possible that parts of that argument could also be streamlined using Merzlyakov's Theorem, but our approach is more direct and elementary (modulo \autoref{res: lifting quotient}).
\end{rema}

%%%%%%%%%%%%%%%%%%%%%%%%%%%%%%%%%%%%%%%%%%%%%%%%%%%%%%%%%%%%%%%%%%%%%%%%%%%%%%%%%%%%%
%%%%%%%%%%%%%%%%%%%%%%%%%%%%%%%%%%%%%%%%%%%%%%%%%%%%%%%%%%%%%%%%%%%%%%%%%%%%%%%%%%%%%
%

\section{Applications: torsion-free Tarski monsters}
\label{sec: tarski}

We now apply the results on lifting morphisms to construct torsion-free Tarski monsters with control on the first-order theory, proving the main results from the introduction.

\subsection{Finding small cancellation relations}

Finding small cancellation relations is especially easy in the free group, where it is a purely combinatorial condition. Here we explain how to reduce to this case.
Recall that a subgroup $\Gamma_0$ of $\Gamma$ is \emph{malnormal} if the following property holds: for every $\gamma \in \Gamma$, if $\gamma \Gamma_0 \gamma^{-1} \cap \Gamma_0$ is non-trivial, then $\gamma \in \Gamma_0$.

\begin{prop}
\label{res: carrying sc assumption}
	Let $\Gamma$ be a torsion-free group acting properly, by isometries on a $\delta$-hyperbolic geodesic space $X$.
	Let $\Gamma_0$ be a non-trivial malnormal subgroup of $\Gamma$.
	Suppose that there is a $\delta_0$-hyperbolic geodesic space $X_0$ endowed with a proper, cobounded action by isometries of $\Gamma_0$ and a $\Gamma_0$-equivariant quasi-isometric embedding $f \colon X_0 \to X$.
	Then there is $r \in \R_+^*$ such that for every $\lambda, \epsilon \in (0,1)$, there are $\lambda_0, \epsilon_0 \in (0,1)$ with the following property.
	
	Choose a relation family $\mathcal Q_0$ satisfying the strengthened $C'(\lambda_0, \epsilon_0)$ small cancellation condition for the action of $\Gamma_0$ on $X_0$.
	Consider the collection
	\begin{equation*}
		\mathcal Q = \set{ (\gamma R\gamma^{-1},\gamma f_i(Y)^{+r})}{ (R,Y) \in \mathcal Q_0, \gamma \in \Gamma}.
	\end{equation*}
	Then $\mathcal Q$ satisfies the strengthened $C'(\lambda, \epsilon)$ small cancellation condition for the action of $\Gamma$ on $X$.
	Moreover, if $\mathcal Q_0$ is tight, then so is $\mathcal Q$.
\end{prop}

\begin{proof}
	We first define some auxiliary parameters.
	By assumption there are $(\kappa, \ell) \in \R_+^*\times \R_+$ such that for every $x,x' \in X_0$, we have
	\begin{equation*}
		\kappa^{-1} \dist x{x'} - \ell \leq \dist {f(x)}{f(x')} \leq \kappa \dist x{x'} + \ell.
	\end{equation*}
	It follows from the Morse Lemma that there exist $\alpha, L \in \R_+$ with the following properties.
	For every $2\delta_0$-quasi-convex subsets $Y_1, Y_2 \subset X_0$, their images $f(Y_1)$ and $f(Y_2)$ are $\alpha$-quasi-convex.
	Hence, their $(\alpha + 2 \delta)$-neighborhoods in $X$ are strongly quasi-convex \cite[Lemma~2.14]{Coulon:2014fr}.
	Moreover,
	\begin{equation*}
		\diam {f(Y_1)^{+\alpha+ 5\delta} \cap f(Y_2)^{+\alpha + 5\delta}} \leq \kappa \diam{ Y_1^{+3\delta_0} \cap Y_2^{+ 3\delta_0}} + L.
	\end{equation*}
	We set $r = \alpha + 2 \delta$.
	By assumption, the action of $\Gamma_0$ on $X_0$ is cobounded.
	Hence, so is the action of $\Gamma_0$ on $f(X_0)$.
	Using the malnormality, we obtain that there is $D \in \R_+$ such that for every $\gamma \in \Gamma$, if 
	\begin{equation*}
		\diam { f(X_0)^{+\alpha+5\delta} \cap \gamma f(X_0)^{+\alpha+5\delta}} > D,
	\end{equation*}
	then $\gamma \in \Gamma_0$.
	Recall that the subgroup $\Gamma_0$ is non-trivial and acts properly co-compactly on $X_0$.
	Hence, $\inj{\Gamma_0}{X_0} > 0$.
	Let $\lambda, \epsilon \in (0,1)$, we now choose $\lambda_0, \epsilon_0$ such that the following hold:
	\begin{align}
		\label{eqn: carrying sc assumption - 1}
		\frac 1\kappa \left[ \frac{\inj{\Gamma_0}{X_0}}{\epsilon_0} - 8\delta_0 \right] & 
		\geq \max \left\{ \frac {\inj \Gamma X}{\epsilon}, \frac \delta \epsilon, \frac D\lambda, \frac L\lambda \right\} \\
		\label{eqn: carrying sc assumption - 2}
		\lambda_0&\leq  \frac 1{2\kappa} \min \left\{ \frac \lambda \kappa, \frac L{8\delta_0}\right\}.
	\end{align}

	Consider a relation family $\mathcal Q_0$ satisfying the strengthened $C'(\lambda_0, \epsilon_0)$ small cancellation condition for the action of $\Gamma_0$ on $X_0$, and set 
	\begin{equation*}
		\mathcal Q = \set{ (\gamma R\gamma^{-1},\gamma f(Y)^{+r})}{ (R,Y) \in \mathcal Q_0,\ \gamma \in \Gamma}.
	\end{equation*}
	We first claim that $\Delta(\mathcal Q, X)$ is bounded above by 
	\begin{equation*}
		 \Delta =  \max \left\{ D,  \kappa \Delta(\mathcal Q_0, X_0) + L \right\}.
	\end{equation*}
	Consider indeed $(R_1,Y_1)$ and $(R_2, Y_2)$ in $\mathcal Q_0$, as well as two elements $\gamma_1, \gamma_2 \in \Gamma$ such that 
	\begin{equation}
	\label{eqn: carrying sc assumption}
		\diam {\gamma_1 f(Y_1)^{+r+3\delta} \cap {\gamma_2 f(Y_2)}^{+r+3\delta}} > \Delta.
	\end{equation}
	We are going to show that $\gamma = \gamma_1^{-1}\gamma_2$ belongs to $\Gamma_0$ and $(R_1, Y_1) = \gamma(R_2, Y_2)$.
	Up to translating everything by $\gamma_1^{-1}$, we can assume without loss of generality that $\gamma_1 = 1$, so that $\gamma = \gamma_2$.
	It follows from \eqref{eqn: carrying sc assumption} that 
	\begin{equation*}
		\diam { f(X_0)^{+\alpha+ 5\delta} \cap \gamma f(X_0)^{+\alpha+ 5\delta}} > D.
	\end{equation*}
	Hence, according to our choice of $D$, we already know $\gamma \in \Gamma_0$.
	It now follows from our choice of $L$ that 
	\begin{equation*}
		 \diam{ Y_1^{+3\delta_0} \cap \gamma Y_2^{+3\delta_0}} > \Delta(\mathcal Q_0, X_0).
	\end{equation*}
	Consequently, $(R_1, Y_1) = \gamma (R_2, Y_2)$, which completes the proof of our claim.
	Let us now control $T(\mathcal Q, X)$.
	Observe that for every $\gamma \in \Gamma_0$, we have
	\begin{equation*}
		\snorm[X_0] \gamma \leq \kappa \snorm[X] \gamma, 
		\quad \text{hence} \quad 
		\norm[X_0] \gamma \leq \kappa \norm[X] \gamma + 8\delta_0,
	\end{equation*}
    (see \eqref{eqn: regular vs stable length}).
	Consequently,
	\begin{equation*}
		T(\mathcal Q_0, X_0) \leq \kappa T(\mathcal Q, X) + 8\delta_0.
	\end{equation*}
	One checks now, using \eqref{eqn: carrying sc assumption - 1} and \eqref{eqn: carrying sc assumption - 2}, that $\mathcal Q$ satisfies the $C'(\lambda, \epsilon)$ strengthened small cancellation condition.
	
	Suppose moreover that $\mathcal Q_0$ is tight.
	In particular, $\mathcal Q_0 / \Gamma_0$ is finite, hence so is $\mathcal Q/\Gamma$.
	Let $(R,Y) \in \mathcal Q_0$.
	We already assumed the action of $R$ on $Y$ is cobounded, hence so is the action of $R$ on $f(Y)^{+d}$.
	It follows from \eqref{eqn: carrying sc assumption} and the quasi-convexity of $Y$ that the stabilizer in $\Gamma$ of $f(Y)^{+d}$ coincides with the stabilizer in $\Gamma_0$ of $Y$, i.e.~$R$.
	Hence, $\mathcal Q$ is tight as well.
\end{proof}

\begin{rema}
\label{rem. carrying sc free}
	We will apply the above statement in the following special case.
	Let $\Gamma_0$ be a finitely generated, non-abelian, free group and $X$ the Cayley graph of $\Gamma$ with respect to a basis. 
	Let $\lambda \in (0,1)$.
	Consider a finite set $R \subset \Gamma_0$ satisfying the classical $C''(\lambda)$ condition\footnote{Recall that this requires that all pieces are short compared to the \emph{shortest} relation in $R$, regardless of whether they arise in that relation or not. Although the notations are not the same, the condition already appears in \cite{Delzant:2008tu} and \cite{Dahmani:2017ef}.}.
	For every $h \in R$, denote by $Y_h$ the axis of $h$ in $X$.
	Then the collection
	\begin{equation*}
		\mathcal Q = \set{ \left( \gamma \group h \gamma^{-1}, \gamma Y_h\right)}{ h \in R,\ \gamma \in \Gamma_0}
	\end{equation*}
	satisfies our $C'(\lambda, \epsilon)$ strengthened small cancellation condition (for the action of $\Gamma_0$ on $X$), where $\epsilon$ is the inverse of the length of the shortest relation in $R$.
	Note also that $\mathcal Q / \Gamma_0$ is finite.
	Hence, $\mathcal Q$ is tight whenever no element in $R$ is a proper power.
\end{rema}

In order to apply \autoref{res: carrying sc assumption}, in the form of \autoref{rem. carrying sc free}, we need a tool to find a wealth of quasi-convex malnormal free subgroups. Recall that an element $\gamma$ in a non-elementary torsion-free hyperbolic group $\Gamma$ is \emph{primitive} if it is not a proper power, which is equivalent to the fact that $\group \gamma$ is not properly contained in an elementary subgroup of $\Gamma$. We record the following result here, which is well-known to the experts.

\begin{prop}
\label{res: finding malnormal free}
    Let $\Gamma$ be a torsion-free, non-elementary hyperbolic group.
    Let $\gamma \in \Gamma$ be a primitive element and $\Lambda \subset \Gamma$ a non-elementary subgroup.
    For every $n \in \N$, there is a free subgroup $\Lambda_1 \subset \Lambda$ of rank $n$, such that the subgroup $\group{\gamma, \Lambda_1}$ is quasi-convex, malnormal and isomorphic to $\group \gamma \ast \Lambda_1$.
\end{prop}

\begin{proof}
    The hypothesis on $\gamma$ implies that $\group \gamma$ is quasi-convex and malnormal in $\Gamma$. 
    Then the result follows from \cite[Theorem 6.2]{abbott:hull}.
    In fact, the group $\Lambda_1$ can be chosen to be generated by  ``random elements'' of $\Lambda$.
\end{proof}

\subsection{Torsion-free Tarski monster relations}

The main type of relations we are interested in is those that impose that a given element is absorbed by a given non-elementary subgroup. The theorem below achieves these while preserving good lifting properties.

\begin{theo}
\label{res: induction step toward tarski monsters}
	Let $\Gamma$ be a non-elementary, torsion-free, hyperbolic group.
    For $i \in \intvald1n$, let $G_i$ be a finitely generated group, $H_i$ be any group, and $\jmath_i \colon H_i \to G_i$ and $\iota_i \colon H_i \to \Gamma \ast \Z$ be morphisms.
	Let $\Lambda$ be a non-elementary subgroup of $\Gamma$.
	Let $S$ be a finite subset of $\Gamma$.
	Then for every finite subset $U$ of $\Gamma \ast \Z$, there is a non-elementary, torsion-free, hyperbolic quotient $\bar \Gamma$ of $\Gamma$ with the following properties.
	\begin{enumerate}
		\item \label{eqn: induction step toward tarski monsters - onto}
		Under the projection $\pi \colon \Gamma \onto \bar \Gamma$, the image of $S$ is contained in the image of $\Lambda$, which is a non-elementary subgroup of $\bar \Gamma$.
		\item \label{eqn: induction step toward tarski monsters - embedding}
		The map $\Gamma \ast \Z \onto \bar \Gamma \ast \Z$ induced by $\pi$ is injective when restricted to $U$.
		\item \label{eqn: induction step toward tarski monsters - lifting bis}
		For every $i \in \intvald1n$, the map
		\begin{equation*}
			\begin{array}{ccc}
				\hom[{\jmath_i}, {\iota_i}] {G_i}{\Gamma\ast \Z} & \to & \hom[\jmath_i, \pi \circ \iota_i]{G_i}{\bar \Gamma\ast \Z} \\
				\varphi & \mapsto & \pi \circ \varphi,
			\end{array}
		\end{equation*}
		induced by $\pi$, is onto. 
	\end{enumerate}
\end{theo}

\begin{proof}
	By induction, it suffices to prove the statement when $S$ is reduced to a single element, say $\gamma$.
	Since $\Gamma$ is torsion-free and hyperbolic, every element can be written as a power of a primitive element.
	Hence, without loss of generality, we can assume that $\gamma$ is primitive.
	According to \autoref{res: finding malnormal free}, there are $x,y \in \Lambda$ such that $\Lambda_1 = \{\gamma, x,y\}$ is a quasi-convex, malnormal, free subgroup of rank three of $\Gamma \ast \Z$, contained in $\Gamma$.
	Consider now a relation in $\Lambda_1$ of the form
	\begin{equation*}
		w = \gamma xy^pxy^{p+1} \cdots xy^{p+q}.
	\end{equation*}
    Then $w$ is not a proper power in $\Lambda_1$, and if $p$ and $q$ are sufficiently large, then it can be made arbitrarily long and satisfy an arbitrary classical small cancellation condition in $\Lambda_1$.
	Hence, the associated relation family $\mathcal Q$ of $\Gamma \ast \Z$ given by \autoref{res: carrying sc assumption} satisfies the tight $C'(\lambda, \epsilon)$ strengthened small cancellation condition, where $\lambda, \epsilon \in (0,1)$ can be chosen as small as desired.
    If $\lambda$ and $\epsilon$ are sufficiently small, then \ref{eqn: induction step toward tarski monsters - embedding} follows from \autoref{res: recap sc}\ref{enu: recap sc - injectivity}; and \ref{eqn: induction step toward tarski monsters - lifting bis} follows from \autoref{res: lifting quotient}.

	Note that the corresponding quotient is simply $\bar \Gamma = \Gamma / \normal w$ so that the image of $\gamma$ in $\bar \Gamma$ belongs to the image of $\Lambda$. Finally, up to adding to the set $U$ a non-trivial commutator in $\Lambda$, we see that $\pi(\Lambda)$ is non-abelian. Since $\bar \Gamma$ is torsion-free, $\pi(\Lambda)$ is therefore non-elementary, completing the proof of \ref{eqn: induction step toward tarski monsters - onto}.
\end{proof}

\begin{rema}
\label{rem: induction step toward tarski monsters}
    If $\iota_i$ takes values in $\Gamma \subset \Gamma \ast \Z$, then by \autoref{rem: lifting free product}, the map
    \[\begin{array}{ccc}
				\hom[{\jmath_i}, {\iota_i}] {G_i}{\Gamma} & \to & \hom[\jmath_i, \pi \circ \iota_i]{G_i}{\bar \Gamma} \\
				\varphi & \mapsto & \pi \circ \varphi,
	\end{array}\]
    is onto.
\end{rema}

\subsection{Construction of the group}

\begin{defi}
	A \emph{torsion-free Tarski monster} is a torsion-free, finitely generated non-abelian group all of whose proper, non-trivial subgroups are cyclic.
\end{defi}

We will impose an additional property on these groups.
Given a morphism $\eta \colon G_1 \to G_2$ and a generating set $S$ of $G_1$, we denote by $r_S(\eta)$ the largest ball for the word metric of $G$ (with respect to $S$) on which the map $\eta$ is one-to-one.

\begin{defi}
	A finitely generated group $\Gamma$ is \emph{lacunary hyperbolic} if it is the direct limit of a sequence of groups 
	\begin{equation*}
		\Gamma_0 \overset{\eta_0}\onto \Gamma_1 \overset{\eta_1}\onto \Gamma_2 \onto \cdots \onto \Gamma_k \overset{\eta_k}\onto \Gamma_{k+1} \onto \cdots
	\end{equation*}
	with the following properties: 
	there is a finite generating set $S_k$ of $\Gamma_k$ for every $k$ such that the corresponding Cayley graph of $\Gamma_k$ is $\delta_k$-hyperbolic, $S_{k+1} = \eta_k(S_k)$, and 
	the ratio $\delta_k / r_{S_k}(\eta_k)$ converges to zero as $k$ tends to infinity.
\end{defi}

The original definition says that a group is lacunary hyperbolic if one of its asymptotic cones is an $\R$-tree; the definition we give here is equivalent \cite[Theorem 1.1]{lacunary}, and is the only one that we will use.

\begin{nota}
	For a sequence of sets $(A_k)_{k \in \N}$, we let
	\begin{equation*}
    		\liminf\limits_{k \to \infty} A_k = \bigcup \limits_{k \geq 1} \bigcap\limits_{j \geq k} A_j
		 \quad \text{ and } \quad 
		 \limsup\limits_{k \to \infty} A_k = \bigcap\limits_{k \geq 1} \bigcup\limits_{j \geq k} A_j.
	\end{equation*}
	In words: $a \in \liminf A_k$ if $a \in A_k$ holds eventually, while $a \in \limsup A_k$ if $a \in A_k$ holds infinitely often. Clearly $\liminf A_k \subset \limsup A_k$. In case the inclusion is an equality, we simply write $\lim\limits_{k \to \infty} A_k$.
\end{nota}

\begin{nota}
Given a group $\Gamma$ we write $\mathfrak M(\Gamma)$ for the collection of all tuples $(H, G, \jmath, \iota)$ where $H$ and $G$ are finitely presented groups while $\jmath \colon H \to G$ and $\iota \colon H \to \Gamma$ are morphisms.
\end{nota}

Note that $\mathfrak M(\Gamma)$ is countable. In fact, in the statement of \autoref{res: tarski}, we could enlarge $\mathfrak M(\Gamma)$ so that it contains more tuples $(H, G, \jmath, \iota)$ with $G$ finitely generated and $H$ arbitrary, as in \autoref{res: induction step toward tarski monsters}, as long as it remains countable.

\medskip

We now have everything in place for our main result.

\begin{theo}[\autoref{intro:thm:precise}]
\label{res: tarski}
	Let $\Gamma$ be a non-elementary, torsion-free, hyperbolic group.
	Then there is a sequence of non-elementary, torsion-free, hyperbolic groups
	\begin{equation*}
		\Gamma = \Gamma_0 \onto \Gamma_1 \onto \Gamma_2 \onto \cdots \onto \Gamma_k \onto \Gamma_{k+1} \onto \cdots
	\end{equation*}
	with the following properties.
	Let $\Gamma_\infty$ be the direct limit of the above sequence.
	For every $k \in \N \cup \{\infty\}$, denote by $\pi_k \colon \Gamma \onto \Gamma_k$ and $\zeta_k \colon \Gamma_k \onto \Gamma_\infty$ the corresponding projections.
	\begin{enumerate}
		\item \label{enu: tarski - tarski}
		The group $\Gamma_\infty$ is a lacunary hyperbolic, simple, torsion-free Tarski monster.
		\item \label{enu: tarski - morphism}
		For every $(H,G,\jmath, \iota)$ in $\mathfrak M(\Gamma)$, there is $k_0 \in \N$, such that for every $k \geq k_0$ the map
		\begin{equation*}
			\begin{array}{ccc}
				\hom[\jmath, \pi_k \circ \iota]G{\Gamma_k} & \to & \hom[\jmath, \pi_\infty \circ \iota]G{\Gamma_\infty} \\
				\varphi & \mapsto & \zeta_k \circ \varphi
			\end{array}
		\end{equation*}
		is onto.
		\item \label{enu: tarski - theory}
		For every tuple $\vec \gamma$ of elements in $\Gamma$ it holds
		\begin{equation*}
			{\Th}_{\forall \exists}(\Gamma_\infty, \pi_\infty(\vec \gamma)) 
			= \lim\limits_{k \to \infty} {\Th}_{\forall \exists}(\Gamma_k, \pi_k(\vec \gamma)).
		\end{equation*}
		\item \label{enu: tarski - embedding}
		The group $\Gamma_\infty$ $\exists\forall\exists$-elementarily embeds in $\Gamma_\infty \ast \Z$.
	\end{enumerate}
\end{theo}

\begin{rema}
	Item \ref{enu: tarski - embedding} and its consequences were not in the first version of the article.
	This fact results from discussions we had with Simon André.
\end{rema}

\begin{proof}
	Let $\Lambda$ be a group (in practice, it will be either $\Gamma$ or $\Gamma \ast \Z$).
	Besides the collection $\mathfrak M(\Lambda)$ which has already been defined, we also consider the following data.
	\begin{itemize}
		\item $\mathfrak F(\Lambda)$ is the set of all tuples $(H,C, \mathcal G, \sigma)$, where $(H, C, \mathcal G)$ belongs to $\mathfrak R$ while $\sigma \colon C \to \Lambda$ is a morphism.
		The set $\mathfrak F(\Lambda)$ captures all possible $\forall \exists$-formulas with coefficients in $\Lambda$.
		\item $\mathfrak T(\Lambda)$ is the collection of all tuples $(H,C, \mathcal G, \iota, \sigma)$ where $(H, C, \mathcal G)$ belongs to $\mathfrak R$ while $\iota \colon H \to \Lambda$ and $\sigma \colon C \to \Lambda$ are morphisms.
		In other words, $\mathfrak T(\Lambda)$ encodes all possible $\forall \exists$-formulas with coefficients in $\Lambda$, together with an assignment in $\Lambda$ of the $\forall$-variables.
		It will become handy to track formulas that could fail in a quotient of $\Lambda$.
	\end{itemize}
  	We also fix a generating set $S$ for $\Gamma$.
	
	\paragraph{Induction.}
	We are going to build the sequence $(\Gamma_k)$ by induction on $k \in \N$.
	We will later prove that it satisfies all the required properties.
	The sets $\mathfrak M(\Gamma)$, $\mathfrak M(\Gamma \ast \Z)$, $\mathfrak F(\Gamma)$, $\mathfrak F(\Gamma\ast \Z)$, $\mathfrak T(\Gamma)$, and $\mathfrak T(\Gamma\ast \Z)$ are countable, hence we can endow them with an order that is order-isomorphic to $\N$.
	We also enumerate all the pairs of elements in $\Gamma$:
	\begin{equation*}
		(\gamma_0, \gamma'_0), (\gamma_1, \gamma'_1), (\gamma_2, \gamma'_2), \dots
	\end{equation*}
	starting with $(\gamma_0, \gamma'_0) = (1, 1)$.

	We now build by induction a sequence of non-elementary, torsion-free, hyperbolic quotients $\pi_k \colon \Gamma \onto \Gamma_k$ satisfying, among other things, the following property: for every $k \in \N$, the images of $\gamma_k$ and $\gamma'_k$ in $\Gamma_k$ generate either a cyclic subgroup or the entire quotient $\Gamma_k$.
	Along the way, we will also produce the following data.
		\begin{itemize}
			\item An increasing exhaustion $(\mathfrak M_k(\Gamma))$ of $\mathfrak M(\Gamma)$ by finite sets. 
			These will capture morphisms to $\Gamma_k$ on which we have some control.
			\item An increasing sequence  $(\mathfrak F^-_k(\Gamma))$ of finite subsets of $\mathfrak F(\Gamma)$. 
			Any tuple $(H,C, \mathcal G, \sigma)$ in  $\mathfrak F(\Gamma)$ that belongs to $\mathfrak F_k^-(\Gamma)$ will not be positively realized in $(\Gamma_k, \pi_k \circ \sigma)$.
			\item Two non-decreasing sequences $(\mathfrak T_k^-(\Gamma))$ and $(\mathfrak T_k^+(\Gamma))$ of finite subsets of $\mathfrak T(\Gamma)$. 
			A tuple $(H,C, \mathcal G, \iota, \sigma)$ in $\mathfrak T_k^-(\Gamma) \cup \mathfrak T_k^+(\Gamma)$ will belong to $\mathfrak T_k^-(\Gamma)$ or $\mathfrak T_k^+(\Gamma)$ depending whether $(H,C, \mathcal G)$ is negatively or positively realized in $(\Gamma_k, \pi_k \circ \iota, \pi_k \circ \sigma)$.

			\item A collection $\Phi_k(\Gamma) = (\varphi_t)$ indexed by $t \in \mathfrak T_k^+(\Gamma)$ such if $t = (H,C, \mathcal G, \iota, \sigma)$, then $\varphi_t$ is a witness for $(H, C, \mathcal G)$ in $(\Gamma_k, \pi_k \circ \iota, \pi_k \circ \sigma)$.
		\end{itemize}
		Similarly we will define at each step finite sets  $\mathfrak M_k(\Gamma \ast \Z)$, $\mathfrak F_k^-(\Gamma \ast \Z)$, and $\mathfrak T_k^\pm(\Gamma\ast \Z)$ which will be to $\Gamma_k \ast \Z$ what the above sets are to $\Gamma_k$.
		We also keep track of witnesses with a collection $\Phi_k(\Gamma \ast \Z)$ indexed by $t \in \mathfrak T_k^+(\Gamma\ast \Z)$.

	\subparagraph{Basis of the induction.}
	We begin the induction by setting $\Gamma_0 = \Gamma$.
	In addition, we let
	\begin{equation*}
		\mathfrak M_0(\Gamma) = \mathfrak F^-_0(\Gamma)  = \mathfrak T^\pm_0(\Gamma) = \Phi_0(\Gamma) = \emptyset
	\end{equation*}
	and
	\begin{equation*}
		\mathfrak M_0(\Gamma \ast \Z) = \mathfrak F^-_0(\Gamma \ast \Z)= \mathfrak T^\pm_0(\Gamma \ast \Z)  = \Phi_0(\Gamma \ast \Z) = \emptyset.
	\end{equation*}
	Note that the pair $(\gamma_0, \gamma'_0) = (1, 1)$ generates a cyclic subgroup of $\Gamma_0$.
	
	\subparagraph{Induction step.}
	Consider now $k \in \N$,  for which the quotient $\pi_k \colon \Gamma \onto \Gamma_k$ as well as
	\begin{equation*}
		\mathfrak M_k(\Gamma), \  \mathfrak F^-_k(\Gamma), \ \mathfrak T^\pm_k(\Gamma),  \ \Phi_k(\Gamma),
	\end{equation*}
	and
	\begin{equation*}
		\mathfrak M_k(\Gamma \ast \Z),  \ \mathfrak F^-_k(\Gamma \ast \Z), \  \mathfrak T^\pm_k(\Gamma \ast \Z),  \ \Phi_k(\Gamma \ast \Z),
	\end{equation*} 
	have already been defined. 
	We write $S_k$ for the image in $\Gamma_k$ of the generating set $S$ of $\Gamma$.
	The Cayley graph of $\Gamma_k$ with respect to the generating set $S_k$ is $\delta_k$-hyperbolic, for some $\delta_k \in \R^*_+$.
	
	\medskip
	First, we build a finite subset $U_k \subset (\Gamma_k \ast \Z) \setminus \{1\}$ that will allow us to preserve witnesses at the next step.
	Let $t = (H,C,\mathcal G, \iota, \sigma)$ in $\mathfrak T^+_k(\Gamma)$.
	Let $\varphi_t$ be the witness of $(H,C,\mathcal G)$ in $(\Gamma_k, \pi_k \circ \iota, \pi_k \circ \sigma)$ coming from $\Phi_k(\Gamma)$. 
	By definition of witnesses, there is $(G, V, \jmath) \in \mathcal G$ such that $\varphi_t$ belongs to $\hom[\jmath, \pi_k \circ (\iota \ast \sigma)] G{\Gamma_k}$ and satisfies $V \cap \ker \varphi_t = \emptyset$.
	We make sure that $U_k$ contains every element of $\varphi_t(V)$. 
	We proceed in the same way with $\mathfrak T^+_k(\Gamma \ast \Z)$, increasing if necessary the set $U_k$, to be able to track later all witnesses coming from $\Phi_k(\Gamma \ast \Z)$.
	
	Second, we extend $\mathfrak M_k(\Gamma)$ as follows.
	For every $(H, C, \mathcal G, \iota, \sigma) \in \mathfrak T_k(\Gamma)$, for every $(G, V, \jmath) \in \mathcal G$, we add $(H \ast C,G, \jmath, \iota \ast \sigma)$ to $\mathfrak M_k(\Gamma)$.
	We proceed in the same way with $\mathfrak M_k(\Gamma \ast \Z)$.
	
	\medskip
	We now distinguish two cases.
	If the images in $\Gamma_k$ of $\gamma_{k+1}$ and $\gamma'_{k+1}$ generate a cyclic subgroup, we simply let $\Gamma_{k+1} = \Gamma_k$.
	Otherwise, $\langle \gamma_{k+1}, \gamma'_{k+1} \rangle$ is non-elementary, so it follows from \autoref{res: induction step toward tarski monsters} (and \autoref{rem: induction step toward tarski monsters}) that there is a non-elementary, torsion-free, hyperbolic quotient $\eta_k \colon \Gamma_k \onto \Gamma_{k+1}$ with the following properties.
	Let $\pi_{k+1} = \eta_k \circ \pi_k$.
	\begin{enumerate}[label=({\arabic*)}]
		\item \label{enu: tarski proof - gen} 
		The group $\Gamma_{k+1}$ is generated by the images of $\gamma_{k+1}$ and $\gamma'_{k+1}$.
		\item \label{enu: tarski proof - kernel}
		No element of $U_k$ belongs to the kernel of the projection $\Gamma_k \ast \Z \onto \Gamma_{k+1} \ast \Z$ induced by $\eta_k$.
        		\item \label{enu: injectivity radius} 
		The ratio between $\delta_k$ and $r_{S_k}(\eta_k)$ is at most $2^{-k}$.
        		\item \label{enu: tarski proof - morphism} 
		For every $(G, H, \jmath, \iota) \in \mathfrak M_k(\Gamma)$ the natural map 
		\begin{equation*}
			\hom[\jmath, \pi_k \circ \iota]G{\Gamma_k} \to \hom[\jmath, \pi_{k+1} \circ \iota]G{\Gamma_{k+1}} 
		\end{equation*}		
		induced by $\eta_k$ is onto.
		\item \label{enu: tarski proof - morphism bis} 
		For every $(G, H, \jmath, \iota) \in \mathfrak M_k(\Gamma \ast \Z)$ the natural map 
		\begin{equation*}
			\hom[\jmath, \pi_k \circ \iota]G{\Gamma_k \ast \Z} \to \hom[\jmath, \pi_{k+1} \circ \iota]G{\Gamma_{k+1}\ast \Z} 
		\end{equation*}		
		induced by $\eta_k$ is onto.
	\end{enumerate}
	Now that the group $\Gamma_{k+1}$ is built, we define the remaining data for the index $k+1$.
	We first explain the construction for the group $\Gamma$, and then perform the same operations for $\Gamma \ast \Z$.
	
	\begin{itemize}
		\item 
		First we define $\mathfrak M_{k+1}(\Gamma)$ by adding to $\mathfrak M_k(\Gamma)$ the first element of $\mathfrak M(\Gamma)$ that is not already in $\mathfrak M_k(\Gamma)$.
		\item 
		Let us focus on $\mathfrak F^-_{k+1}(\Gamma)$.
		Along the way, we will define an auxiliary set $\mathfrak T^-_{k+1/2}(\Gamma)$ that will be used after to define $\mathfrak T^-_{k+1}(\Gamma)$ and will help us to keep track of sentences that do not hold true in $\Gamma_{k+1}$.
		
		Consider the first tuple $t = (H, C, \mathcal G, \sigma)$ in $\mathfrak F(\Gamma) \setminus \mathfrak F^-_k(\Gamma)$ such that $(H,C,\mathcal G)$ is not positively realized in $(\Gamma_{k+1}, \pi_{k+1} \circ \sigma)$.
		Since the projection $\pi_{k+1} \colon \Gamma \onto \Gamma_{k+1}$ is onto, there is a morphism $\iota \colon H \to \Gamma$ such that $(H, C, \mathcal G)$ is negatively realized in $(\Gamma_{k+1}, \pi_{k+1} \circ \iota, \pi_{k+1} \circ \sigma)$.
		We let 
		\begin{equation*}
			\mathfrak F^-_{k+1}(\Gamma) = \mathfrak F^-_k(\Gamma)  \cup \{ t\}.
			\quad \text{and}\quad 
			\mathfrak T^-_{k + 1/2}(\Gamma) = \mathfrak T^-_k(\Gamma) \cup \{(H,C, \mathcal G, \iota, \sigma)\}.
		\end{equation*}
		\item
		We now move to $\mathfrak T^+_{k+1}(\Gamma)$ and $\mathfrak T^-_{k+1}(\Gamma)$.
		Let $t = (H, C, \mathcal G, \iota, \sigma)$ be a tuple of $\mathfrak T^+_k(\Gamma)$.
		We denote by $\varphi'_t$ the morphism obtained by post-composing the corresponding witness $\varphi_t$  by $\eta_k$.
		It follows from our choice of $U_k$ that $\varphi'_t$ is still a witness of $(H,C,\mathcal G)$ in $(\Gamma_{k+1}, \pi_{k+1} \circ \iota,  \pi_{k+1} \circ \sigma)$.		
		We denote by $\Phi'_k(\Gamma)$ the collection indexed by $\mathfrak T^+_k(\Gamma)$ of all new witnesses $\varphi'_t$ obtained in this way.

		Consider now the first tuple $t = (H, C, \mathcal G, \iota, \sigma)$ in $\mathfrak T(\Gamma)$ that is not already contained in 
		\begin{equation*}
			\mathfrak T^-_{k+1/2}(\Gamma) \cup \mathfrak T^+_k(\Gamma).
		\end{equation*}
		If $(H, C, \mathcal G)$ is negatively realized in $(\Gamma_{k+1}, \pi_{k+1}\circ \iota, \pi_{k+1} \circ \sigma)$ we simply let 
		\begin{equation*}
			\mathfrak T^-_{k+1}(\Gamma) = \mathfrak T^-_{k+1/2}(\Gamma) \cup \{t\}, \quad
			\mathfrak T^+_{k+1}(\Gamma) = \mathfrak T^+_k(\Gamma),
			\quad \text{and} \quad
			\Phi_{k+1}(\Gamma) = \Phi'_k(\Gamma).
		\end{equation*}
		Otherwise, $(H, C, \mathcal G)$ is positively realized in $(\Gamma_{k+1}, \pi_{k+1} \circ \iota, \pi_{k+1} \circ \sigma)$.
		In particular, there exists a witness $\varphi_t$ of  $(H, C, \mathcal G)$ in $(\Gamma_{k+1}, \pi_{k+1} \circ \iota, \pi_{k+1} \circ \sigma)$.
		In this case, we let 
		\begin{equation*}
			\mathfrak T^-_{k+1}(\Gamma) = \mathfrak T^-_{k+1/2}(\Gamma),
			\quad \text{and} \quad
			\mathfrak T^+_{k+1}(\Gamma) = \mathfrak T^+_k(\Gamma) \cup \{t\}.
		\end{equation*}
		Moreover, we define $\Phi_{k+1}(\Gamma)$ by adding the new witness $\varphi_t$ to $\Phi'_k(\Gamma)$.
	\end{itemize}
	We now perform the exact same operations with $\Gamma \ast \Z$ in order to define
	\begin{equation*}
		\mathfrak M_{k+1}(\Gamma \ast \Z), \quad
		\mathfrak F^-_{k+1}(\Gamma \ast \Z), \quad
		\mathfrak T^\pm_{k+1}(\Gamma \ast \Z), 
		\quad\text{and}\quad
		\Phi_{k+1}(\Gamma \ast \Z).
	\end{equation*}
	This completes the induction process.
	
	\begin{rema}
		Note that if $(H,C,\mathcal G) \in \mathfrak R$ is positively realized in some $(\Gamma_k, \pi_k\circ \sigma)$, we do not ``throw away'' the tuple $(H,C,\mathcal G,\sigma)$ from $\mathfrak F(\Gamma)$ but rather put in ``on hold'' in case $(H,C,\mathcal G)$ would not be positively realized in $(\Gamma_{k'}, \pi_{k'}\circ \sigma)$ for some larger value of $k'$.
		Consequently, if for infinitely many $k \in \N$, the tuple $(H,C,\mathcal G,\sigma)$ is not positively realized $(\Gamma_k, \pi_k\circ \sigma)$, then $(H,C,\mathcal G,\sigma)$ will end up in $\mathfrak F^-_k(\Gamma)$ for all but finitely many $k \in \N$.
	\end{rema}
	
	We now claim that the direct limit $\Gamma_\infty$ of the sequence $(\Gamma_k)$ satisfies the desired properties.

	 \paragraph{Lacunary hyperbolicity.} By construction, the ratio $\delta_k / r_{S_k}(\eta_k)$ converges to zero.
	  So $\Gamma_\infty$ is lacunary hyperbolic, proving the first part of \ref{enu: tarski - tarski}.

	\paragraph{Morphisms.}
	Let $\jmath \colon H \to G$ be a homomorphism between finitely presented groups.
	Let $\iota \colon H \to \Gamma$ be another morphism.
	According to our construction, there is $k_0 \in \N$ such that for every integer $k \geq k_0$, the tuple $(H, G, \jmath, \iota)$ belongs to $\mathfrak M_k(\Gamma)$.
	A proof by induction using \ref{enu: tarski proof - morphism} shows that for every $k \geq k_0$, the natural map 
	\begin{equation*}
		\hom[\jmath, \pi_{k_0} \circ \iota]G{\Gamma_{k_0}} \to \hom[\jmath, \pi_k \circ \iota]G{\Gamma_k} 
	\end{equation*}	
	is onto.
	Consider now a morphism $\varphi \in  \hom[\jmath, \pi_\infty \circ \iota]G{\Gamma_\infty}$.
	Since $G$ is finitely presented and $\Gamma_\infty$ is the direct limit of the sequence $(\Gamma_k)$, the morphism $\varphi$ is the image of some element in $\hom[\jmath, \pi_k \circ \iota]G{\Gamma_k}$.
	The above discussion tells us now that $\varphi$ lifts to a morphism in $\hom[\jmath, \pi_{k_0} \circ \iota]G{\Gamma_{k_0}}$.
	This completes the proof of \ref{enu: tarski - morphism}.
	
	\paragraph{$\forall\exists$-theory.}
	We start with the following observation.
	\begin{clai}
	\label{cla: tarski}
		Let $(H,C, \mathcal G, \sigma)$ in $\mathfrak F(\Gamma)$.
		The following are equivalent
		\begin{enumerate}[label=({\alph*)}]
			\item \label{enu: tarski3 - limit}
			$(H,C, \mathcal G)$ is positively realized in $(\Gamma_\infty, \pi_\infty \circ \sigma)$.
			\item \label{enu: tarski3 - limit inf}
            		$(H,C, \mathcal G)$ is positively realized in $(\Gamma_k, \pi_k \circ \sigma)$, for all but finitely many $k \in \N$.
			\item \label{enu: tarski3 - limit sup}
            		$(H,C, \mathcal G)$ is positively realized in $(\Gamma_k, \pi_k \circ \sigma)$, for infinitely many $k \in \N$.
		\end{enumerate}
	\end{clai}
		
	\begin{proof}[Proof of the claim]
		The implication \ref{enu: tarski3 - limit inf} $\Rightarrow$ \ref{enu: tarski3 - limit sup} is straightforward.
		Let us show \ref{enu: tarski3 - limit sup} $\Rightarrow$ \ref{enu: tarski3 - limit}.
		Consider a morphism $H \to \Gamma_\infty$.
		Since $\pi_\infty$ is onto and $H$ is free, we can write it as $\pi_\infty \circ \iota$ where $\iota \colon H \to \Gamma$ is a morphism.
		According to \ref{enu: tarski3 - limit sup}, the tuple $(H, C, \mathcal G)$ is positively realized in $(\Gamma_k, \pi_k \circ \iota, \pi_k \circ \sigma)$ for infinitely many $k \in \N$.
		It follows from our construction that there is $k_0 \in \N$ such that $t = (H,C, \mathcal G, \iota, \sigma)$ belongs to $\mathfrak T^+_{k_0}(\Gamma)$.
		Moreover, the corresponding morphism $\varphi_t$ in $\Phi_{k_0}(\Gamma)$ is a witness for $(H,C, \mathcal G)$ in $(\Gamma_{k_0}, \pi_{k_0} \circ \iota, \pi_{k_0} \circ \sigma)$.
		At each step, the finite subset $U_k \subset \Gamma_k \ast \Z$ was precisely designed in such a way that for every $k \geq k_0$, the morphism $\varphi_t$ post-composed by the natural projection $\Gamma_{k_0} \onto \Gamma_k$ is still a witness  for $(H, C, \mathcal G)$ in $(\Gamma_k, \pi_k \circ \iota, \pi_k \circ \sigma)$. 
		Recall that $\Gamma_\infty$ is the direct limit of the sequence $(\Gamma_k)$.
		Hence, an element $\gamma \in \Gamma$ is trivial in $\Gamma_\infty$, if and only if $\gamma$ is trivial in $\Gamma_k$ for sufficiently large $k \in \N$.
		Consequently, $\zeta_{k_0} \circ \varphi_t$ is a witness of $(H, C, \mathcal G)$ in $(\Gamma_\infty, \pi_\infty \circ \iota, \pi_\infty \circ \sigma)$. 
		Therefore, $(H, C, \mathcal G)$ is positively realized in $(\Gamma_\infty, \pi_\infty \circ \iota, \pi_\infty \circ \sigma)$.
		The argument works for every morphism $H \to \Gamma_\infty$, hence proving \ref{enu: tarski3 - limit}. 
		
		We complete the proof with \ref{enu: tarski3 - limit} $\Rightarrow$  \ref{enu: tarski3 - limit inf}.
		We proceed by contraposition.
		Suppose that $(H,C, \mathcal G)$ is not positively realized in $(\Gamma_k, \pi_k \circ \iota, \pi_k \circ \sigma)$ for infinitely many $k \in \N$.
		It follows from the way we went through the elements of $\mathfrak F(\Gamma)$, that there is $k_0 \in \N$ such that $(H,C, \mathcal G, \sigma)$ belongs to $\mathfrak F_{k_0}^-(\Gamma)$.
		Without loss of generality, we can assume that $k_0$ is the smallest integer with this property.
		Remember that as we added $(H,C, \mathcal G, \sigma)$ to $\mathfrak F^-_{k_0}(\Gamma)$ we also chose a suitable morphism $\iota \colon H \to \Gamma$ and added $(H,C, \mathcal G, \iota, \sigma)$ to $\mathfrak T^-_{k_0}(\Gamma)$.
		Since the sequence $(\mathfrak T^-_k(\Gamma))$ is non-decreasing, $(H,C, \mathcal G, \iota, \sigma)$ belongs to $\mathfrak T^-_k(\Gamma)$, for every integer $k \geq k_0$.
		Consider now $(G,V, \jmath)$ in $\mathcal G$ and a morphism $\varphi \in \hom[\jmath, \pi_\infty\circ (\iota \ast \sigma)] G{\Gamma_\infty}$.
		According to our previous study of morphisms, there is $k \geq k_0$ such that $\varphi$ lifts to a morphism $\tilde \varphi \in \hom[\jmath, \pi_k\circ (\iota \ast \sigma)] G{\Gamma_k}$.
		However, since $(H,C,\mathcal G)$ is negatively realized in $(\Gamma_k, \pi_k \circ \iota, \pi_k \circ \sigma)$, the intersection $V \cap \ker \tilde \varphi$, and thus $V \cap \ker \varphi$, is non-empty.
		This exactly proves that $(H, C, \mathcal G)$ is negatively realized in $(\Gamma_\infty, \pi_\infty \circ \iota, \pi_\infty \circ \sigma)$, whence the result.
	\end{proof}

	Recall that the tuples in $\mathfrak F(\Gamma)$ are encoding $\forall\exists$-formulas with coefficients in $\Gamma$.
    Hence, the above statement translates to the following assertion
	\begin{equation*}
		{\Th}_{\forall\exists}(\Gamma_\infty, \pi_\infty(\vec \gamma))
		= \bigcap_{K \in \N} \bigcup_{k \geq K} {\Th}_{\forall\exists}(\Gamma_k, \pi_k(\vec \gamma))
		=\bigcup_{K \in \N} \bigcap_{k \geq K} {\Th}_{\forall\exists}(\Gamma_k, \pi_k(\vec \gamma)),
	\end{equation*}
	for all tuples $\vec \gamma$ of elements of $\Gamma$: this proves \ref{enu: tarski - theory}.
    As always in this construction, \autoref{cla: tarski} also holds when replacing $\Gamma$ by $\Gamma \ast \Z$, and hence we also get
    \[{\Th}_{\forall\exists}(\Gamma_\infty \ast \Z, \pi_\infty(\vec \gamma)) = \lim\limits_{k \to \infty}{\Th}_{\forall\exists}(\Gamma_k \ast \Z, \pi_k(\vec \gamma)),\]
    for all tuples $\vec \gamma$ of elements of $\Gamma$.

    \paragraph{Elementary embedding.}
    As observed in \cite[Lemma 7.1]{positiveAH}, it suffices to prove that $\Gamma_\infty$ $\forall\exists$-elementarily embeds in $\Gamma_\infty \ast \Z$.
    We know that each $\Gamma_k$ is elementarily embedded into $\Gamma_k \ast \Z$ \cite{positivehyp1} (see also \cite{positiveAH}). This, together with the previous paragraph, gives:
    \begin{align*}
        {\Th}_{\forall\exists}(\Gamma_\infty, \pi_\infty(\vec \gamma)) &= \lim\limits_{k\to\infty} {\Th}_{\forall\exists}(\Gamma_k, \pi_k(\vec \gamma)) \\
        &\subset \lim\limits_{k\to\infty}{\Th}_{\forall\exists}(\Gamma_k \ast \Z, \pi_k(\vec \gamma)) = {\Th}_{\forall\exists}(\Gamma_\infty \ast \Z, \pi_\infty(\vec \gamma)),
    \end{align*}
    for all tuples $\vec \gamma$ of elements of $\Gamma$. Since $\pi_\infty$ is surjective, this shows that $\Gamma_\infty$ is $\forall \exists$-elementarily embedded in $\Gamma_\infty \ast \Z$, and we obtain \ref{enu: tarski - embedding}.

    \paragraph{Simple torsion-free Tarski monster.}
	By construction, every $2$-generated subgroup $\Gamma_\infty$ is either cyclic or coincides with $\Gamma_\infty$.
	So every proper subgroup $\Lambda$ of $\Gamma_\infty$ is abelian.
	Suppose that $\Lambda$ is non-trivial.
	Let $\gamma \in \Gamma$ be the pre-image of a non-trivial element in $\Lambda$.
	Consider the group
	\begin{equation*}
		G = \group{x,y \mid [x,y]}.
	\end{equation*}
	Denote by $H$ the subgroup of $G$ generated by $y$ and by $\jmath \colon H \to G$ the corresponding embedding.
	Let $\iota \colon H \to \Gamma$ be the morphism sending $y$ to $\gamma$.
	According to \ref{enu: tarski - morphism}, there is $k \in \N$ such that the natural map 
	\begin{equation*}
		\hom[\jmath, \pi_k \circ \iota]G{\Gamma_k}  \to  \hom[\jmath, \pi_\infty \circ \iota]G{\Gamma_\infty}
	\end{equation*}
	is onto.
	This means that the projection $\Gamma_k \onto \Gamma_\infty$ maps the centralizer of $\pi_k(\gamma)$ in $\Gamma_k$ onto the centralizer of $\pi_\infty(\gamma)$ in $\Gamma_\infty$.
	However, since $\Gamma_k$ is torsion-free hyperbolic, the centralizer in $\Gamma_k$ of any non-trivial element is cyclic.
	Hence, $\Lambda$ is cyclic.

    The group $\Gamma_\infty$ is non-abelian and torsion-free because every $\Gamma_k$ is. 
    It remains to show that $\Gamma_\infty$ is simple. Suppose that $N < \Gamma_\infty$ is a proper, non-trivial, normal subgroup. 
    According to the previous discussion, $N$ must be cyclic. 
    Because $\Gamma_\infty$ has no index-$2$ subgroups, the action by conjugation of $\Gamma_\infty$ on $N$ is trivial, and therefore $N$ is central. We have just seen that centralizers in $\Gamma_\infty$ are cyclic, and so $\Gamma_\infty$ is itself cyclic. 
    This contradicts the fact that $\Gamma_\infty$ is non-abelian and concludes the proof of \ref{enu: tarski - tarski}.
\end{proof}

\subsection{Further properties}

We start with the following direct consequence of \autoref{res: tarski}.

\begin{coro}
\label{res: 2 quant theory free product}
    Let $\Gamma_\infty$ be a torsion-free Tarski monster as in \autoref{res: tarski}. Then $\Th_{\forall \exists}(\Gamma_\infty) = \Th_{\forall \exists}(\Gamma_\infty \ast F_n)$, where $F_n$ denotes the free group of rank $n \geq 0$.
\end{coro}

\begin{proof}
    By \autoref{res: tarski} the embedding $\Gamma_\infty \to \Gamma_\infty \ast \Z$ is $\exists \forall \exists$-elementary. In particular, if a $\exists \forall$-sentence holds in $\Gamma_\infty$, then it holds in $\Gamma_\infty \ast \Z$, and similarly if a $\forall \exists$-sentence holds in $\Gamma_\infty$ then it holds in $\Gamma_\infty \ast \Z$. This shows that $\Th_{\forall \exists}(\Gamma_\infty) = \Th_{\forall \exists}(\Gamma_\infty \ast \Z)$.

    Next, we apply \cite{positiveAH}: if $\Gamma$ is an acylindrically hyperbolic group with no non-trivial finite normal subgroups, then $\Th_{\forall \exists}(\Gamma) = \Th_{\forall \exists}(\Gamma \ast \Z)$. This shows that, for all $n \geq 1$, we have $\Th_{\forall \exists} (\Gamma_\infty \ast F_n) = \Th_{\forall \exists}(\Gamma_\infty \ast F_{n+1})$, and we conclude by induction.
\end{proof}

This answers \cite[Question 10.2]{positiveAH}: there exists a finitely generated group $\Gamma$ that is not acylindrically hyperbolic, with the property that $\Th_{\forall \exists}(\Gamma) = \Th_{\forall \exists}(\Gamma \ast \Z)$. Without the finite generation assumption, sharper examples were previously known, see \cite[Section 5]{andre:AH}. It remains an open problem whether acylindrical hyperbolicity is preserved under elementary equivalence of finitely generated groups: this appears in \cite[Question 10.3]{positiveAH} and \cite[Question 1.1]{andre:AH} where it is attributed to Osin.

\medskip

We now specialize to the \emph{positive theory}. 
Recall that a sentence is positive if it does not include negations. 
The positive theory of a group $\Gamma$ is the collection of positive sentences that are true in $\Gamma$, and is denoted by $\Th^+(\Gamma)$.
All non-abelian free groups have the same positive theory \cite{positivefree, makanin}, and the positive theory of a group is included in that of its quotients. 
In view of this, we say that $\Gamma$ has \emph{trivial positive theory} if $\Th^+(\Gamma)$ equals the positive theory of a non-abelian free group.

\begin{coro}
\label{res: tarski positive}
    Let $\Gamma_\infty$ be a torsion-free Tarski monster as in \autoref{res: tarski}. Then $\Th^+(\Gamma_\infty)$ is trivial.
\end{coro}

\begin{proof}
    By \cite[Theorem F]{positivetrees}, it suffices to show that $\Th^+_{\forall \exists}(\Gamma_\infty)$ is trivial. By \autoref{res: 2 quant theory free product}, we have $\Th^+_{\forall \exists}(\Gamma_\infty) = \Th^+_{\forall \exists}(\Gamma_\infty \ast F_2)$, and since positive sentences pass to quotients we see that $\Th^+_{\forall \exists}(\Gamma_\infty) \subset \Th^+_{\forall \exists}(F_2)$, and we conclude.
\end{proof}

Groups with trivial positive theory have an important property relating to word lengths. Let $w \in F_n$ be a word, $\Gamma$ a group, and let 
\begin{equation*}
	w(\Gamma) = \group{ w(\vec{\gamma})\  \colon \vec \gamma \in \Gamma^n }
\end{equation*}
denote the corresponding \emph{verbal subgroup}. 
We can define the \emph{$w$-length} of an element in $w(\Gamma)$ to be the word length in the generating set $\set{w(\vec{\gamma})}{\vec \gamma \in \Gamma^n}$; by convention we also say that the $w$-length of an element in $\Gamma \setminus w(\Gamma)$ is infinite.
If $\Gamma = w(\Gamma)$, and the $w$-length is bounded from above by $k \geq 1$, then $\Gamma$ satisfies a positive $\forall\exists$-sentence
\begin{equation}
\label{eqn: bound w length}
	\forall x, \ \exists \vec{y}_1, \dots, \vec{y}_k, \quad x = w(\vec{y}_1) \cdots w(\vec{y}_k).
\end{equation}
There are two degenerate cases. 
On the one hand, $w$ could represent the trivial element of $F_n$, in which case $w(\Gamma) = \{1\}$. 
On the other hand, $w$ could be of the form $w = c x_1^{i_1} \cdots x_n^{i_n}$, where $c \in [F_n, F_n]$ and $\mathrm{gcd}(i_1, \ldots, i_n) = 1$. 
In the latter case, $w(\Gamma) = \Gamma$ and $w$-length is bounded by $1$ for every group $\Gamma$. 
If either of these two cases occurs, we say that $w$ is \emph{silly}, after Segal, who showed that if $w$ is not silly, then for every $k \geq 1$, the positive sentence above is non-trivial \cite[Section 3.1]{segal}.
If $\Gamma$ is such that the $w$-length is unbounded on $w(\Gamma)$ for all non-silly words $w$, we say that $\Gamma$ is \emph{verbally parabolic}.

\begin{rema}
    The positive sentence (\ref{eqn: bound w length}) says that there is a bound of $k$ on the $w$-length on $w(\Gamma)$ \emph{and} that $w(\Gamma) = \Gamma$. One could also express a bound of $k$ on $w$-length on $\Gamma$, without requiring that $w(\Gamma) = \Gamma$, by a single positive sentence:
\begin{equation*}
	\forall \vec{x}_1, \ldots, \vec{x}_{k+1}, \ \exists \vec{y}_1, \ldots, \vec{y}_k, \quad w(\vec{x}_1) \cdots w(\vec{x}_{k+1}) = w(\vec{y}_1) \cdots w(\vec{y}_k).
\end{equation*}
However, if $\Gamma$ is simple and does not satisfy a law, then automatically $w(\Gamma) = \Gamma$ for every non-trivial word $w$. So in the context of the next result, these two properties are equivalent.
\end{rema}

\begin{coro}
\label{res: tarski verbally parabolic}
    Let $\Gamma_\infty$ be a torsion-free Tarski monster as in \autoref{res: tarski}. Then $\Gamma_\infty$ is verbally parabolic.
\end{coro}

\begin{proof}
    By \autoref{res: tarski positive}, for every non-silly word $w$, the $w$-length on $\Gamma_\infty$ is unbounded. 
    Moreover, $w(\Gamma_\infty) = \Gamma_\infty$.
    Indeed, $w(\Gamma_\infty)$ is a normal subgroup, and is thus non-trivial since $\Gamma_\infty$ does not satisfy a law.
    Thus, $\Gamma_\infty$ is verbally parabolic.
\end{proof}

These corollaries answer two questions from \cite{positivetrees}.

\cite[Question 9.11]{positivetrees} asks for the existence of a finitely generated non-amenable group without non-abelian free subgroups with trivial positive theory. 
If we run our construction starting with a group $\Gamma$ with property (T), \autoref{res: tarski positive} yields a torsion-free Tarski monster that is non-amenable and has trivial positive theory. Clearly, torsion-free Tarski monsters have no non-abelian free subgroups.

\medskip

The second part of \cite[Question 9.11]{positivetrees} suggests Golod--Shafarevich groups \cite{GSgroups} as a possible source of finitely generated non-amenable groups without free subgroups with trivial positive theory. However, most Golod--Shafarevich groups have non-trivial positive theory.

More generally, we make the following observation.
Let $\Gamma$ be a torsion group.
Given $\gamma \in \Gamma$ we denote its order by $o(\gamma)$.
We write $\mathcal O(\Gamma)$ for the set of all the orders of its elements, i.e.
\begin{equation*}
	\mathcal O(\Gamma) = \set{m \in \N}{ \exists \gamma \in \Gamma, \ m = o(\gamma)}.
\end{equation*}

\begin{lemm}
	Let $\Gamma$ be a torsion group.
	If $\Gamma$ has trivial positive theory, then for every integer $n \in \N \setminus\{0\}$, there are infinitely many numbers in $\mathcal O(\Gamma)$ which are divisible by $n$.
\end{lemm}

\begin{proof}
	Fix an integer $n \in \N \setminus\{0\}$ and assume that only finitely many numbers in $\mathcal O(\Gamma)$ are divisible by $n$ (in particular, $n \geq 2$).
	Denote them by $m_1, \dots, m_\ell$.
	We claim that $\Gamma$ satisfies the positive $\forall\exists$-sentence $\Sigma$ below:
	\begin{equation*}
		\forall x, \exists y, \quad
		\left( \bigvee_{i = 1}^\ell x^{m_i} = 1 \right) \vee \left( \bigvee_{\substack{d | n \\ d \neq n}} y^n = x^d \right).
	\end{equation*}
	Consider indeed an element $\gamma \in \Gamma$ and denote its order by $m$.
	Suppose that $m$ is distinct from each $m_i$.
	By definition, $n$ does not divide $m$, hence their greatest common divisor $d$ is distinct from $n$.
	Moreover, there are $p,q \in \Z$ such that $pn + qm = d$.
	Hence the element $\mu = \gamma^p$ satisfies $\mu^n = \gamma^d$.
	Therefore $\Sigma$ holds in $\Gamma$.
	We now claim that $\Sigma$ does not hold in the free group $F_2 = \langle a, b \rangle$. 
	Indeed, $a$ has infinite order and the equation $y^n = a^d$ has no solution when $0 <d < n$.
	Hence $\Sigma$ is non-trivial.
\end{proof}
 
\begin{rema}
\label{rema:pthroot}

This leaves open the possibility that there exists a torsion group $\Gamma$ with torsion of all possible orders, with trivial positive theory. However, the stronger statement that $\Th_{\forall \exists}(\Gamma) = \Th_{\forall \exists} (\Gamma \ast \Z)$ cannot hold for a torsion group. Indeed, a torsion group satisfies
    \[\forall x \exists y, \quad (y^2 = x) \vee (x^2 = y^2 \wedge x \neq y).\]
    That is, either $x$ generates a cyclic group of odd order, and then it admits a square root, or $x$ generates a cyclic group of even order, and then the square endomorphism of $\langle x \rangle$ is not injective.
    On the other hand, the displayed sentence is false for $\Gamma \ast \Z$, as witnessed by $x$ being the generator of $\Z$.
\end{rema}
    
\cite[Question 9.16]{positivetrees} asks for the existence of an infinite simple verbally parabolic group without non-abelian free subgroups. \autoref{res: tarski verbally parabolic} answers this in a strong way: the group is not only verbally parabolic but has trivial positive theory. It is an open question whether these two properties coincide \cite[Question 9.15]{positivetrees}.

Examples of simple groups with trivial positive theory are given in \cite{positivesimple}; however, they arise as amalgamated products of free groups and therefore contain non-abelian free subgroups. Our examples are also distinct in that they can be chosen to have property (T).

\medskip

Next, we consider mixed identities. Recall that $w \in \Gamma \ast F_n$ is a mixed identity for $\Gamma$ if $w(\vec{\gamma}) = 1$ for all $\vec{\gamma} \in \Gamma^n$. A group is \emph{mixed identity free (MIF)} if it has no mixed identities, apart from the trivial ones, that is, the identity elements of $\Gamma \ast F_n$ for all $n \geq 1$.

\begin{coro}[\autoref{intro:cor:MIF}]
\label{res: tarski MIF}
    Let $\Gamma_\infty$ be a torsion-free Tarski monster as in \autoref{res: tarski}. Then $\Gamma_\infty$ is MIF, but does not admit a $2$-transitive action.
\end{coro}

\begin{proof}
    By \cite[Remark 5.1]{hull:osin}, it suffices to show that $\Gamma_\infty$ has no non-trivial mixed identities with one variable. So let $w \in \Gamma_\infty \ast \Z$ be a mixed identity of $\Gamma_\infty$. Then $\Gamma_\infty$ satisfies the formula with constants $\forall x \,\, w(x) = 1$. By \autoref{res: tarski}, the same formula holds in $\Gamma_\infty \ast \Z$. Replacing $x$ with the generator of the free factor $\Z$, we see that $w$ must be trivial in $\Gamma_\infty \ast \Z$.

    Finally, torsion-free Tarski monsters do not admit $2$-transitive actions \cite[Example 6.9]{hull:osin}. More generally, no group can act $2$-transitively on a set with all stabilizers infinite cyclic \cite{mazurov}.
\end{proof}

\subsection{Stable boundedness of norms}
\label{subsec: boundedness}

We end by showing that \autoref{res: tarski} can be strengthened to include a control on conjugacy-invariant norms on $\Gamma_\infty$.
Because the proof of \autoref{res: tarski} is already quite long and complicated, we treat this argument separately for the reader's convenience.

\medskip

A \emph{norm} on a group $\Gamma$ is a map $\ell \colon \Gamma \to \R_+ \cup \{\infty\}$ such that $\ell(\gamma) = 0$ if and only if $\gamma = 1$; $\ell(\gamma^{-1}) = \ell(\gamma)$; and $\ell(\gamma_1\gamma_2) \leq \ell(\gamma_1) + \ell(\gamma_2)$.
A norm is \emph{trivial} if it only takes the values $\{0, \infty \}$. A norm is \emph{conjugacy-invariant} if $\ell(\gamma\gamma'\gamma^{-1}) = \ell(\gamma')$ for all $\gamma, \gamma' \in \Gamma$. 

Given a norm $\ell$, we can \emph{stabilize} it as follows:
\begin{equation*}
	s\ell(\gamma) = \liminf_{n \to \infty} \frac{\ell(\gamma^n)}{n}.
\end{equation*}
Note that if $\ell(\gamma) < \infty$, then by subadditivity, the sequence $\ell(\gamma^n) / n$ actually converges to its infimum.
Keep in mind that $s\ell$ need not be a norm (both the positivity and the triangle inequality could fail).
We say that $\ell$ is \emph{stably bounded} if $s\ell$ is identically $0$. That is, $\ell$ is stably bounded if for all $\gamma \in \Gamma$ there exists $k \geq 1$ such that $\ell(\gamma^k) < \infty$, and moreover, $\ell(\gamma^n)$ grows sublinearly in $n$ whenever $\ell(\gamma) < \infty$.

\medskip
In the construction below, we will only treat one particular example. 
Given an element $\alpha \in \Gamma$, we write $\ell_\alpha \colon \Gamma \to \N \cup \{ \infty \}$ for the word length on the alphabet $\{ \gamma \alpha^{\pm 1}\gamma^{-1} : \gamma \in \Gamma \}$. 
This is a conjugacy-invariant norm on $\Gamma$, and $s\ell_\alpha$ denotes the corresponding stabilization. 
Note that $\ell_\alpha$ takes values in $\N$ if and only if $\Gamma$ is normally generated by $\alpha$. 
The next lemma shows that the stable boundedness of such norms is representative of the stable boundedness of all conjugacy-invariant norms.

\begin{lemm}
\label{res: criterion stably bounded}
	Let $\alpha \in \Gamma$ and let $\ell$ be a conjugacy-invariant norm on $\Gamma$ such that $\ell(\alpha) < \infty$.
	If $\ell_\alpha$ is stably bounded, then so is $\ell$.  

    In particular, if $\Gamma$ is simple and there exists $\alpha \in \Gamma$ such that $\ell_\alpha$ is stably bounded, then every non-trivial, conjugacy-invariant norm on $\Gamma$ is stably bounded.
\end{lemm}

\begin{proof}
	Let $\gamma \in \Gamma$. 
	Since $s \ell_\alpha(\gamma) = 0$, for all $\sigma > 0$ there exist $p, q \geq 1$ with $p/q< \sigma$ and $\kappa_1, \ldots, \kappa_p \in \Gamma$ such that
	\begin{equation*}
		\gamma^q = \prod\limits_{i = 1}^p \kappa_i \alpha^{\pm 1} \kappa_i^{-1}.
	\end{equation*}
	Therefore
	\begin{equation*}
		\ell(\gamma^q) \leq \sum\limits_{i = 1}^p \ell(\kappa_i \alpha^{\pm 1} \kappa_i^{-1}) = p \cdot \ell(\alpha)
	\end{equation*}
	which implies
	\begin{equation*}
		s \ell(\gamma) \leq \frac pq\ell(\alpha) < \sigma \ell(\alpha).
	\end{equation*}
	This inequality holds for every $\sigma >0$, hence $s \ell(\gamma) = 0$.
	Since $\gamma$ was arbitrary, we conclude that $\ell$ is stably bounded.
	
	Now suppose that $\Gamma$ is simple, and let $\ell$ be a conjugacy-invariant norm on $\Gamma$. 
	Observe that the set of all elements $\gamma \in \Gamma$ such that $\ell(\gamma) < \infty$ is a normal subgroup of $\Gamma$, which is not trivial since $\ell$ is non-trivial.
	Thus, it coincides with $\Gamma$, hence contains $\alpha$.
	So the above discussion implies that $\ell$ is stably bounded.
\end{proof}

So, from now on, we focus on the norm $\ell_\alpha$. The following result is an analogue of \autoref{res: induction step toward tarski monsters}.

\begin{theo}
\label{res: induction step toward boundedness}
    Let $\Gamma$ be a non-elementary, torsion-free, hyperbolic group.
    For $i \in \intvald1n$, let $G_i$ be a finitely generated group, $H_i$ be any group, and $\jmath_i \colon H_i \to G_i$ and $\iota_i \colon H_i \to \Gamma \ast \Z$ be morphisms.
	Let $\Lambda$ be a non-elementary subgroup of $\Gamma$.
	Let $\alpha, \gamma \in \Gamma$ be arbitrary non-trivial elements. For every finite subset $U$ of $\Gamma$, and every $\sigma > 0$, there exists a non-elementary, torsion-free, hyperbolic quotient $\pi \colon \Gamma \onto \bar \Gamma$ with the following properties.
	\begin{enumerate}
		\item Denoting the projection $\pi \colon \Gamma \onto \bar \Gamma$, we have $s\ell_{\pi(\alpha)}(\pi(\gamma)) < \sigma$.
		\item The map $\Gamma \ast \Z \onto \bar \Gamma \ast \Z$ induced by $\pi$ is injective when restricted to $U$.
		\item For every $i \in \intvald1n$, the map
		\begin{equation*}
			\begin{array}{ccc}
				\hom[\jmath_i, \iota_i] {G_i}{\Gamma\ast \Z} & \to & \hom[\jmath_i, \pi \circ \iota_i]{G_i}{\bar \Gamma\ast \Z} \\
				\varphi & \mapsto & \pi \circ \varphi,
			\end{array}
		\end{equation*}
		induced by $\pi$, is onto. 
	\end{enumerate}
\end{theo}

\begin{proof}
    Let $\gamma_0 \in \Gamma$ be a primitive element such that $\gamma = \gamma_0^k$. 
    Then $s\ell_{\pi(\alpha)}(\pi(\gamma)) = k s\ell_{\pi(\alpha)}(\pi(\gamma_0))$, so up to replacing $\sigma$ by $\sigma/k$, we may assume from now on that $\gamma$ is primitive.
    Let $\Lambda$ denote the normal closure of $\alpha$ in $\Gamma$, which is a non-elementary subgroup. 
    By \autoref{res: finding malnormal free}, there exist $\gamma_1, x, y \in \Lambda$ such that $\Lambda_1 = \langle \gamma, \gamma_1, x, y \rangle$ is a quasi-convex, malnormal, free subgroup of rank four. Note that $\ell_\alpha(\gamma_1) < \infty$, because $\gamma_1$ belongs to the normal closure of $\alpha$.

    Consider now an element of $\Lambda_1$ of the form
    \begin{equation*}
    	w = \gamma^{-q} \prod\limits_{i = 1}^p \kappa_i \gamma_1 \kappa_i^{-1},
    \end{equation*}
    where $p, q \geq 1$ are chosen so that $p\ell_\alpha(\gamma_1)/q < \sigma$, and $\kappa_1, \ldots, \kappa_p$ are chosen in $\langle x, y \rangle$. 
    Then $w$ is not a proper power.
    Moreover, if $p$ is sufficiently large, and the $\kappa_i$ are sufficiently long (compared to $q$) and with sufficiently small overlaps, then $w$ can be made arbitrarily long and satisfies an arbitrary classical  $C'(\lambda_0)$ small cancellation condition (where $\lambda_0$ is of the order of $1/p$).
    Hence, the associated relation family $\mathcal Q$ of $\Gamma$ given by \autoref{res: carrying sc assumption} satisfies the tight $C'(\lambda, \epsilon)$ strengthened small cancellation condition, where $\lambda, \epsilon \in (0,1)$ can be chosen as small as desired.
    
    Consider now the quotient $\bar \Gamma = \Gamma / \normal w$ and denote by $\pi \colon \Gamma \onto \bar \Gamma$ the corresponding projection.
    The last two items follow from \autoref{res: recap sc}\ref{enu: recap sc - injectivity}, \autoref{res: lifting quotient}.
    By construction, $\pi(\gamma^q)$ is equal to a product of $p$ conjugates of $\pi(\gamma_1)$, and so to a product of $p \ell_\alpha(\gamma_1)$ conjugates of $\pi(\alpha)$ and its inverse. Therefore,
    \begin{equation*}
	    \ell_{\pi(\alpha)}(\pi(\gamma^q)) \leq p \ell_\alpha(\gamma_1) \quad \Rightarrow \quad s\ell_{\pi(\alpha)}(\pi(\gamma)) \leq \frac{p}{q} \ell_\alpha(\gamma_1) < \sigma,
    \end{equation*}
    which gives the first item.
\end{proof}

The analogue of \autoref{rem: induction step toward tarski monsters} also holds in this case.

\begin{theo}[\autoref{intro:thm:precise}]
\label{res: tarski bounded}
    The group $\Gamma_\infty$ from \autoref{res: tarski} can be moreover chosen so that every non-trivial conjugacy-invariant norm on $\Gamma_\infty$ is stably bounded.
\end{theo}

\begin{proof}
    We proceed as in the proof of \autoref{res: tarski}, and keep the same notation.
    Moreover, we fix an element $\alpha \in \Gamma \setminus\{1\}$ and enumerate $\Gamma \setminus \{1\} = \{ \gamma_1, \gamma_2, \ldots \}$. 
    We apply the same induction procedure, but moreover, ensure that
    \begin{equation*}
    	    s\ell_{\pi_k(\alpha)}\pi_k(\gamma_i) < \frac{1}{k}, \quad \forall i \in \{ 1, \ldots, k\}.
    \end{equation*}
    For this, it suffices to construct $\Gamma_{k+1}$ as the composition of a quotient obtained via \autoref{res: induction step toward tarski monsters}, with the same parameters as in the proof of \autoref{res: tarski}, followed by $k$ successive quotients obtained via \autoref{res: induction step toward boundedness}, applied with the image of the same $U$, and the pair of elements $\{\alpha, \gamma_i\}$, for all $i  \in \{1, \ldots, k\}$.

    The resulting torsion-free Tarski monster satisfies
    \begin{equation*}
	s\ell_{\pi_\infty(\alpha)}(\pi_\infty(\gamma_i)) \leq s\ell_{\pi_k(\alpha)}(\pi_k(\gamma_i)) \leq \frac{1}{k}    
    \end{equation*}
    whenever $k \geq i$, and therefore $\ell_{\pi_\infty(\alpha)}$ is stably bounded. 
    Since $\Gamma_\infty$ is moreover simple, we conclude from \autoref{res: criterion stably bounded} that every non-trivial, conjugacy-invariant norm on $\Gamma_\infty$ is stably bounded.
\end{proof}

In \cite{muranov}, Muranov constructed a simple group where the commutator length is stably bounded but unbounded --- he also shows that the square length is unbounded.
\autoref{res: tarski bounded} generalizes this to all words. 
Namely, if $w$ is not silly, then the $w$-length over $\Gamma_\infty$ is unbounded (and only takes finite values, see \autoref{res: tarski verbally parabolic}) but stably bounded.

\medskip

Still, the case of commutator length is especially important, so let us discuss it in more detail. We denote it by $\cl$, and its stabilization by $\scl$. Note that $\cl$ is a non-trivial norm on any non-abelian group. There is a rich theory connecting $\scl$ with various parts of geometry \cite{calegari}.

Most importantly, $\scl$ is connected to \emph{quasimorphisms}. Recall that a quasimorphism is a function $f \colon \Gamma \to \R$ with the property that
\begin{equation*}
	\sup\limits_{\gamma_1, \gamma_2 \in \Gamma} \abs{f(\gamma_1) + f(\gamma_2) - f(\gamma_1\gamma_2)} < \infty.
\end{equation*}
A quasimorphism is moreover \emph{homogeneous} if it restricts to a homomorphism on every cyclic subgroup. The connection is made by the \emph{Bavard duality}, which in a specialized form states that $\scl$ vanishes identically on $[\Gamma, \Gamma]$ if and only if every homogeneous quasimorphism on $\Gamma$ is a homomorphism \cite{bavard}. Quasimorphisms are important to the study of groups acting on hyperbolic spaces. In particular, if the only homogeneous quasimorphism on a group is the trivial homomorphism, then $\Gamma$ has no focal or oriented lineal actions on a hyperbolic space: see, for instance, \cite[Section 4.1]{manning}.

Quasimorphisms are connected to \emph{bounded cohomology}. This is a functional analytic analogue of ordinary cohomology, which is central in geometric topology and rigidity theory \cite{frigerio, monod:book}. There is a natural \emph{comparison map} from bounded to ordinary cohomology (with trivial real coefficients) $c^n \colon H^n_b(\Gamma) \to H^n(\Gamma)$. The kernel of $c^n$ is called the \emph{exact bounded cohomology}, denoted $EH^n_b(\Gamma)$. The connection with quasimorphisms happens in low degree: $EH^2_b(\Gamma)$ is isomorphic to the space of homogeneous quasimorphisms modulo homomorphisms \cite[Corollary 2.11]{frigerio}.

Putting all these remarks together, we obtain the following.

\begin{coro}
\label{res: bounded cohomology}
    Let $\Gamma_\infty$ be a group as in \autoref{res: tarski bounded}. Then the stable commutator length on $\Gamma_\infty$ vanishes identically, the only homogeneous quasimorphism is the trivial homomorphism, and the comparison map $c^2 \colon H^2_b(\Gamma_\infty) \to H^2(\Gamma_\infty)$ is injective. Moreover, $\Gamma_\infty$ cannot act on a hyperbolic space with a loxodromic element.
\end{coro}

\begin{proof}
    Because $\Gamma_\infty$ is non-abelian, $\cl$ is a non-trivial norm, and therefore by \autoref{res: tarski bounded} it is stably bounded, that is, $\scl$ vanishes identically. By the discussion above, every homogeneous quasimorphism is a homomorphism, and because $\Gamma_\infty$ is perfect, it must be the trivial homomorphism. Also, by the discussion above, $EH^2_b(\Gamma_\infty) = 0$, that is, the comparison map $c^2$ is injective. 
    
    Actions on a hyperbolic space with a loxodromic element are of one of the following types: non-oriented lineal, oriented lineal, focal, and general type \cite[Section 8.2]{Gromov:1987tk}. The first type is excluded because $\Gamma_\infty$ has no index-$2$ subgroup, the second and third because it has no homogeneous quasimorphisms besides the trivial homomorphism, and the fourth because it has no non-abelian free subgroup.
\end{proof}

\begin{rema}
    Let $\Gamma$ be a torsion-free hyperbolic group and let $\gamma \in \Gamma$. The construction above shows that there exist tight $C'(\lambda, \epsilon)$ strengthened small cancellation quotients $\pi \colon \Gamma \to \bar \Gamma$ such that $\scl(\pi(\gamma)) < \scl(\gamma)$, for arbitrarily small $\lambda, \epsilon \in (0, 1)$.
    
    On the other hand, if $\lambda$ and $\epsilon$ are small enough, then $\cl(\pi(\gamma)) = \cl(\gamma)$. Indeed, suppose that $\cl(\pi(\gamma)) \leq k$. 
    Let 
    \begin{equation*}
    	G = \group{x_1, y_1, \ldots, x_k, y_k}
    \end{equation*}
    be a free group, and let $H$ the subgroup of $G$ generated by $z = [x_1, y_1] \cdots [x_k, y_k]$. 
    Consider the morphism $\iota \colon H \to \Gamma$ sending $z$ to $\gamma$.
    The inequality $\cl(\pi(\gamma)) \leq k$ is witnessed by the fact that $ \hom[\id, \pi \circ \iota]{G}{\bar \Gamma}$ is non-empty.
    When $\lambda$ and $\epsilon$ are small enough, it follows from \autoref{res: lifting quotient} that $\hom[\id, \iota]{G}{\Gamma}$ is non-empty as well.
    Hence, $\cl(\gamma) \leq k$.
\end{rema}

\autoref{res: tarski bounded} also gives a partial answer to \cite[Question 9.4]{positivetrees}, which asks for a (finitely generated) group with trivial positive theory and finite-dimensional second bounded cohomology. Above, we constructed a finitely generated group with trivial positive theory and vanishing \emph{exact} second bounded cohomology. \cite[Question 9.4]{positivetrees} was motivated by the observation that for groups acting on trees, a natural weak acylindricity condition implies at once that the second bounded cohomology is infinite-dimensional \cite{IPS} and the positive theory is trivial \cite[Theorem 6.4]{positivetrees}. This persists in other classes of groups with features of negative curvature, for example, the positive theory of acylindrically hyperbolic groups is trivial \cite{positiveAH}, and their second bounded cohomology is infinite-dimensional \cite{bestvinafujiwara}. Let us stress that in all of these examples, the non-trivial classes constructed in second bounded cohomology are all exact, so our construction is just as surprising in this direction.

In principle, one could attempt to construct the groups $\Gamma_\infty$ ensuring moreover that $H^2(\Gamma_\infty)$ is finite dimensional, which would give a full answer to \cite[Question 9.4]{positivetrees}. However, the second cohomology of limits of small cancellation quotients constructed as in \autoref{res: tarski} tends to be infinite-dimensional \cite[Remark 4.5]{dimensions:simple}.

\medskip

Let us end by exhibiting another peculiar behavior of commutator length on the group $\Gamma_\infty$.

\begin{coro}[\autoref{intro:cor:ab}]
\label{res: torsion in ab of product}
Let $\Gamma_\infty$ be a group as in \autoref{res: tarski bounded}. Then there exist elements $\gamma_k \in \Gamma_\infty$ such that $\cl(\gamma_k) \to \infty$ but $\cl(\gamma_k^2)$ is bounded. Therefore, the direct power $\Gamma_\infty^{\N}$ has $2$-torsion in its abelianization.
\end{coro}

\begin{proof}
There exist elements $g_k$ in the free group of rank $2$ such that $\cl(g_k) \to \infty$ but $\cl(g_k^2)$ is bounded. As observed by Kharlampovich--Myasnikov \cite[Theorem 3]{KM:genus}, this follows from their implicit function theorem \cite{KM:implicit} (or Sela's version of Merzlyakov's Theorem \cite{sela:merzlyakov}). The proof is non-constructive (although the bound on $\cl(g_k^2)$ can be made explicit) and such elements have only been found explicitly for low values of $\cl$ \cite{torsion:ab:product}.

Hence, there exists $m \geq 1$ such that for every $k \geq 1$ the free group of rank $2$ satisfies the sentence $\Sigma_k$ below:
\[\exists g, \exists x_1, y_1, \ldots, x_m, y_m, \forall a_1, b_1, \ldots a_k, b_k : g \neq \prod\limits_{i = 1}^k [a_i, b_i] \bigwedge g^2 = \prod\limits_{j = 1}^m [x_j, y_j].\]
Because $\Gamma_\infty \ast F_2$ retracts onto $F_2$, it also satisfies $\Sigma_k$. By \autoref{res: 2 quant theory free product}, we see that $\Gamma_\infty$ also satisfies $\Sigma_k$, and the witnesses of $\exists g$ in $\Sigma_k$ give the desired sequence $(\gamma_k)$.
The element $(\gamma_k) \in \prod_k \Gamma_\infty$ is not in the commutator subgroup, but its square is.
\end{proof}

\begin{rema}
    The conclusion of \autoref{res: torsion in ab of product} holds for all acylindrically hyperbolic groups. Indeed, the only property of $\Gamma_\infty$ that we used above is that it has the same two-quantifier theory as a group that surjects onto $F_2$, a property that all acylindrically hyperbolic groups have \cite{positiveAH}.
\end{rema}

\makebiblio

\vspace{0.5cm}

\normalsize

\noindent{\textsc{Université Bourgogne Europe, CNRS, IMB UMR 5584, 21000 Dijon, France}} \\
\noindent{\textit{E-mail address:} \texttt{remi.coulon@cnrs.fr}} \\

\noindent{\textsc{Department of Pure Mathematics and Mathematical Statistics, University of Cambridge, UK}} \\
\noindent{\textit{E-mail address:} \texttt{ff373@cam.ac.uk}} \\

\noindent{\textsc{Natural Sciences Division, New College of Florida, Sarasota, Florida 34243, USA}} \\
\noindent{\textit{E-mail address:} \texttt{mho@ncf.edu}} \\

\end{document}